\documentclass[11pt]{amsart}
\allowdisplaybreaks[1]

\usepackage{a4wide}
\usepackage{amssymb}
\usepackage{amsmath} 
\usepackage{amscd} 
\usepackage{epsfig}
\usepackage{graphicx}
\usepackage{epstopdf}
\usepackage{textcomp}
\usepackage{color}
                
\theoremstyle{plain} 
\newtheorem{theorem}{Theorem} [section]
\newtheorem{lemma}[theorem]{Lemma}
\newtheorem{prop}[theorem]{Proposition} 
\newtheorem{coro}[theorem]{Corollary}

\theoremstyle{definition}
\newtheorem{definition}[theorem]{Definition}
\newtheorem{remark}[theorem]{Remark}

\newcommand{\ts}{\hspace{0.5pt}}

\newcommand{\RR}{\mathbb{R}\ts}

\newcommand{\ZZ}{\mathbb{Z}}

\newcommand{\PP}{\mathbb{P}}

\newcommand{\EE}{\mathbb E}

\newcommand{\CalP}{\mathcal{P}}
\newcommand{\CalQ}{\mathcal{Q}}

\newcommand{\CalT}{\mathcal{T}}

\title[Taylor--Socolar Hexagonal Tilings As Model Sets]
{Taylor--Socolar Hexagonal Tilings As Model Sets}

\author{Jeong-Yup Lee}
\address{Department of Mathematics Education, Kwandong University, 
\newline 
\hspace*{12pt} Gangneung, Gyeonggi-do, 210-701, Korea}
\email{jylee@kwandong.ac.kr}

\author{Robert V.\ Moody}
\address{Department of Mathematics and Statistics, 
University of Victoria, \newline
\hspace*{12pt}Victoria, British Columbia V8W 3P4, Canada}
\email{rmoody@uvic.ca}

\thanks{The first author is grateful for the support of Basic Research Program through the
National Research Foundation of Korea(NRF) funded by the Ministry
of Education, Science and Technology(2010-0011150)  as well as the support of KIAS}

\begin{document}

\maketitle

\begin{abstract} The Taylor--Socolar tilings \cite{ST,Taylor} are regular hexagonal tilings of the plane but are distinguished in being comprised of hexagons of two colors in an aperiodic way.  We place the Taylor--Socolar tilings into an algebraic setting
which allows one to see them directly as model sets and to understand the corresponding tiling hull
along with its generic and singular parts. 

Although the tilings were originally obtained by matching rules and by substitution, our approach sets the tilings into the framework of a
cut and project scheme and studies how the tilings relate to the corresponding internal space. 
The centers of the entire set of tiles of one tiling form a lattice
$Q$ in the plane. If $X_Q$ denotes the set of all Taylor--Socolar tilings with centers on $Q$
then $X_Q$ forms a natural hull under the standard local topology of hulls
and is a dynamical system for the action of $Q$.
The $Q$-adic completion $\overline Q$ of $Q$ is a natural factor
of $X_Q$ and the natural mapping $X_Q \longrightarrow \overline{Q}$ is bijective except at a 
dense set of points of measure $0$ in $\overline Q$.
We show that $X_Q$ consists of three LI classes under translation. Two of these LI classes are very small, namely countable $Q$-orbits in $X_Q$. The other is a minimal dynamical system which 
maps surjectively to $\overline Q$ and which is variously $2:1$, $6:1$, and $12:1$ at the singular points. 

We further develop the formula of \cite{ST} that determines the parity of the tiles
of a tiling in terms of the co-ordinates of its tile centers. Finally we show that the hull of the parity tilings can be identified
with the hull $X_Q$; more precisely the two hulls are mutually locally derivable.

 \end{abstract}

\section{Introduction}\label{intro}

This paper concerns the aperiodic hexagonal mono-tilings created by Joan Taylor. We learned
about these tilings from the unpublished (but available online) paper of Joan Taylor
\cite{Taylor}, the extended paper of Socolar and Taylor \cite{ST}, and a talk given
by Uwe Grimm at the KIAS conference on aperiodic order in September, 2010 \cite{BG}. 
These tilings are in essence regular hexagonal tilings of the plane, but there are two forms
of marking on the hexagonal tile (or if one prefers, the two sides of the tile are marked differently).
We refer to this difference as {\bf parity} (and eventually distinguish the two sides 
as being sides $0$ and $1$), and 
in terms of parity the tilings are aperiodic. In fact the parity patterns of tiles created in this way
are fascinating in their apparent complexity, see Fig.~\ref{parityPattern} and Fig.~\ref{ST-whole-tiling}. 

\begin{figure}
\centering
\includegraphics[width=13cm]{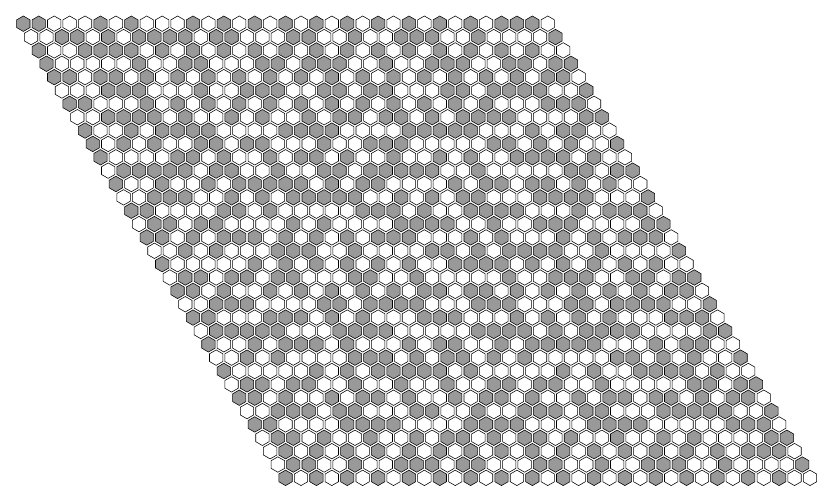}
\caption{A section of a Taylor--Socolar tiling showing the complex patterning arising
from the two sides of the hexagonal tile, here indicated in white and gray. Notice
that there are islands (Taylor and Socolar call the llamas) both of white and gray
tiles.}
\label{parityPattern}
\end{figure}

The two Taylor--Socolar tiles are shown in Fig.~\ref{two-basic-tiles}, the main features being the
black lines, one of which is a stripe across the tile, and the three colored diameters,
one of which is split in color\footnote{Note that the two tiles here are not mirror images of
each other, unless one switches color during the reflection. In \cite{ST} there is an alternative
description of the tiles in which the diagonals have flags at their ends, and in this
formulation the two tiles are mirror images of each other.}. 
The difference in the two tiles is only in which
side of the color-split diameter the stripe crosses. In the figure the tiles are colored
white and gray to distinguish them, but it is the crossing-color of the black stripe that
is the important distinguishing feature.

\begin{figure}
\centering
\includegraphics[width=6cm]{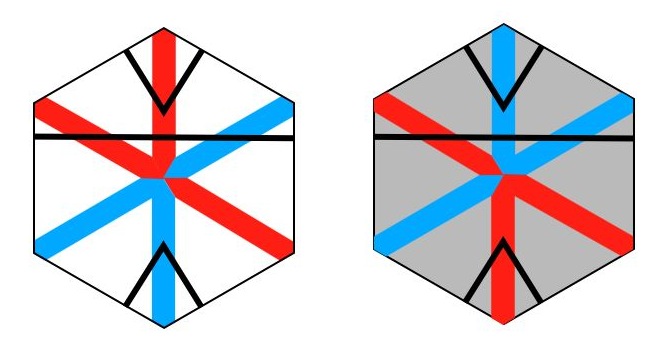}
\caption{The two basic hexagonal tiles. One is a white tile and the other a light
gray. These are colored with red and blue
diameters. 
The rotational position of the tiles is
immaterial. Note how the tiles are identical as far as the red diagonal and blue diagonal
are concerned. The distinction is in which color of the red-blue diagonal cuts the  black stripe.}
\label{two-basic-tiles}
\end{figure}

Taylor--Socolar tilings can be defined by following simple matching rules ({\bf R1, R2}) and can also be constructed by substitution (the scaling factor being $2$). In this paper it is the matching rules that are of
importance. 

\begin{itemize}
\item[{\bf R1}] the black lines must
join continuously when tiles abut;
\item[{\bf R2}] the ends of the diameters
of two hexagonal tiles that are separated by an edge of another tile must be of opposite colors,
Fig.~\ref{rule2}.
\end{itemize}

\begin{figure}
\centering
\includegraphics[width=2cm]{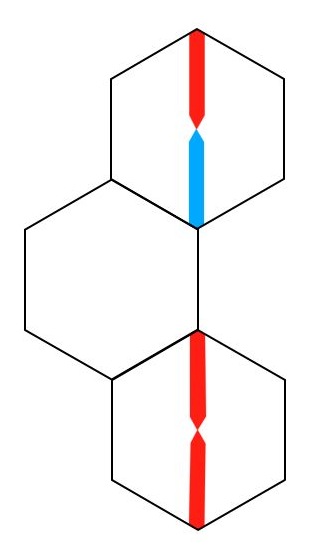}
\caption{Rule {\bf R2}: two hexagon tiles separated by the edge of another hexagon tile. Note that
the diameter colors of the two hexagons are opposite at the two ends of the separating edge.
It makes no difference whether or not the diameters are color-split -- the diameters must have different
colors where they abut the separating edge.}
\label{rule2}
\end{figure}
The paper \cite{ST} emphasizes the tilings from the point of view of matching rules, whereas \cite{Taylor} emphasizes
substitution (and the half-hex approach). There is a slight mis-match between the two approaches,
see \cite{BG}, which we will discuss later. 

If one looks at part of a tiling with the full markings of the tiles made visible, then one is immediately
struck by how the black line markings of the tiles assemble to form nested equilateral triangles,
see Fig.~\ref{ST-whole-tiling}. Although these triangles are slightly shrunken (which ultimately
is important), we see that basically the vertices of the triangles are tied to the centers
of the hexagons, and the triangle side-lengths are $1,2,4,8, \dots$
in suitable units. This triangle pattern is highly reminiscent of the square patterns that underlie
the famous Robinson tilings \cite{Robinson, GJS} which also appear in sizes which scale up by
factors of $2$. These tilings are limit-periodic tilings and can be described by model sets 
whose internal spaces are $2$-adic spaces. The Taylor-Socolar tilings are also limit-periodic and it seems
natural to associate some sort of $2$-adic spaces with them and to give a model-set
interpretation of the picture. 

One purpose of this paper is to do this, and it has the natural consequence that the tilings are
pure point diffractive. It is convenient to base the entire study on a fixed
standard hexagonal tiling of the coordinate plane $\RR^2$. The centers of the hexagonal tiles can then be
interpreted as a lattice in the plane (with one center at $(0,0)$). 
The internal space of the cut and project scheme that we shall construct is based on a $2$-adic completion $\overline{Q}$ of the group $Q$
consisting of all translation vectors between the centers of the hexagons. We shall show that there
is a precise one-to-one correspondence between triangulations and elements of $\overline Q$. But the 
triangulation is not the whole story.

The set of all Taylor--Socolar tilings associated with a fixed standard hexagonal tiling of the plane form a tiling hull $X_Q$. This hull is a dynamical system (with group $Q$) and carries the standard topology of tiling hulls. Each tiling has an associated
triangulation, but the mapping $\xi: X_Q \longrightarrow \overline Q$ so formed, while generically $1-1$,
is not globally $1-1$. What lies behind this is the question of backing up from the triangulations
to the actual tilings themselves. The question is how are the tile markings deduced from the
triangulations so as to satisfy the rules {\bf R1, R2}? There are two aspects to this. The triangulations
themselves are based on hexagon centers, whereas in an actual tiling the triangles are shrunken
away from vertices. This shrinking moves the triangle edges and is responsible for the 
off-centeredness of the black stripe on each hexagon tile. How is this shrinking 
(or edge shifting, as we call it) carried out? The second feature
is the coloring of the diagonals of the hexagons. What freedom for coloring exists, given that
the coloring rule {\bf R2} must hold?

In this paper we explain this and give a complete description of the hull and the mapping $\xi$,
Theorem~\ref{minimal}.
There are numerous places at which $\xi$ is singular (not bijective); in fact the set of singular points
in $\overline Q$ is dense. Two special classes of singular points are those corresponding to the
central hexagon triangulations
({\bf CHT})(see Fig.~\ref{CHT-tiling}) and the infinite concurrent $w$-line tilings ({\bf iCw-L}) (see Fig.~\ref{iCw-L-tiling}). In both cases there is
$3$-fold rotational symmetry of the {\em triangulation} and in both cases the mapping $\xi$ is many-to-one. These
two types of tilings play a significant role in \cite{ST}.

The hull has a minimal invariant component of full measure and
this is a single LI class. There are two additional orbits, whose origins are the {\bf iCw-L} triangulations,
and although they perfectly obey the matching rules they are not in the same LI class as all the other tilings.
On the other hand the {\bf CHT} tilings (those lying over the {\bf CHT} triangulations) are in the main LI class and,
because of the particular simplicity of the unique one whose center is $(0,0)$, the question of describing the parity (which tiles are facing up and which are facing down)
becomes particularly easy. Here we reproduce the parity formula for this {\bf CHT} tiling as given in \cite{ST}
(with some minor modifications in notation). We use this to give parity formulas for
all the tilings of $X_Q$.

A couple of comments about earlier work on aperiodic hexagonal tilings are appropriate here. 
D.~Frettl\"oh \cite{Dirk} discusses the half-hex tilings (created out of a simple substitution rule) and proves
that natural point sets associated with these can be expressed as model sets. Half-hexes don't play an explicit role in this paper, though the hull of the half-hex tilings is a natural factor of $X_Q$ lying between
$X_Q$ and $\overline Q$ \cite{Dirk, HF, GS}. They were important to Taylor's descriptions of her tilings and are implicitly embedded in them. 

In \cite{Penrose}, Roger Penrose gives a fine introduction to aperiodic tilings and then goes on to create a class of aperiodic hexagonal tilings, which he calls $1 + \epsilon + \epsilon^2$-tilings in which there are three types of tiles that assemble by matching rules. The main tiles are hexagonal, with keyed edges. The other two are a linear-like tile  with an arbitrarily small width ($\epsilon$ tiles) which fit along the hexagon edges, and some very tiny tiles ($\epsilon^2$) which fit at the corners of the tiles. 

Clearly his objective was to create a single tile that only tiles aperiodically, although that was not achieved
in \cite{Penrose}. Subsequently, however, Penrose did find a solution to the problem that uses
a single hexagonal tile with matching rules for the edges and corners, \cite{PenroseTwistor}. This has only recently become more widely known after Joan Taylor's work started to circulate. 

One can quibble about whether or not Taylor's tiling stretches the concept of matching rules since
the second rule relates non-adjacent tiles and also in her tiling there are two tiles, though (at least in the right markings) they are mirror images of each other. However, the tilings of Penrose and Taylor tilings are a fascinating pair. Extensive computational work of F.~G\"ahler indicates that the two tilings are quite distinct from one another, though they both have $\overline Q$ as a factor and apparently both have the same dynamical zeta functions \cite{Baake-Gaehler}. 

There is an algorithmic computation for determining that certain classes of substitution tilings have pure-point spectrum. It has been used to confirm that the Taylor--Socolar substitution tilings have pure point spectrum or, equivalently, are regular model sets \cite{AL2}.

\section{The triangulation} \label{construction}

In principle the tilings that we are interested in are not connected to the points of lattices 
and their cosets in $\RR^2$,
but are only point sets that arise in Euclidean space $\EE$ as the vertices and centers of tilings.
However, our objective here to realize tiling vertices in an algebraic
context and for that we need to fix an origin and a coordinate system so as to reduce the
language to that of $\RR^2$.
Let $Q$ be the triangular lattice in $\RR^2$ defined by
\[Q := \ZZ a_1 + \ZZ a_2\]
where $a_1 = (1,0)$ and $a_2 = (-\frac{1}{2}, \frac{\sqrt{3}}{2})$.
Then $P := \ZZ w_1 + \ZZ w_2$ where
$w_1= \frac{2}{3}a_1 + \frac{1}{3}a_2$ and $w_2= \frac{1}{3}a_1 + \frac{2}{3}a_2$
is a lattice containing $Q$ as sublattice of index $3$, see Fig.~\ref{hex}. For future
reference we note that $|a_1| = |a_2| =|a_1+a_2| =1$ and $|w_1| =|w_2| = 
|w_2-w_1|=1/\sqrt{3}$.

\begin{figure}
\centering
\includegraphics[width=9cm]{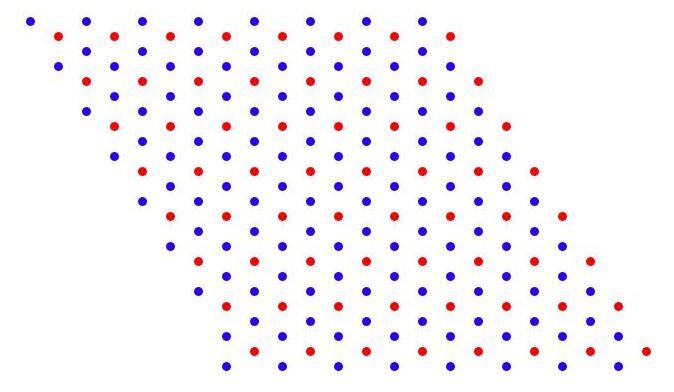}
\caption{The figure shows the standard triangular lattice $Q$ (the red points)  and the larger
lattice $P$ (red and blue points) in which $Q$ lies with index $3$.  The points of $Q$ may
be viewed as the vertices of a triangularization of the plane by equilateral triangles
of side length $1$. The red points are the centres of these triangles. The color here has nothing
to do with the coloring of the diagonals of the tiles -- it only distinguishes the two cosets.}
\label{hex}
\end{figure}

\smallskip
Joining the points of $Q$ that lie at distance $1$ from one another creates a triangular
tiling. Inside each of the unit triangles so formed there lies a point of $P$, and indeed
$P$ consists of three $Q$ cosets: $Q$ itself, the centroids of the ``up'' triangles
(those with a vertex above a horizontal edge), and the ``down" triangles
(those with a vertex below a horizontal edge), see Fig.~\ref{QP}. What we aim to do is to 
create a hexagonal tiling of $\RR^2$. When this tiling is complete, the points of
$Q$ will be the centers of the hexagonal tiles and the points of $P$ immediately 
surrounding the points of $Q$ will make up the vertices of the tiles.\footnote{
Nearest neighbours in $Q$ are distance $1$ apart and the short diameters of the hexagons
 are of length $1$ while the edges
of the hexagons are of length $r=1/\sqrt{3}$. The main diagonals
of the hexagons are of length $2r$ in the directions of $\pm w_1, \pm w_2, \pm(w_2 - w_1)$.
One notes that each of these vectors of $P$ is also of length $r$.}

Each of the hexagonal tiles 
will be marked by colored diagonals and a black stripe, see Fig.~\ref{two-basic-tiles}. 
These markings divide the tiles 
into two basic types, and it is describing the pattern made from these two types in model-set
theoretical terms that is a primary objective of this paper (see Fig.~\ref{ST-whole-tiling}). The other objective is to describe
the dynamical hull that encompasses all the tilings that belong to the Taylor--Socolar tiling family. 

 \begin{figure}
\centering
\includegraphics[width=6cm]{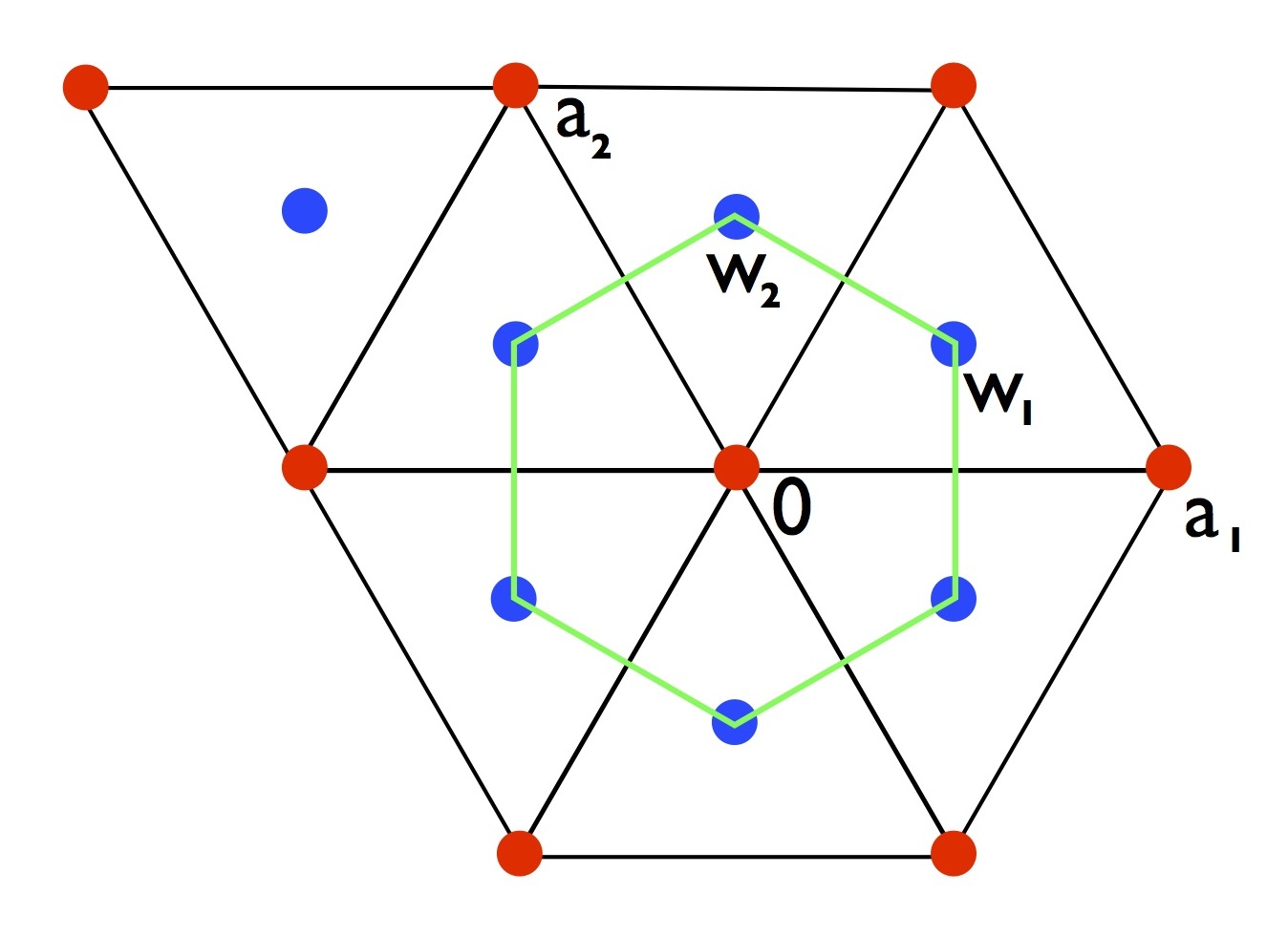}
\caption{The generators $a_1,a_2$ of $Q$ and the generators
$w_1,w_2$ of $P$, showing how the cosets of $Q$ in $P$ split
into the points of $Q$ and the centroids (centers) of the up and down triangles. Around
the point $0\in Q$ we see the hexagonal tile centered on $0$ with vertices in $P\backslash Q$.}
\label{QP}
\end{figure}

We let the coset of up (respectively down) points be denoted by 
$S^\uparrow_1 = w_1 + Q$ and $S^\downarrow_1 = w_2 +Q$ respectively:
\[ P = Q \cup S^\uparrow_1 \cup S^\downarrow_1 \,.\]

\begin{remark} There are three cosets of $Q$ in $P$. In our construction of the triangle patterns
we have taken the point of view that $Q$ itself will be used for triangle vertices and the other
two cosets for triangle centroids. However, we could use any of the three cosets as the 
triangle vertices and arrive at a similar situation. This amounts to a translation of the plane
by $w_1$ or $w_2$. We come back to this point in \S\ref{parityHull}.
\end{remark}

We now wish to re-triangularize the plane still using points of $Q$ as vertices, but this time 
making triangles of side length equal to $2$ using as vertices a coset of $2Q$ in $Q$.
There are four cosets of $2Q$ in $Q$ and they lead to four different ways to make
the triangularization. Fig.~\ref{cosets} shows the four types of triangles of side length $2$.
The lattices generated by the points of any one of these triangles is a coset of $2Q$
and together they make up all four cosets of $2Q$ in $Q$.

 \begin{figure}
\centering
\includegraphics[width=6cm]{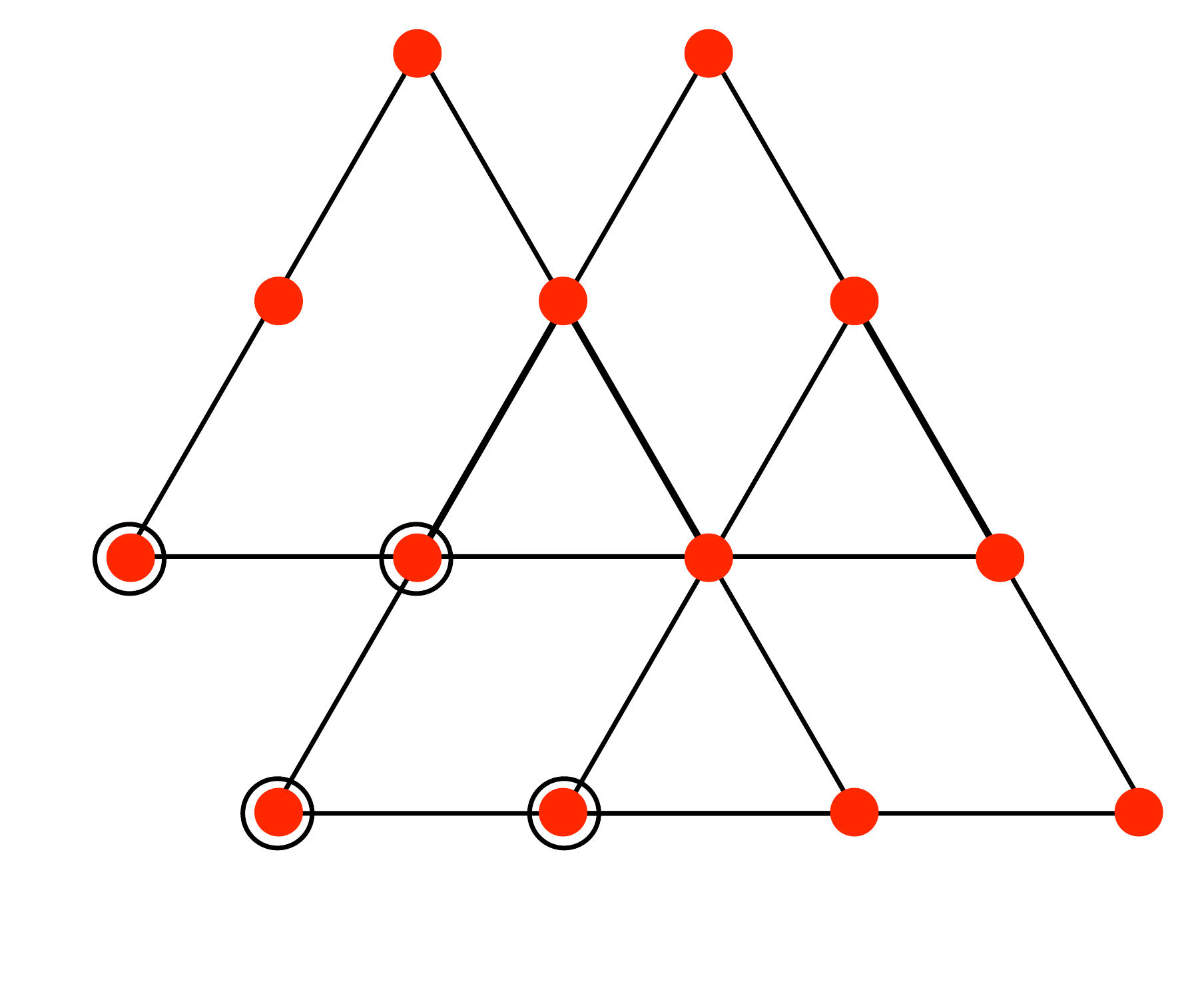}
\caption{Four superimposed triangles, each indicated by its circled
bottom lefthand vertex. The vertices of each triangle generate a different
coset of $Q$ modulo $2Q$.}
\label{cosets}
\end{figure}
Choose one of these cosets, call it $q_1+ 2Q$, where $q_1 \in Q$, and thereby triangulate the plane
with triangles of side length $2$. The centroids of the new triangles are a subset of
the original set of centroids and, in fact, together with the vertices $q_1 +2Q$
they form the coset $q_1 + 2P$. This is explained in the Fig.~\ref{newCentres}, which
also explains the important fact that the new centroids,
namely those of the new edge-length-$2$ triangles of $q_1+2Q$, make up two
cosets of $2Q$ in $q_1+ 2P$ depending on the orientation of the new 
triangles, and these orientations
are {\em opposite} to those that these points originally had. Thus we obtain 
$S^\uparrow_2 = q_1 +2w_1 + 2Q$ (which is in $w_2 +Q$\, !),
$S^\downarrow_2 = q_1 +2w_2 + 2Q$ (which is in $w_1 +Q$), and
the coset decomposition
\[q_1+ 2P = (q_1+2Q) \cup S^\uparrow_2 \cup S^\downarrow_2 \,\]
with $S^\uparrow_2 \subset S^\downarrow_1$ and 
$S^\downarrow_2 \subset S^\uparrow_1$.

\begin{figure}
\centering
\includegraphics[width=6cm]{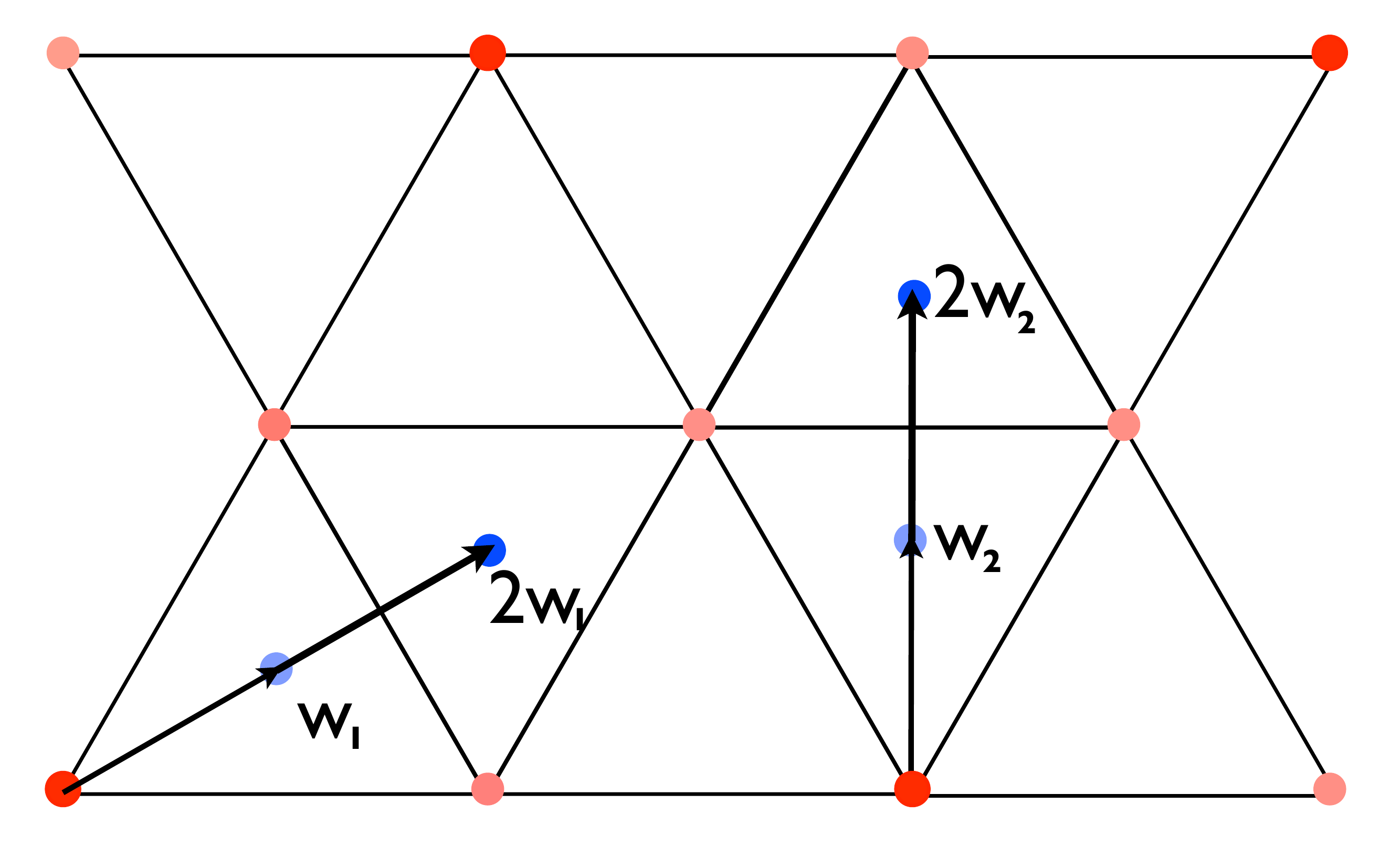}
\caption{The figure shows how the centroids (indicated with solid blue dots) of the new 
side-length-$2$ triangles
(indicated with solid red dots) are obtained as vectors from $q_1+2P$.
The two $2Q$-cosets of $q_1+2P$ which are not $q_1+2Q$ itself indicate 
the centroids of the new up and down
triangles. Notice that the orientations of the new triangles, and hence the orientations
associated with the new centroids, are opposite to the orientations associated with
these points when they were viewed as centroids of the original triangulation. This
explains why $S^\uparrow_2 \subset S^\downarrow_1$ and 
$S^\downarrow_2 \subset S^\uparrow_1$ .}
\label{newCentres}
\end{figure}

We now repeat this whole process. There are four cosets of $4Q$ in $q_1+ 2Q$ and
we select one of them, say $q_1 +q_2 + 4Q$, with $q_2 \in 2Q$, and this gives us a new triangulation
with triangles of side length $4$. Their centroids in $q_1 + q_2 + 4P$  form $4Q$-cosets
$S^\uparrow_3 \subset S^\downarrow_2$ and 
$S^\downarrow_3 \subset S^\uparrow_2$, and we have the decomposition
\[q_1+ q_2 + 4P = (q_1+ q_2 + 4Q) \cup S^\uparrow_3 \cup S^\downarrow_3 \, .\]

Continuing this way we obtain $q_1, q_2, q_3, \dots $ with $q_k \in 2^{k-1}Q$, and
sets $S^\uparrow_k, S^\downarrow_k$ with $S^\uparrow_{k+1}  \subset S^\downarrow_{k}$
and $S^\downarrow_{k+1}  \subset S^\uparrow_{k}$ for all $k=1,2, \dots$, and the partition
\begin{equation}\label{partition}
q_1+ \cdots + q_k+ 2^k P = (q_1+ \cdots + q_k + 2^kQ) \cup S^\uparrow_{k+1} \cup S^\downarrow_{k+1} \, .
\end{equation}

We have 
\begin{eqnarray}\label{Sk defined}
S^\uparrow_{k+1} &=& q_1+ \cdots + q_k + 2^k(w_1+ Q)\\
S^\downarrow_{k+1} &=& q_1+ \cdots + q_k + 2^k(w_2+ Q) \nonumber \,.
\end{eqnarray}
Explicit formulas for
$2^k w_1$ and $2^k w_2$ are given in Lemma~\ref{sFunctions}.

\medskip
We now carry out the entire construction based on an arbitrary infinite sequence
\[(q_1, q_2, \dots, q_k, \dots )\]
where $q_k \in 2^{k-1}Q$ for all $k$. This results in a pattern of overlapping triangulations
based on triangles of edge lengths $1,2,4,8, \dots$ (these are referred to as being
triangles of levels $0,1,2,3, ...$). In \S\ref{Tilings} we shall  make our tiling
out of this pattern. But certain features of the entire pattern are clear:

\begin{itemize}
\item all points involved as vertices of triangles are in $Q$;
\item all triangle centroids are in $P\backslash Q$;
\item there is no translational symmetry.
\end{itemize}
The last of these is due to the fact that there are triangles of all scales, and no
translation can respect all of these scales simultaneously.

A point $x\in P\backslash Q$ is said to {\bf have an orientation} (up or down) if there is
a positive integer $k$ such that for all $k' >k$, $x \notin S^{\uparrow}_{k'}
\cup S^{\downarrow}_{k'}$. Every element of $P\backslash Q$ is in 
$S^{\uparrow}_k$ or $S^{\downarrow}_k$ for $k=1$, and some for other values of $k$ as well. 
For the elements
$x$ which have an orientation there is a largest $k$ for which this is true and this gives its final
orientation. If $x$ has an orientation, we shall say that the {\bf level of its orientation} is $k$
if its orientation stabilizes at $k$. If it does not stabilize we shall say that 
$x$ is {\bf not oriented}. We shall see below (Prop.~ \ref{noOrientation}) what 
it means for a point not to have an orientation.\footnote{We shall introduce levels for a number
of objects that appear in this paper: points, lines, edges, triangles.}

\begin{figure}
\centering
\includegraphics[width=12cm]{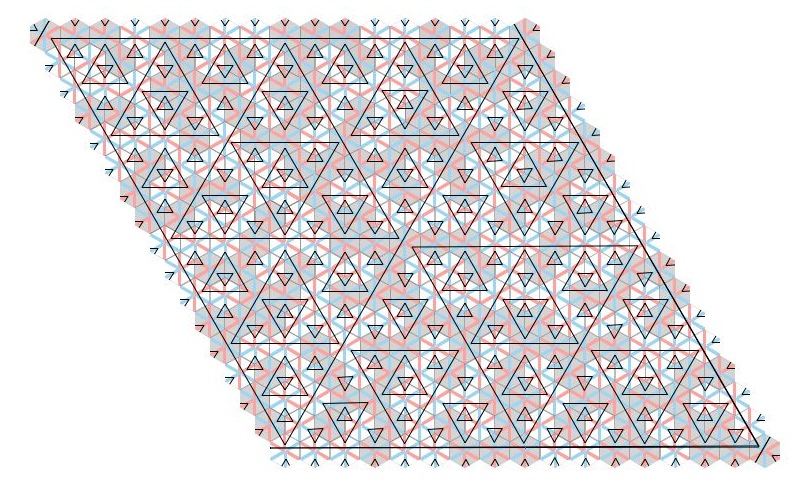}
\caption{The figure shows a pattern of triangles emerging from the construction
indicated in \S\ref{construction}, manifesting the rule {\bf R1}. The underlying hexagonal tiling is 
indicated in light and dark shades which indicate the parity of the hexagons.
The underlying diagonal shading on the hexagons manifests the rules {\bf R2}.}
\label{ST-whole-tiling}
\end{figure}

\section{The $Q$-adic completion}\label{Qadic}

In this section we create and study a completion of $P$ under the 
$Q$-adic topology. The $Q$-adic topology is the uniform topology based
on the metric on $P$ defined by $d(x,y) = 2^{-k}$ if $x-y \in 2^kQ\backslash 2^{k+1}Q$
and $d(x,y):=2$ when $x,y$ are in different cosets of $Q$. This metric is $Q$-translation
invariant.
  $\overline P$ is the completion of $P$ in this topology and $\overline Q$ is
 the closure of $Q$ in $\overline P$, which is also the completion of $Q$ in the
 $Q$-adic topology.
 
 $\overline P$ may be viewed as the set of
sequences
\[(b_1, b_2, \dots ) \]
where $b_k \in P$ for all $k$ and $b_{k+1} \equiv b_k \mod 2^k Q$.

$\overline P$ is a group under component-wise addition and $\overline Q$ is the subgroup of
all such sequences with all components in $Q$. There is the obvious coset decomposition
\[\overline P=
\overline Q \cup (w_1+\overline Q) \cup (w_2 +\overline Q)\, ,\]
so $\overline Q$ has index $3$ in $\overline P$.
We note that $\overline Q$ and $\overline P$ are compact topological groups.

We have $i: P \longrightarrow \overline P$ via
\[ b \mapsto (b,b,b,\dots) \,.\]
We often identify $P$ as a subgroup of $\overline P$ via the embedding $i$.

Note that the construction of expanding triangles of \S\ref{construction} depends on the choice of the element
$(q_1, q_2, \dots )$, where $q_k \in 2^{k-1}Q$. Then we can obtain the compatible sequence
\[{\bf q} = (q_1, q_1+q_2, \dots,   q_1 + q_2 + \cdots + q_k, \dots) \in \overline Q \]
and thus we can identify each possible construction with an element of $\overline Q$.
Let $\CalT({\bf q})$ denote the pattern of triangles arising from ${\bf q}\in \overline Q$.

Let $\mu$ denote the unique Haar measure on $\overline P$ for which 
$\mu(\overline P) =1$. The key feature of $\mu$ is that $\mu(p +2^k \overline Q) = 2^{-k}/3$
for all $p \in \overline P$. We note that $P \subset \overline P$ is countable and has measure
$0$, and that $\mu(\overline Q) = \frac{1}{3}$ and $\mu(\overline{S^\uparrow_k}) = 
\mu(\overline{S^\downarrow_k})
=2^{-k+1}/3$.

\begin{remark} \label{subtlePoint}
We should note a subtle point here. In $\overline Q$ one can divide by $3$.
In fact, for all ${\bf x}\in \overline Q$,
$ -\lim_{k\to\infty} ({\bf x} + 4{\bf x} + 4^2{\bf x} +\dots + 4^k{\bf x}) $ exists since 
$4^k{\bf x} \in 2^{2k}\overline Q$, and 
\[-3 \lim_{k\to\infty} ({\bf x} + 4{\bf x} + 4^2{\bf x} +\dots + 4^k{\bf x})
= \lim_{k\to\infty} (1-4)({\bf x} + 4{\bf x} + 4^2{\bf x} +\dots + 4^k{\bf x}) = \lim_{k\to\infty} (1-4^{k+1}) {\bf x}= {\bf x} \,.\]
Thus we can find an element ${\bf w_1}$ of $\overline Q$ corresponding to $w_1 = \frac{2}{3}a_1 + \frac{1}{3}a_2$ and similarly ${\bf w_2} \in \overline Q$ corresponding to $w_2$ .
However, our view is that $P = Q \cup (w_1 +Q) \cup (w_2+Q)$ and $\overline P$ is the $Q$-adic
completion of this, with each of the three cosets leading to a different coset of $\overline Q$ in $\overline P$. 
Thus $w_1 -{\bf w_1} \neq 0$ but $3(w_1 -{\bf w_1}) =0$ and we conclude that $\overline P$ has $3$-torsion. 

Two examples of this are important in what follows. Define $s_1^{(-1)} :=0$ and
$s_1^{(k)} :=  a_1 + 4 a_1 + 4^2 a_1 + \cdots + 4^{k} a_1) $ 
for $k= 0, 1, \dots$. 
and similarly $s_2^{(k)}$ based on $a_2$. 
Their limits are denoted by ${\bf s_1}, {\bf s_2}$ respectively. They lie in $\overline Q$.
\end{remark}

\begin{lemma} \label{sFunctions}
 For all $k = 0, 1, \dots$,
\begin{eqnarray*}
2^{2k} w_1 &=& w_1 + s_2^{(k-1)} + 2 s_1^{(k-1)}\\
2^{2k+1} w_1 &=& w_2 + s_1^{(k)} + 2 s_2^{(k-1)} \, .
\end{eqnarray*}
Similarly for $2^m w_2$, interchanging the indices $1,2$.

In particular $\lim_{k\to\infty} 2^{2k} w_1 = w_1 + {\bf s_2} + 2{\bf s_1}$
and $\lim_{k\to\infty} 2^{2k+1} w_1 = w_2 + {\bf s_1} + 2{\bf s_2}$.
Furthermore, $3(w_1 + {\bf s_2} + 2{\bf s_1}) = 0 = 3( w_2 + {\bf s_1} + 2{\bf s_2})$.

\end{lemma}

\begin{proof} From the definitions, 
$2w_1 = w_2 +a_1$ and $2w_2 = w_1 +a_2$.
This gives the case $k=0$ of the Lemma. Now proceeding by induction,
\[2^{2k} w_1= 2(w_2 + s_1^{(k-1)} + 2 s_2^{(k-2)}) = w_1 + a_2 + 2s_1^{(k-1)} + 4 s_2^{(k-2)} 
= w_1 + s_2^{(k-1)} + 2 s_1^{(k-1)} \, ,\] as required. Similarly
\[2^{2k+1} w_1 = 2(w_1 + s_2^{(k-1)} + 2 s_1^{(k-1)}) = w_2 +a_1 + 2s_2^{(k-1)} + 4 s_1^{(k-1)}
= w_2 + s_1^{(k)} + 2 s_2^{(k-1)}\,.\]

Taking the limits and using the formula for multiplication by $3$ at the beginning
of Remark~\ref{subtlePoint}, we find that $-3({\bf s_2} + 2{\bf s_1}) = a_2 + 2a_1 = 3w_1$
and similarly with the indices $1, 2$ interchanged.
\end{proof}

Consider what happens if there is a point $x\in P\backslash Q$ which does not have orientation.
This means that there is an infinite sequence $k_1 < k_2 < \cdots$ with 
$x \in S^{\uparrow}_{k_j} \cup S^{\downarrow}_{k_j}$. 
Then from \eqref{Sk defined},
$x \in \left((q_1 + \cdots + q_{k_j-1} + 2^{k_j-1}(w_1 +Q)\right) \cup \left((q_1 + \cdots + q_{k_j-1} + 2^{k_j-1}(w_2 +Q)
\right)$ for each $k_j$. This means $x = {\bf q} + w_1 + {\bf s_2} + 2{\bf s_1}$ or
$x = {\bf q} + w_2 + {\bf s_1} + 2{\bf s_2}$. 

\begin{prop} \label{noOrientation}
$\CalT({\bf q})$ has at most one point without orientation. A point without orientation
can occur if and only if
${\bf q} \in -{\bf s_2} - 2{\bf s_1} + Q$ or ${\bf q} \in -{\bf s_1} - 2{\bf s_2} + Q$. These two
families are countable and disjoint.
\end{prop}

\begin{proof} If $x \in P\backslash Q$ does not have an orientation then either
$x = {\bf q} + w_1 + {\bf s_2} + 2{\bf s_1}$ and $-w_1 + x \in Q$, 
which gives one of the cases; or
$x = {\bf q} + w_2 + {\bf s_1} + 2{\bf s_2}$, which gives the other. Conversely, in either
case we have points without orientation. 
Since in one case $x \in w_1 + Q$ and in the other case
$x\in w_2 + Q$, we see that the two families are disjoint.
\end{proof}

\begin{remark}
We do not need to go into the exact description of the orientations of 
triangles, but confine ourselves to a few remarks here. For any fixed ${\bf q}$, define the sequence of sets $W^\uparrow_k$ and $W^\downarrow_k$, $k=1, 2, \dots$, inductively
by $W^\uparrow_1 = S^\uparrow_1$ and 
\[ W^\uparrow_{k+1} = (W^\uparrow_k \backslash S^\downarrow_{k+1}) \cup S^\uparrow_{k+1}\,,\]
and similarly for $W^\downarrow_{k}$. In other words we put together into $W^\uparrow_k$ all the points which are oriented upwards at step $k$, and likewise all that are oriented downwards at step $k$.

Since $S^\downarrow_{k+1}$ and $S^\uparrow_{k+1}$ have measure $2^{-k}/3$
we see that the sets $W^\uparrow_k$ change by less and less as $k$ increases.
Furthermore it is clear that $\mu(W^\uparrow_k) = 1/3$ for all $k$.

\begin{prop} \label{Wks}
 For all $k$ the sets $\overline{W^\uparrow_k}$ and $\overline{W^\downarrow_k}$ are clopen
 and disjoint. They each have measure $1/3$. \qed
\end{prop}

For each ${\bf q} \in \overline Q$ we define $W^\uparrow({\bf q}):= \overline{\{x: x \;\mbox{which have up orientation}\}}$,
and similarly for $W^\downarrow({\bf q})$.

\begin{prop} \label{Wupdownarrows} 
$\overline P = \overline Q \cup W^\uparrow({\bf q}) \cup W^\downarrow({\bf q}) $ where
$\overline Q$ is disjoint from $W^\uparrow({\bf q}) \cup W^\downarrow({\bf q}) $, and  
\[W^\uparrow ({\bf q}) \cap W^\downarrow ({\bf q}) =
\{{\bf q}+ w_1 + {\bf s_2} + 2{\bf s_1}, {\bf q}+w_2 + {\bf s_1} + 2{\bf s_2}  \}\, .  \] 
$W^\uparrow({\bf q})$ is the union of an open set and 
$\{{\bf q}+ w_1 + {\bf s_2} + 2{\bf s_1}, {\bf q}+w_2 + {\bf s_1} + 2{\bf s_2} \,. \}$
The same goes for $W^\downarrow({\bf q})$. In particular $W^\uparrow({\bf q})$
and $W^\downarrow({\bf q})$ are the closures of their interiors.
Both  $W^\uparrow({\bf q})$ and $W^\downarrow({\bf q})$ are sets of measure $1/3$.\qed
\end{prop}

\end{remark}

\section{The Tiles}\label{Tilings}

 Let us assume that we have carried out a triangulation $\CalT({\bf q})$ as described
 in  \S\ref{construction}. We now
 have an overlaid pattern of equilateral triangles of side lengths
$1,2,4, \dots$. Each of these triangles has vertices in $Q$ and its centroid 
in $P\backslash Q$. The points of the two cosets of $P$ different from $Q$ (shown
as blue points in Fig.~\ref{hex}) form the vertexes of a tiling of hexagons made from the
triangulation, see Fig.~\ref{triangulation}. This tiling, with the tiles
suitably marked, is the tiling that we wish to understand. 
Our objective is to give each hexagon of the tiling markings in the form of a black
stripe and three colored diagonals as shown in Fig.~\ref{two-basic-tiles}.

Apart from the lines of the triangulation (which give rise to short diagonals 
of the hexagons of the tiling) we also have the lines on which the long diagonals of
the hexagons lie and which carry the color. To distinguish these sets of lines we 
call the triangulation lines
a-{\bf lines} (since they are in the directions $a_1, a_2, a_1+a_2$) and 
the other set of lines w-{\bf lines} (since they are in the directions $w_1, w_2, w_2-w_1$). 
We also call the w-lines  {\bf coloring lines}, since they are the ones carrying the
colors red and blue. The w-lines pass through
the centroids of the triangles of the triangulation. We say that a w-line has {\bf level} $k$ if
there are centroids of level $k$ triangles on it, but none of any higher level. We shall discuss
the possibility of w-lines that do not have a level in this sense below. Note that every point of
$P\backslash Q$ is the centroid of some triangle, some of several, or even many!

There are two steps required to produce the markings on the tiles. One is to shift
triangle edges off center so as to produce the appropriate stripes on the tiles. We refer
to this step as {\em edge shifting}. The second is to appropriately color the main diagonals of each
tile. This we refer to as {\em coloring}. The two steps can be made in either order.
However, each of the two steps requires certain generic aspects of the triangulation to be
respected in order to be carried out to completion. We  first discuss the nature of these generic
conditions and then finish this section by showing how edge shifting is carried out. 

We need to understand the structure of the various lines (formed from the edges of the 
various sized triangles) that pass through each hexagon. Let us say that an \emph{element} of $Q$
is of {\bf level} $k$ if it is a vertex of a triangle of edge length $2^k$ but is not
a vertex of any longer edge length. Similarly an \emph{edge} of a triangle is of
{\bf level} $k$ if it is of length $2^k$, and an $a$-line (made up of edges) is of {\bf level}
$k$ if the longest edges making it up are of length $2^k$. All lines of all levels
are made from the original set of lines arising from the original triangulation 
by triangles of edge length $1$, so a line of level $k$ has edges of lengths
$1, 2, \dots, 2^k$ on it.

\bigskip
\begin{figure}
\centering
\includegraphics[width=10 cm]{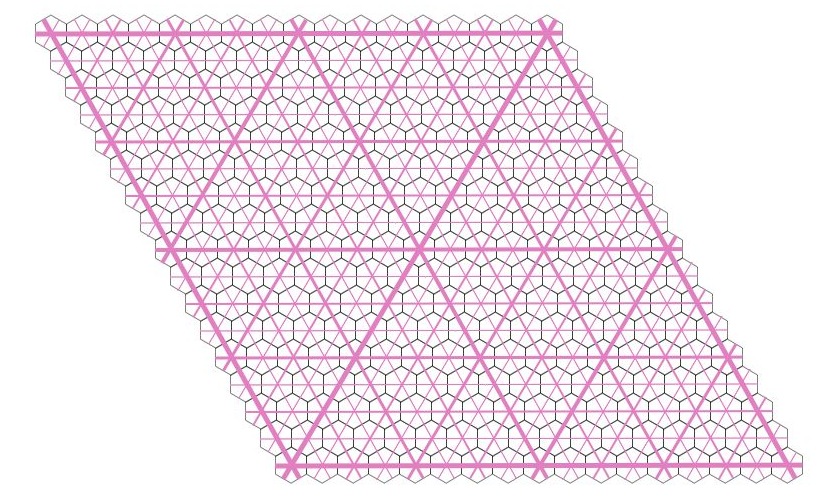}
\caption{A partial triangulation of the plane overlaid on the basic lattice
of hexagons which will make up the tiles. The levels of the triangles are indicated
by increasing thickness. One can clearly see triangles of levels $0,1,2,3,4$ and one can
also see how triangle edges of level $k$ ultimately become edges passing through the
interior of triangles of level $k+1$. This will be used to make the shifting of edges later on.}
\label{triangulation}
\end{figure}

The word `level' occurs in a variety of senses in the paper. These are summarized in Table~\ref{levelTable}. 

\begin{table}[htdp]
\caption{Uses of the word `level' $k$ and section number where it is defined. 
If there is no such $k$ the level is infinite.}
\begin{center}
\begin{tabular}{|l|l|}\hline
of a triangle & \S\ref{construction} $k$ if the side length is $2^k$, \\
& ~~~ where a side length $1=2^0$ is the length of $a_1$ and $a_2$\\
of orientation of $x\in P$ & \S\ref{construction} $k$ at which $x$ stops switching between 
$S_k^\uparrow$ and $S_k^\downarrow$\\
of a $w$-line & \S\ref{Tilings} max. $k$ of centroids of level $k$ triangles on it \\
of a point of $Q$ & \S\ref{Tilings} max. $k$ for which it is a vertex of a triangle of level $k$\\
of a triangle edge & \S\ref{Tilings} $k$ for which it is an edge of a level $k$ triangle\\
of an $a$-line & \S\ref{Tilings} max. $k$ for $k$-edges on this line\\
\hline
\end{tabular}
\end{center}
\label{levelTable}
\end{table}

There are two types of generic assumptions that we need to consider.

\begin{definition} A triangulation (or the value of ${\bf q}$ associated with it) in which every w-line
has a finite level is called {\bf generic}-$w$. This means that for every w-line
there is a finite bound
on the levels of the centroids (points of $P\backslash Q$) that lie on that line. In this case for any ball of
 any radius anywhere in the plane, there is a level beyond which no w-lines
 of higher level cut through that ball. See Fig.~\ref{iCw-L-tiling} for an example that shows
 failure of the generic-w condition. 
  
A triangulation (or the value of ${\bf q}$ associated with it) is said to be {\bf generic}-$a$ if every a-line has a finite level. This means
for every a-line there is a finite bound on the levels of edges that lie in that line. In this case for any ball
of any radius anywhere in the plane, there is a level beyond which no lines of the triangulation
of higher level cut through that ball. See Fig.~\ref{genericwNota} and Fig.~\ref{iCw-L-tiling}.

A tiling is said to be {\bf generic} if it is both generic-w and generic-a. All other tilings (or elements
${\bf q}\in Q$) are called {\bf singular}.
One case of the failure of generic-w is discussed in Prop.~\ref{noOrientation} above. 
The only way for one of our generic conditions to fail is that
there are a-lines or w-lines of infinite level. This situation is discussed in \S\ref{TheHull}. 
\end{definition}

Every element of $Q$ has a hexagon around it and three lines passing through it
in the directions $\pm a_1, \pm a_2, \pm(a_1+a_2)$. These lines 
pass through pairs of opposite edges of the hexagon at right-angles to those
edges. We shall call these lines {\bf short diameters}. These short diameters arise
out of the edges of the triangles of the triangulations that we have created. 
Each triangle edge is part of a line
which is a union of edges, all of the same level. As we have pointed out, the line (and its edges)
have level $k$ if they occur at level $k$ (and no higher). The original triangulation has level $0$. 
One should
note that a line may occur as part of the edges of many levels of triangles, but under
the assumption of generic-a there will be a highest level of triangles utilizing a given line, and it is
this highest level that gives the line its level and determines the corresponding edges.

In looking at the construction of level $1$ triangles out of the original triangulation
of level $0$ triangles, we note immediately that every point of $Q$ has at least one line
 of level $1$ through it (though by the time the triangularization is complete this line may 
 have risen to higher level), see Fig.~\ref{cosets}. The vertices of the level $1$ triangles have three lines
 of level $1$ through them, and the rest (the mid-points of the sides
 of the level $1$ triangles) have just one of level $1$ and the other two of level $0$. 
 Thus at this stage of the construction each hexagon has either one short diameter from a level $1$ line  
 or it has $3$ short diameters all of level $1$. 
 
 This is the point to remember: at each stage of determining the higher level triangles,
 we find that the hexagon around each element of $Q$ is of one of two kinds: it
 either has three short diameters of which two have equal level and the third a higher level,
 or three short diameters all of the same level $k$. The latter only occurs when the element of
 $Q$ is a vertex of a triangle of level $k$. 
 Since we are in the generic-a case, there is no element of $Q$
 which is a vertex of triangles of unbounded scales, and the second condition cannot hold
 indefinitely. Once an element of $Q$ is not a vertex at some level then it never becomes
 a vertex at any other higher level (all vertices of triangles at each level are formed from
 vertices of triangles at the previous level). 
 
 We conclude ultimately that in the generic-a cases every hexagon has three short
 diameters of which two are of one level and one of a higher level. See Fig.~\ref{triangulation}.
 
 \begin{lemma}  \label{generic-a and three short diameters}
 For ${\bf q}$ satisfying generic-a 
 each hexagonal tile of $\CalT({\bf q})$ has three short diameters of which exactly 
 one has the largest level and the other two equal but lesser levels.
\end{lemma}  

We now describe edge shifting. Fix any $\epsilon$ with $0< \epsilon \le 1/4$. This $\epsilon$ is going to be the distance
 by which lines are shifted. It is fixed throughout, but it exact value plays no role in 
 the discussion. Take a tiling based on ${\bf q}$. 
 
 Now consider any edge that has level $k< \infty$ but does not occur as part of an edge of higher level. This edge occurs as an 
 edge \emph{inside} some triangle $T$ of level $k+1$, 
and this allows us to distinguish two sides of that edge. The side of the edge on which the centroid of $T$
 lies is called the {\bf inner} side of the edge, and the other side its {\bf outer} side.
This edge (but not the entire line) is shifted inwards (i.e. towards the centroid of $T$) by 
 the distance $\epsilon$. Note that the shifting distance $\epsilon$ is independent of $k$. 
 This shifted edge then becomes the {\em black} stripe on the 
hexagonal tiles through which this edge cuts, see Fig.~\ref{basic-hexagon}. 
 Fig.~\ref{triangulationWithShifts} shows how edge shifting works.
At the end of shifting, each hexagon has on it a pattern made by the shifted triangle edges that looks like the one
shown in Fig.~\ref{basic-hexagon}.

\bigskip
\begin{figure}
\centering
\includegraphics[width=8cm]{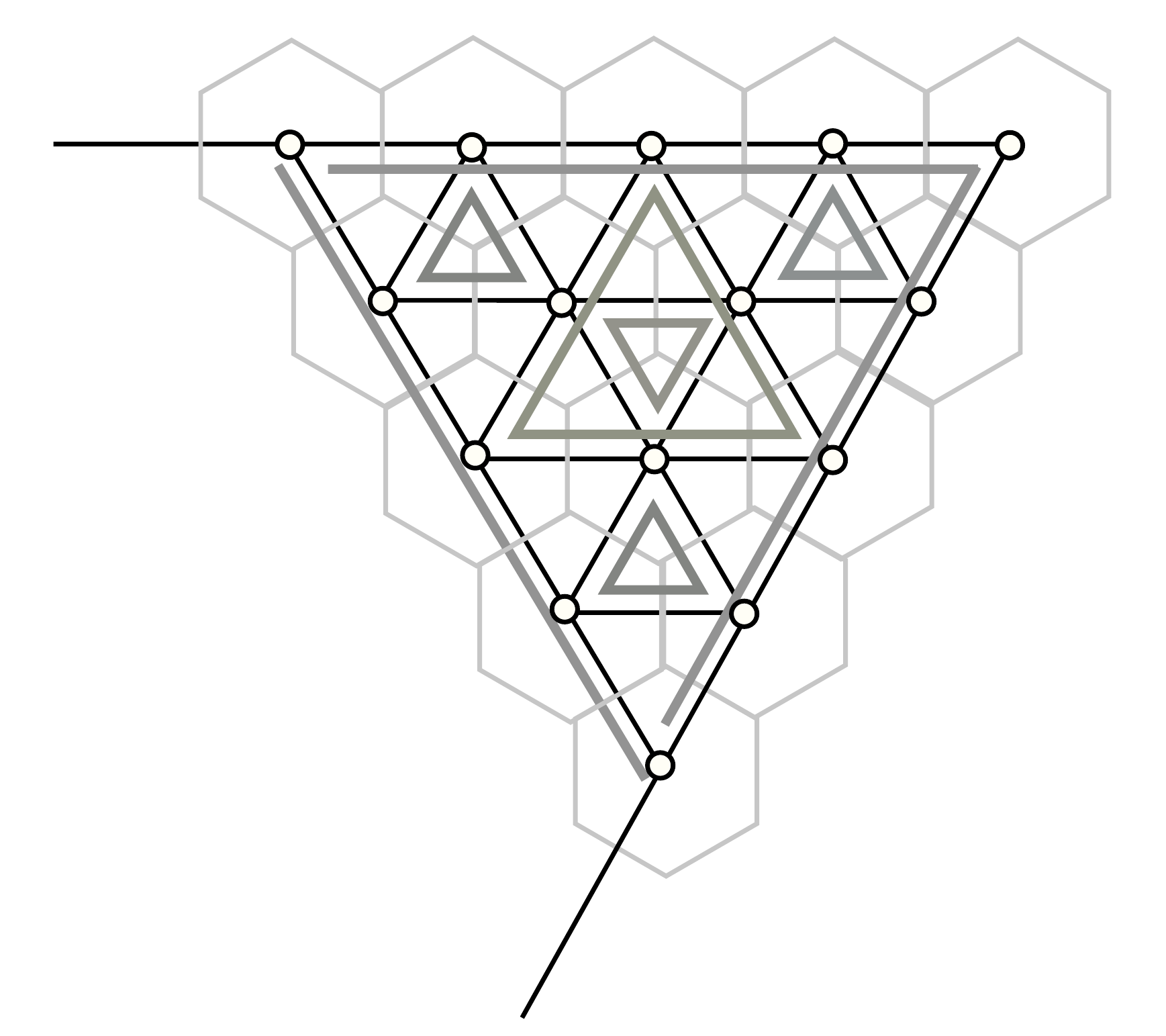}
\caption{The figure shows how edge shifting is done. Part of a triangulation 
is shown in thin black lines. The shifted edges are shown in thicker gray
lines. The extended black lines indicate that the largest (level $2$) triangle sits
in the top right corner of a level $3$ triangle which is not shown in full. Note
how the edges of the level $2$ triangle shift.}
\label{triangulationWithShifts}
\end{figure}

\begin{figure}
\centering
\includegraphics[width=3.5cm]{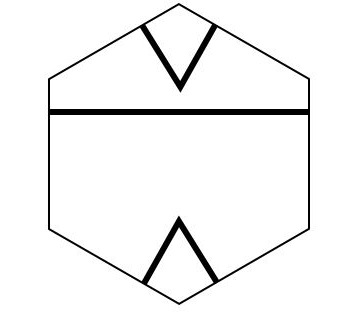}
\caption{The basic hexagon with its markings arising from shrunken
triangles.}
\label{basic-hexagon}
\end{figure}

In the case that ${\bf q}$ satisfies generic-a, the edges of every line of the triangulation are of bounded length.
Thus every edge undergoes a shift by the prescription above. Thus, 

\begin{prop} If $\CalT({\bf q})$ satisfies
the condition generic-a then there is a uniquely determined edge shifting on it. \qed 
\end{prop}

\section{Color}\label{color}

So far we have constructed a triangulation from our choice of ${\bf q}$, and shown
how edges can be shifted to produce the corresponding hexagonal tiling with the tiles suitably marked by black stripes. We wish now to show how the (long) diagonals of the hexagons are to be colored.
This amounts to producing a color (red, blue, or red-blue) for each of the long diagonals
of each hexagon of the tiling. The only requirement is that the overall coloring obey the
rule ${\bf R2}$ that is used to make Taylor--Socolar tilings. 

\medskip

As we mentioned above, coloring is made independently of shifting in the sense that the two processes can
be done in either order. In fact, in this argument
we shall suppose that the stripes have not been shifted, so they still run through the centroids of the tiles.
\medskip

We shall show that for $\bf{q} \in \overline Q$ in the generic case there is exactly one allowable coloring.
\smallskip

Assume that we have a generic tiling (this means both $a$ {\em and} $w$ generic). Now consider any hexagon of the tiling. We note from Lemma \ref{generic-a and three short diameters} 
that it has three short diameters, one of which is uniquely
of highest level, and it is this last short diameter which determines (after shifting) the black stripe for this hexagon. 
We will refer to this short diameter as the stripe, even though in this discussion it has not been shifted.  The
other two colored (long) diameters are a red one which lies at $\pi/6$ clockwise of the stripe and
a blue one which lies $\pi/6$ counterclockwise of the stripe. The red-blue diameter cuts the
stripe at right-angles, but which way around it is (red-blue or blue-red) is not determined yet.

\smallskip

Consider Fig.~\ref{coloring1Rev}
in which we see two complete level $1$ triangles overlaid on the basic level $0$ triangles.
Tiles of the hexagonal tiling are shown on points of $Q$ with the hexagons at the vertices
of the level $1$ triangles shown in green. These latter are points of 
$q_1 + 2Q$. At each point of $Q$ there are three edge lines running through it. But notice that at the midpoints of the sides of the level $1$ triangles (white hexagons), the edge belonging to the level $1$ triangle has higher level than the other two. This is the edge that will become the stripe for the hexagon at that
point. This stripe {\em forces} the red and blue diameters for this hexagon. 

\begin{figure}
\centering
\includegraphics[width=7cm]{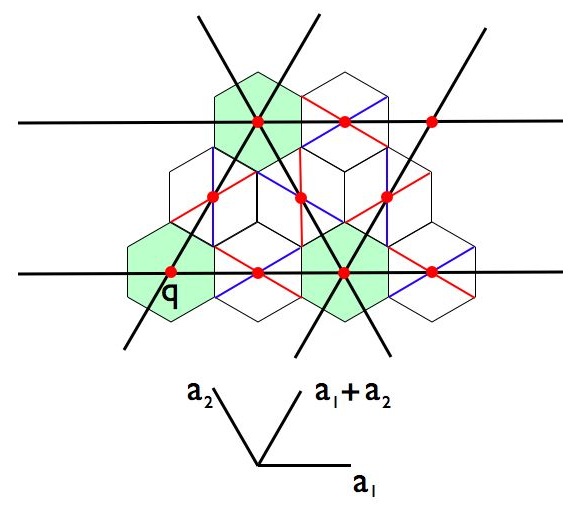}
\caption{The figure shows some hexagonal tiles, each centered on a point of $Q$. The point
$q$ is assumed to be in the coset $q_1+2Q$ and the gray hexagons are those in the
picture whose centers are in this coset.  The white hexagons are centered at points
from all three of the remaining cosets of $Q$ relative to $2Q$. These are the midpoints of the edges
of the level $1$ triangles. Notice in each, the 
red and blue diagonals clockwise and counterclockwise of the direction of the 
black stripes. At the bottom we see the three vectors 
$a_1, a_2, a_1+a_2$. The centers of the white hexagons are, reading left to right and
bottom to top, $q_1+a_1, q_1 +3a_1; q_1+a_1 +a_2,
q_1+2a_1 + a_2,  q_1+3a_1 + a_2;  q_1+ 3a_1 + 2a_2$. 
The picture manifests the rule {\bf R2} and shows that elements of the 
same coset carry the same orientation of diameters. Note that from the rotational symmetry
of the process and the fact that the hexagons centered on $q_1+a_1$ and $q_1+3a_1$ have
identically aligned diagonals, we can infer that this property is retained across each of the
cosets $q_1+ a_1+2Q, q_1+ a_2+2Q, q_1+a_1 +a_2 + 2Q$. }
\label{coloring1Rev}
\end{figure}

The idea behind coloring is based on extrapolating this argument to w-lines passing through midpoints of higher level triangles.  
Consider Fig.~\ref{coloring2}. The point $u$ is the midpoint
of an edge of a triangle $T'$ of level $3$. Drawing the $w$-line $L$ towards the centroid $d$ of the top left 
corner triangle $T$ of level $2$ we see first of all that the edge of the level $3$ triangle through $u$ is
the highest level edge through $u$ and hence the coloring along the $w$-line $L$ starts off red, as shown.
Now the rule R2 forces the next part of the coloring to be blue and we come to the hexagon center
$e$. This has three edges through it, but the one that our $w$-line crosses at right-angles has the highest
level, and so will produce the stripe for the corresponding hexagon. The color must switch at the stripe,
and so we see the next red segment as we come to $d$. 

And so it goes, until we reach the point $v$. Here $L$ meets the midpoint of the edge of another level $3$
triangle. This edge produces the stripe for the hexagon at $v$, but it is not at right-angles
to $L$, so there is no color change on $L$ at $v$. Since $v$ is the midpoint of this level $3$
triangle, the same argument that we used at $u$ shows that the coloring should start off blue,
as indeed we have seen it does. At this point one can see by the glide reflection symmetry
along $L$ that the entire line $L$ will ultimately be colored so as to fully respect the rule R2. 
For a full example where one can see the translational symmetry take over, the reader can fill in the coloring on the gray line through $y$.

One can see a similar $w$-line coloring of the $w$-line passing through $q$ and $c$. This time
the point $q$ is the midpoint of an edge of a level $4$ triangle and $c$ is the centroid of one of the 
level $3$ corner triangles of this level $4$ triangle. The pair $r,c$ produces another example, with this
time the first color out of $r$ being blue. 

Finally, we show part of a potential line coloring starting at $s$ towards $t$. We say `potential' because
from the figure we do not know how the level $5$ triangles lie. If $s$ is a midpoint of an edge of 
a level $5$ triangle, then the indicated $w$-line is colored as shown. If $s$ is not a midpoint then this $w$-line
is not yet colorable. 

\begin{figure}
\centering
\includegraphics[width=8cm]{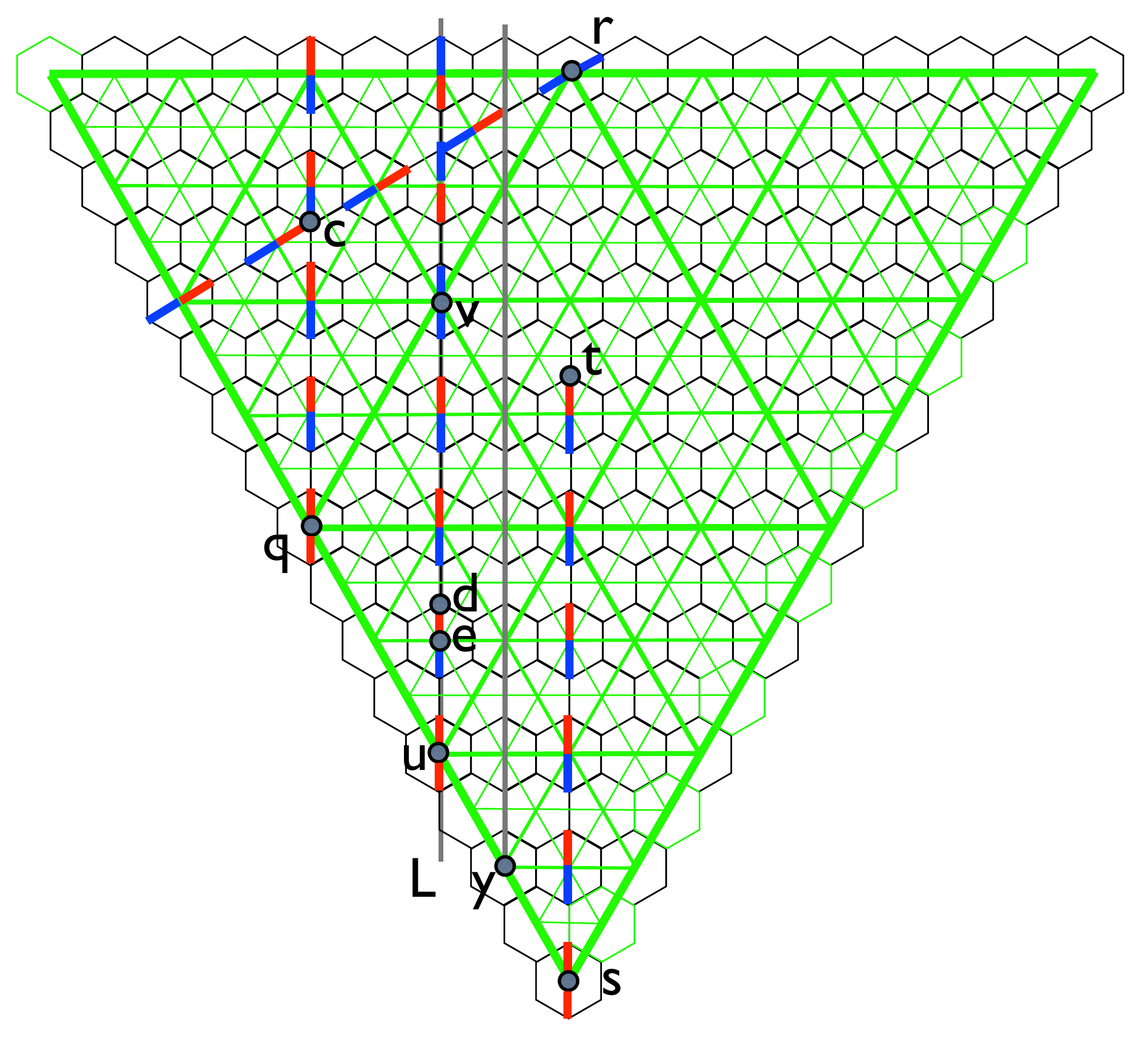}
\caption{Coloring of lines. Colors are forced
on $w$-lines as they pass through the midpoint of a triangle
directed towards the centroid of one of its corner triangles.}
\label{coloring2}
\end{figure}

\smallskip
We can thus continue in this way indefinitely. The important question is, does every tile get fully
colored in the process? 
Using condition generic-w, the answer is yes. To see this note that
each element of $p$ of $P\backslash Q$ has three coloring lines through it. It will suffice
to prove that the process described above will color these three coloring lines. 

Now assuming the condition
generic-w we know that $p$ has an orientation. This means that it is the centroid of some triangle $T$
of level $k$ in the triangulation, and it is not the centroid of any higher level triangle. The triangle
$T$ then sits as one of the corner triangles in a triangle $T'$ of level $k+1$. Up to orientation, the
situation is that shown in Fig.~\ref{coloringCompletion}. The colors of the two hexagons shown 
are then determined because the edges of $T'$ produces stripes on them. Thus the two corresponding coloring lines that pass
through $p$ are indeed colored. Thus the colorings of these two coloring lines through $p$,
the ones that pass through the mid-points of two sides of $T'$ are forced. 

What about the third line $l$ through $p$ (shown as the dotted line in Fig.~\ref{coloringCompletion})?
We wish to see this as a w-line through a midpoint of an edge, just as we saw the other two lines. 
We look at the centroid $p'$ of $T'$, since
the coloring line $l$, which is through $v$ and $p$, is the same as the line through $v$ and $p'$ and the centroid $p'$ is of higher level than $p$ and also has an orientation. We can repeat the process we 
just went through with $p$ with $p'$ instead, to get a new
triangle $T''$ of which $p'$ is the centroid, and a triangle $T'''$ in which $T''$ sits as one of its
corners ($p'$ is the centroid of $T'$ but it may be the centroid of higher level
triangles as well). 

If this still fails to pick up the line $l$ then
 it must be that $l$ still passes through a vertex of $T'''$ (as opposed to through the midpoint of one of its edges)
 and the line $l$ passes through the centroid $p'''$ of $T'''$. However, $p'''$ is of higher level still than that
 of $p'$. The upshot of this is that if we never reach a forced coloring of $l$ (so that it remains forever
 uncolored in our coloring process) then we have on the line $l$ centroids of triangles of
 unbounded levels. This violates condition generic-w. Thus in the generic situation the coloring
 does reach every coloring line and the coloring is complete in the limit.
 
 This completes the argument that there is one and only one coloring for each generic triangularization.

\begin{figure}
\centering
\includegraphics[width=8cm]{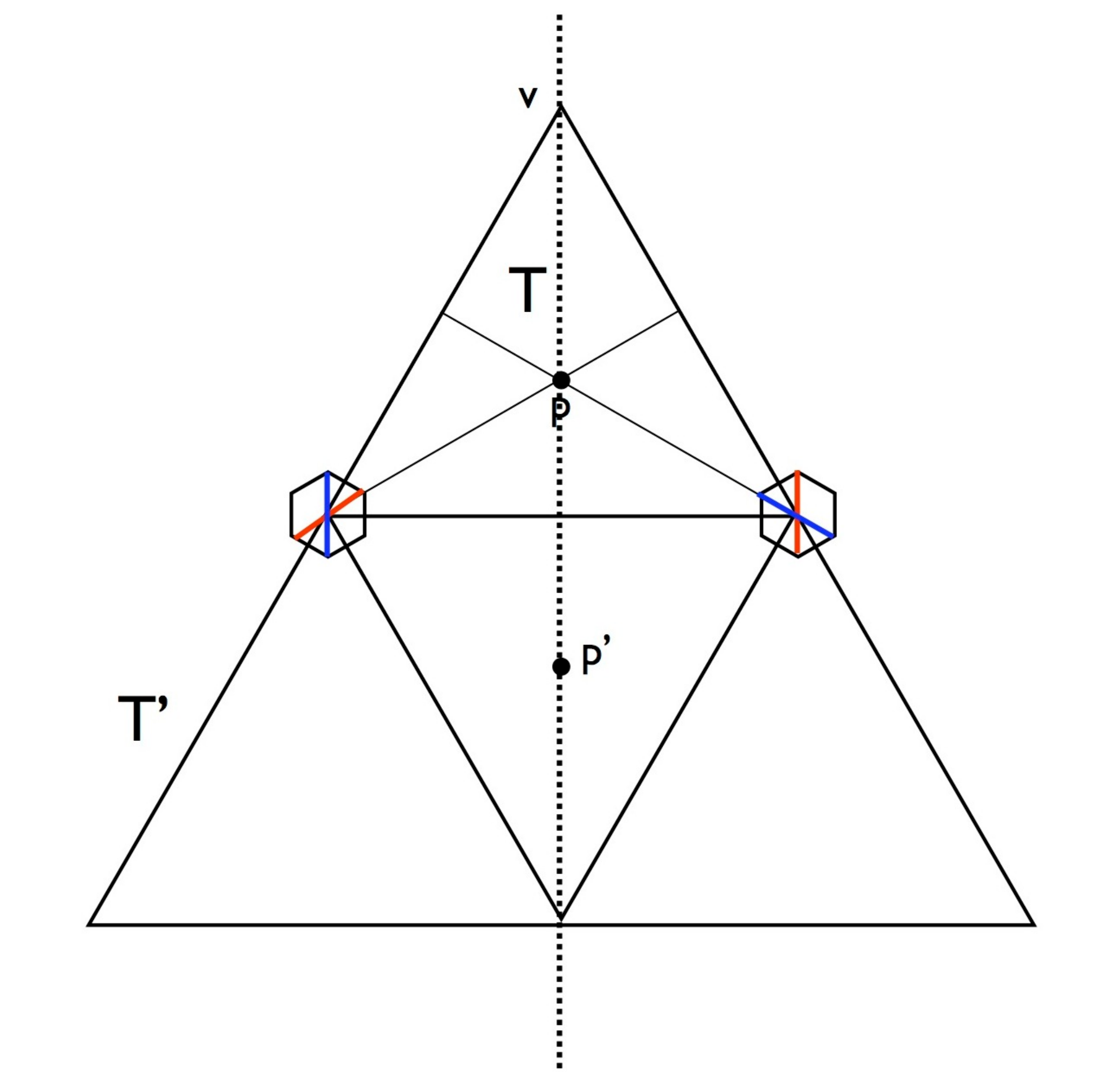}
\caption{The figure shows how the centroid of the triangle $T$, 
which is in the top corner of the main triangle, is on two coloring lines. The third
coloring line through $p$ is the dotted line through $v$. This passes through
the centroid $p'$ of $T'$.}
\label{coloringCompletion}
\end{figure}

\smallskip
If one is presented with a triangularization and wishes to put in the colors, then
one sees that the coloring becomes known in stages, looking at the triangles
(equivalently cosets) of ever increasing levels. Figure~\ref{ST-tiling-Figure3b} shows the amount
of color information that can be gleaned at level $k=2$.

\begin{figure}
\centering
\includegraphics[width=10cm]{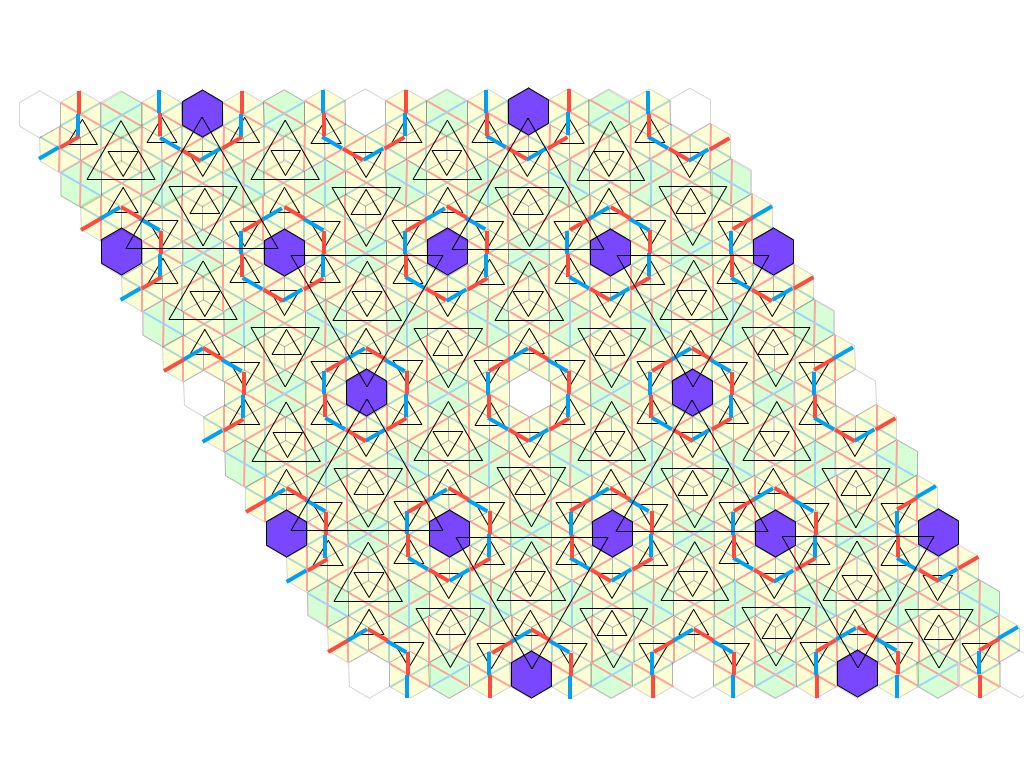}
\caption{This figure shows how the coloring appears if one determines the coloring
by the information in increasing coset levels. This figure corresponds to the process at $k=2$. 
The triangle vertices and their corresponding
hexagons are indicated at levels $0,1,2$ and the corresponding partial coloring is noted. }
\label{ST-tiling-Figure3b}
\end{figure}

\begin{prop} Any generic tiling is uniquely colorable. \qed
\end{prop}

We note that in the generic situation, the shifting and coloring are determined locally. That is,
if one wishes to create the marked tiles for a finite patch of a generic tiling, one need only
examine the tiling in a finite neighbourhood of that patch. That is because the shifting and
coloring depend only on knowing levels of lines, and what the levels of various points
on them are. Because of the generic conditions, these levels are all bounded in any finite 
patch and one needs only to look a finite distance out from the patch in order to pick
up all the appropriate centroids and triangle edges to decide on the coloring and shifting within
the patch.  Of course the radii of the patches are not uniformly bounded across the entire tiling.

Here we offer a different proof of a result that appears in \cite{ST}:

\begin{prop} \label{threeColorThm}
In any generic tiling and at any point $p$ which is a hexagon vertex, the colors of the three
concurrent diagonals of the three hexagons that surround $p$ are not all the same {\rm (}where they meet at $p${\rm )}. 
\end{prop}

\noindent
{\bf Proof}: The point $p$ is the centroid of come corner triangle of one of the triangles
of the triangulation. Fig.~\ref{coloringCompletion} shows how the coloring is forced along two of the 
medians of the corner triangle and that they force opposite colorings at $p$. See also
Fig.~\ref{CHTVertical2}. \qed

\subsection{Completeness}\
We have now shown how to work from a triangularization to a tiling satisfying the matching
rules {\bf R1,R2}. Does this procedure produce all possible tilings satisfying these rules? The answer
is yes, and this is already implicit in \cite{ST}. We refer the reader to the paper for details, but
the point is that in creating a tiling following the rules a triangle pattern emerges from the
stripes of the hexagons. This triangularization can be viewed as the edge-shifting of a triangulation $\CalT$ conforming to our edge shifting rule. 
Thus we know that working with all triangulations,
as we do, we are bound to be able to produce the same shrunken triangle pattern as appears in $T$.

In the generic cases, the coloring that we impose on this triangulation is precisely that forced by {\bf R2}. 
When we discuss the non-generic cases in \S\ref{Non-genericCases}, we shall see that for non-generic
triangulations there are actually choices for the colorings of some lines,
but these choices exhaust the possibilities allowed by the rule {\bf R2}. Thus the
tiling $T$ must be among those that we construct from $\CalT$ and so we see that our procedure
does create all possible tilings conforming to the matching rules. 

When we determine the structure of the hull in \S\ref{HullStructure} we shall also see that 
it is comprised of a minimal hull and two highly exceptional countable families of tilings. 
The former contains triangulations of all types and, as the terminology indicates,  the orbit closure of any one of its tilings
contains all the others in the minimal hull, so in a sense if you have one then you have them all. 
(The two exceptional families of tilings appear in the rule based development of the tilings, but
do not appear in the inflation rule description.)

\section{The Hull}\label{TheHull}

\subsection{Introducing the hull}\
Let $X_Q$ denote the set of all Taylor--Socolar hexagonal tilings whose hexagons are centered
on the points of $Q$ and whose vertices are the points of $P\backslash Q$. The group $Q$ (with the
discrete topology) acts on
$X_Q$ by translations. We let 
$X$ be the set of all translations by $\RR^2$ of the elements of $X_Q$.  We call $X_Q$ and $X$ the
{\bf hulls} of the Taylor--Socolar tiling system.
We give $X_Q$ and $X$ the usual local
topologies - two tilings are close if they agree on a large ball around the origin allowing small shifts. 
In the case of $X_Q$ one can do away with `the small shifts' part. See \cite{LMS1} for the topology. 

In fact it is easy to see that, although we have produced it out of $X_Q$, $X$ is just the standard 
hull that one would expect
from the set of all Taylor--Socolar tilings when they have not been anchored onto the points of $Q$.
Thus $X$ is compact, and since $X_Q$ is a closed subset of it, it too is compact. 

The translation actions of $Q$ on $X_Q$ and $\RR^2$ on $X$ are continuous.
We note that the hulls $X_Q$ and $X$ are invariant under six-fold rotation and under
complete interchange of the two tile types.
Our task is to provide some understanding of $X_Q$ and $X$. Here we shall stick primarily to
$X_Q$ since the corresponding results for $X$ are easily inferred. We let $X_Q^{gen}$ denote the set of all the generic tilings in $X_Q$.

Each element $\Lambda \in X_Q$ produces a triangularization of the plane, and the triangularizations
are parameterized precisely by elements in $\overline Q$. In particular there is an element
${\bf q}(\Lambda) \in \overline{Q}$ corresponding to $\Lambda$, and we have a surjective mapping
\begin{eqnarray} \label{surjective map from X_Q to overlineQ}
\xi: X_Q &\longrightarrow& \overline Q\\
 \Lambda &\mapsto& {\bf q}(\Lambda) \nonumber \,.
\end{eqnarray}

\begin{prop} The mapping $\xi$ is continuous {\rm (}with respect to the local topology on $X_Q$ and
the $Q$-adic topology on $\overline Q${\rm )}. Furthermore $\xi$ is $1-1$ on $X_Q^{gen}$. 
\end{prop}

\noindent
{\bf Proof:} Any ${\bf q} \in \overline Q$ is determined by its congruence classes
modulo $2\overline Q, 4\overline Q, \dots$, which are represented equally well
by the congruence classes of $Q$ modulo $2Q, 4Q, \dots$. These congruence classes
are the sets of vertices
of the triangles of increasing sizes, starting with those of level $1$. Now any patch of tiles
containing a ball $B_R$, $R>2$, will determine part of the triangulation with triangles of all
levels $1,2, \dots n$ for some $n = n(R)$, and we have $n(R) \to \infty $ as
$R\to \infty$. The larger the patch the more congruence classes we know, and
this is the continuity statement. 

In the case of a generic tiling $\Lambda$,  ${\bf q}(\Lambda)$ already determines the entire
markings of the tiles and hence determines $\Lambda$. Thus $\xi$ is $1-1$ on $X_Q^{gen}$.\qed

Below we shall see that with respect to the Haar measure on $\overline Q$ the set of singular (i.e.
non-generic) ${\bf q}$ is of measure $0$. A consequence of this
is \cite{BLM}, Thm. 6:

\begin{coro} \label{ms1}
$X_Q$ is uniquely ergodic and the elements of $X_Q^{gen}$ are regular model sets. \qed
\end{coro}

We shall make the model sets rather explicit in \S\ref{model sets}. 

\subsection{Exceptional cases}\ \label{Non-genericCases}
We now consider what happens in the case of non-generic tilings. To be non-generic a tiling
must violate either generic-a or generic-w. We consider these two situations in turn.

\subsubsection{Violation of generic-a} In the case of violation of generic-a, there is an a-line of infinite level (that is, it does not
have a level as we have defined it). Let $\overline{\ZZ_2}$ be the $Q$-adic completion of $\ZZ$.

\begin{prop} \label{exceptional-a}
A tiling $\Lambda$, where $\xi(\Lambda) = {\bf q}$, has an a-line of infinite level if and only if 
\, ${\bf q} \in x + \overline{\ZZ_2} a$
for some $a \in \{a_1, a_2, a_1 +a_2\}$ and some $x \in Q$. 
Furthermore when this happens the points of $Q$ lying on the infinite-level-line are those of the set
 $x +{\ZZ} a$.
\end{prop}
\noindent
{\bf Proof:} All lines of the triangulation are in the directions $\pm a_1, \pm a_2,\pm (a_1+ a_2)$
and all lines of the triangulation contain edges of all levels from $0$ up to the level of the line itself.
Thus the points of $Q$ on any line $l$ of the triangulation are always a set of the form
$x + \ZZ a$ where $a \in \{a_1,a_2, a_1+a_2\}$ and $x \in Q \cap l$. 

Suppose that we have a line $l$ of infinite level and its intersection with 
$Q$ is contained in $x +\ZZ a$. The line $l$ has elements $y_1, y_2, \dots$ where $y_k$
is a vertex of a triangle of level $k$. This means that $y_1 \in q_1 + 2Q, y_2 \in q_1+ q_2 +4Q, \dots$.
We conclude that $\{y_k\} \to {\bf q}$. Furthermore $y_{k+1} - y_k \in 2^k Q \cap \ZZ a = 2^k \ZZ a$. This is
true for all $k\ge 0$ if we define $y_0 = x$. 
Writing $y_{k+1} - y_k = 2^k u_k a$ with $u_k \in \ZZ$ and ${\bf u} = (0, u_1, \dots, \sum_{j=1}^k u_j 2^j, \dots)$, we have
\[ y_{k+1} = x + \left(\sum_{j=1}^k u_j2^j\right)a \ \rightarrow \ x + {\bf u} a  \]
where $ {\bf u} \in \overline{\ZZ_2}$. Thus ${\bf q} = x + {\bf u} a $. This proves the 
only if part of the Lemma. 

Going in the reverse direction, if ${\bf q} = x + {\bf u} a $ then this is a prescription for
a line of points in $Q$ that have vertices of all levels. Then the line is of infinite level. \qed

\begin{figure}
\centering
\includegraphics[width=8cm]{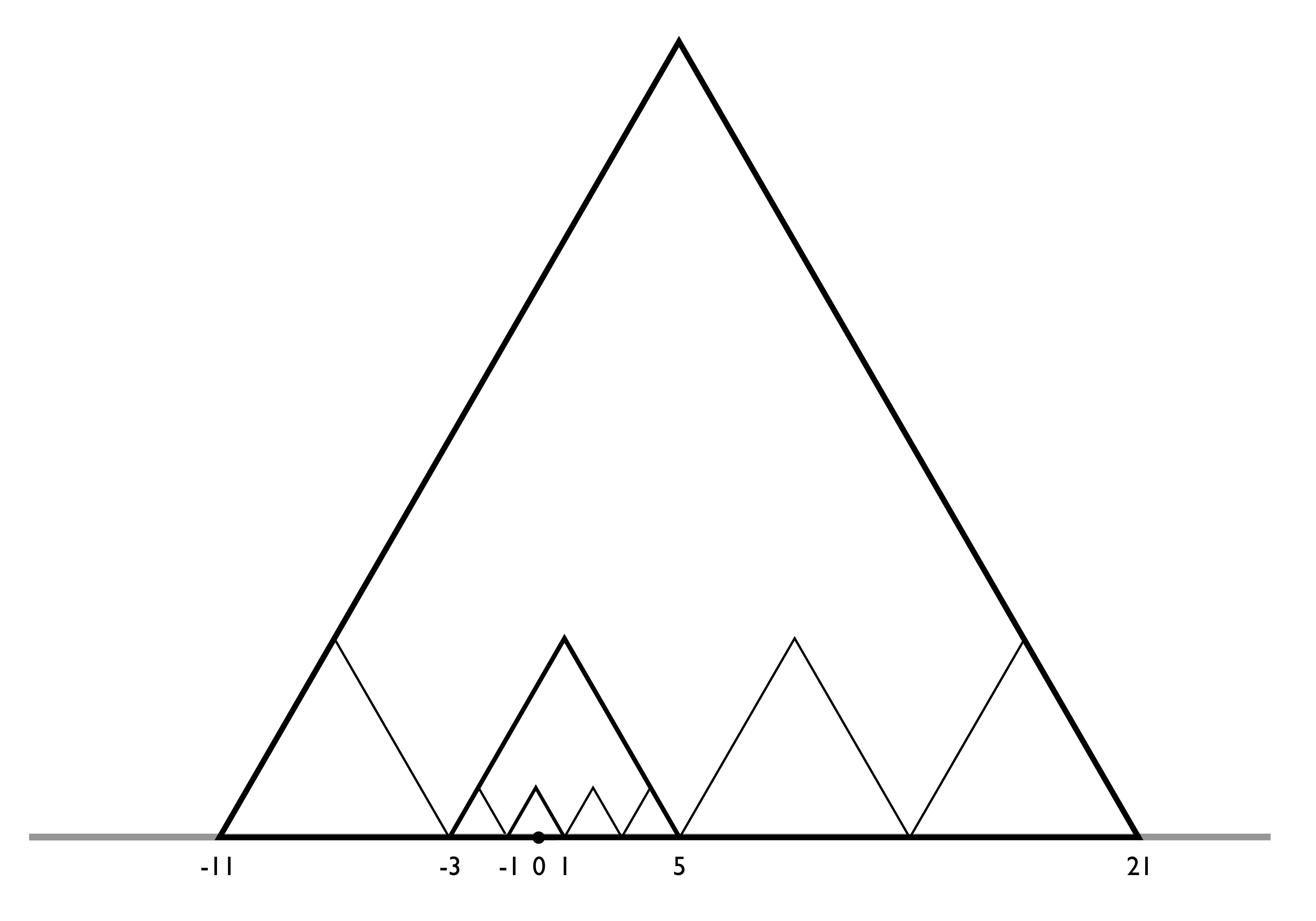}
\caption{This shows a sketch of how a tiling with an infinite $a$-line (the horizontal line)
can be constructed so that it is generic-w. Here ${\bf q} = za_1$, where 
$z$ is the $2$-adic integer $(1,1,5,5,21,21, 85, 85, \dots)$ (the $\mod 2,4,8, 16,32, 64, 128, 256 \dots$ values).
Some triangles of edge lengths $2,8,32$ are shown.
There are triangles of arbitrary large side lengths on the horizontal line, but the triangulation does not admit
a second infinite $a$-line and cannot admit
an infinite $w$-line since ${\bf q}$ is of the wrong form.}
\label{genericwNota}
\end{figure}

\begin{prop} \label{concurrent exceptional-a}
If a tiling $\Lambda \in X_Q$ has an infinite a-line then it is in $X_Q\backslash X_Q^{gen}$ 
and has either precisely one infinite
a-line or three infinite a-lines which are concurrent. The latter case occurs if and only if ${\bf q} \in Q$, where ${\bf q} = \xi(\Lambda)$.
\end{prop}
\noindent
{\bf Proof:} Let $\Lambda \in X_Q$ have an infinite a-line $l$. We already know that $\Lambda \in X_Q\backslash X_Q^{gen}$ and ${\bf q}  
= x + {\bf u} a$ for some $x\in Q$, ${\bf u} \in \overline{\ZZ_2}$, and some $a \in\{a_1,a_2, a_1+a_2\}$ from Lemma~\ref{exceptional-a}.
If it has a second (different) infinite line $l'$ then similarly ${\bf q}  = y + {\bf v} b$
where $y\in Q$, ${\bf v} \in \overline{\ZZ_2}$, and $b \in\{a_1,a_2, a_1+a_2\}$. Certainly $a \ne b$
for otherwise the two lines are parallel and this leads to overlapping triangles of arbitrarily large size, which
cannot happen. But we have $y-x  \in Q \cap (\overline{\ZZ_2}a +\overline{\ZZ_2}b) =
\ZZ a + \ZZ b$. Since $a,b$ are linearly independent over $\overline{\ZZ_2}$ and 
$y-x = {\bf u}a - {\bf v}b$ we see
that ${\bf u}$ and ${\bf v}$ are actually in $\ZZ$. Then ${\bf q} \in Q$. We will indicate this by
writing $q$ for $ {\bf q}$. 

In this case, since $q \equiv q_1 + \cdots + q_k \mod 2^kQ$ we find that $q$
is a vertex of a level $k$ triangle, for all $k$. Since this is true for all $k$, $q$ is a point through 
which infinite level
lines in all three directions $\{a_1,a_2, a_1+a_2\}$ pass. Thus the existence of two infinite
lines implies the existence of three concurrent lines.

In the other direction, if ${\bf q } \in Q$ then as we have just seen there will be three
concurrent infinite lines passing through it. \qed 

\smallskip
The case of a tiling with three concurrent infinite a-lines in Proposition~\ref{concurrent exceptional-a} 
is called a central hexagon tiling 
({\bf CHT} tiling) in \cite{ST}. We also refer to them as {\bf iCa-L} tilings. Edge shifting is not defined along these three lines, and we shall
see that we have the freedom to shift them arbitrarily to produce legal tilings. The tilings 
in which there is one infinite a-line are designated as {\bf ia-L} tilings.
\medskip

\begin{figure}
\centering
\includegraphics[width=8cm]{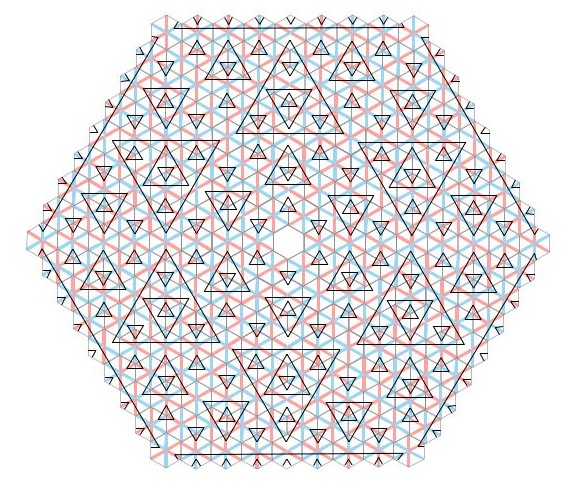}
\caption{The (central part of a) central hexagon ({\bf CHT}) tiling. Full (edge-shifted) triangles of 
levels $0,1,2$ are shown. At the outside edges one
can see the beginnings of triangles of level $3$. The rays from the central hexagon 
in the six $a$-directions will have infinite $a$-lines in them. However the edge shifting rules
cannot be applied to them because they are of infinite level -- they are not composed of edges of
finite triangles. In the end a full tiling is obtained by placing a fully decorated tile into the empty
central hexagon. There are $12$ ways to do this, and each way then determines the rest of the tiling
completely. These tilings violate both forms of generic condition.}
\label{CHT-tiling}
\end{figure}

\subsubsection{Violation of generic-w} The case of violation of generic-w is somewhat similar, though it takes more care. One
aspect of this is to avoid problems of $3$-torsion in $\overline P$, which we shall do by
staying inside $\overline Q$ where this problem does not occur. Thus in the discussion below
the quantity $3w$, where $w \in \{w_1,w_2, w_2 - w_1\}$, is of course in $Q$, but when
we see it with coefficients from $\overline {\ZZ_2}$ we shall understand it as being
in $\overline Q$ (as opposed to being in $\overline P$). Another problem is that the violation
of generic-w is not totally disjoint from the violation of generic-a, as we shall see. 

\begin{prop} \label{exceptional-w}
$\Lambda({\bf q})$ has a w-line of infinite level if and only if \, 
${\bf q} \in x + \overline{\ZZ_2}\,3w$
for some $w \in \{w_1,w_2, w_2 - w_1\}$ and some $x \in Q$. 
Furthermore when this happens the points of $Q$ lying on the infinite-level-line are those of the set
 $x + {\ZZ} \,3w$. 
\end{prop}
\noindent
{\bf Proof:} All w-lines deriving from the triangulation are necessarily in the directions 
$w \in \{\pm w_1, \pm w_2,\pm (w_2- w_1)\}$. Of course $3w \in Q$, and $mw \in Q$
iff $3|m$. It really makes
no difference which of the six choices $w$ is, but for convenience in presentation
we shall take herewith $w= w_2$ so that $3w= a_1 +2a_2$. This is in the vertical direction
in the plane. 

All w-lines contain centroids of levels up to the level of the line itself. Furthermore
if a w-line contains a centroid of level $k$ then it also contains one of the vertices of the
corresponding triangle and so also at least one point of $Q$ of level $k$.
It follows that for any w-line $l$ in the direction $w$ there is an $x\in Q$ so that the set of points
of $Q$ on $l$ is the set $x + \ZZ \,3w = x + \ZZ(a_1+2a_2)$.

Suppose that we have a w-line $l$ of infinite level and its intersection with 
$Q$ is $x + \ZZ(a_1+2a_2)$. Then the line $l$ has elements $y_1, y_2, \dots$ where $y_k$
is a vertex of a triangle of level $k$. This means that $y_1 \in q_1 + 2Q, y_2 \in q_1+ q_2 +4Q, \dots$.
We conclude that $\{y_k\} \to {\bf q}$. Furthermore $y_{k+1} - y_k \in 2^k Q \cap \ZZ (a_1+2a_2)$. 
This is true for all $k\ge 0$ if we define $y_0 = x$. 
Writing $y_{k+1} - y_k = 2^k u_k (a_1+2a_2)$ with $u_k \in \ZZ$ and ${\bf u} = (0, u_1, \dots, \sum_{j=1}^k u_j, \dots) \in \overline{\ZZ_2}$, we have
\[ y_{k+1} = x + \left(\sum_{j=0}^k u_j2^j\right)\,(a_1+2a_2) \rightarrow x + {\bf u}\,(a_1+2a_2).  \]
Thus ${\bf q} = x + {\bf u}\,(a_1+2a_2)$. This proves the 
only if part of the Lemma. 

Going in the reverse direction, if ${\bf q} = x + {\bf u}\, (a_1+2a_2) $ then this is a prescription for
a line of points in $Q$ that have vertices of all levels. The corresponding
w-line has centroids of unbounded levels, so the line is a w-line of infinite level. \qed

\begin{prop} \label{concurrent exceptional-w}
If a tiling $\Lambda \in X_Q$ has an infinite w-line then $\Lambda \in X_Q\backslash X_Q^{gen}$
and it has either precisely one infinite
w-line or three infinite w-lines which are concurrent. The latter case occurs if and only if 
the point of concurrency is either a point of infinite level (discussed in Prop.~\ref{concurrent
exceptional-a}) or a non-orientable point (discussed in Prop.~\ref{noOrientation}).
\end{prop}

\noindent
{\bf Proof:}
Let $\Lambda \in X_Q$ have an infinite w-line $l$.  We again take this to be in the direction
of $w_2$. Then ${\bf q} := \xi(\Lambda) 
= x + {\bf u}\,(a_1+2a_2) $ for some $x\in Q$, ${\bf u} \in \overline{\ZZ_2}$. 

Suppose that it has a second (different) infinite line $l'$. Then similarly ${\bf q} = y + {\bf v}\, 3w'$
where $y\in Q$, ${\bf v} \in \overline{\ZZ_2}$, and $3w'  \in \{2a_1+a_2, a_1+ 2a_2, a_2-a_1\}$. 
As above, we note
that $w \ne w'$ because if $w = w'$, the two lines are parallel and 
each of the two lines contains vertices of arbitrarily large levels. But the parallel lines through vertices and centroids
of level $k$ are spaced at a distance of $2^{k-1}$ apart. Thus no two distinct parallel w-lines can both
carry centroids of arbitrary level. 

Again, for concreteness we shall take a specific choice for $w'$, namely $w'= w_1= 2a_1 +a_2$.
Other choices lead to similar results.

There are two scenarios. Either the two lines $l,l'$ meet at a point of $Q$ or not. Suppose
that they meet in a point of $Q$. Then we can choose $x=y$ and obtain 
\[x + {\bf u}\,(a_1+2a_2) = {\bf q} = x + {\bf v}\, (2a_1+a_2) \,.\]
Since $a_1$ and $a_2$ are independent in $\overline Q$ over $\overline{\ZZ_2}$
we obtain ${\bf u} = 2{\bf v}$ and ${2\bf u} = {\bf v}$. The only solution to this in 
$\overline{\ZZ_2}$ is  ${\bf u} = {\bf v} =0$. Thus ${\bf q} =x \in Q$. This puts us
in the situation of Prop.~\ref{concurrent exceptional-a}, the point of intersection
of the two lines is actually a vertex of infinite level, and this is a {\bf CHT} tiling.

The alternative is that $l,l'$ meet at a point $p$ of $P\backslash Q$. In this case we go back
to the discussion of coloring given in \S\ref{color}. The point $p$ is a centroid and it
either has infinite
level, in which case it has no orientation and we go to Prop.~\ref{noOrientation}, or it has a finite level in which two of the three w-lines through it have forced color and finite level which is a contradiction. This proves
the result. \qed 

Infinite level a-lines occur if and only if ${\bf q} \in x +\overline{\ZZ_2}a $ and
infinite level w-lines occur if and only if ${\bf q} \in x +\overline{\ZZ_2}w $, with $a,w$ being in the
basic $a$ and $w$ directions respectively. Three concurrent a-lines occur if and only ifs ${\bf q} \in Q$,
whereupon the condition for three concurrent w-lines also is true. These are the 
{\bf CHT} tilings. 

\renewcommand{\arraystretch}{1.5}
\begin{table}[htdp]
\caption{Summary of infinite level $a$-lines and $w$-lines.}
\begin{center}
\begin{tabular}{|c||c|c|}\hline
type & single & three concurrent\\
\hline\hline
infinite $a$-line & ${\bf q} \in Q+\overline{\ZZ_2} a$ & ${\bf q} \in Q$ \;{\bf CHT}\\
\hline
infinite $w$-line & ${\bf q} \in Q+\overline{\ZZ_2} w$ & ${\bf q} \in Q$ \;{\bf CHT}\\
& & ${\bf q} \in -{\bf s_2} - 2{\bf s_1}+Q$\, or \, ${\bf q} \in -{\bf s_1} - 2{\bf s_2}+Q$\\
\hline
\end{tabular}
\end{center}
\label{infiniteLevelLines}
\end{table}

Since the singular elements of $\overline Q$ lie on a countable union of lines, it is clear
that their total measure is $0$.

\begin{lemma}\label{singularMeasure0}
The set of singular ${\bf q} \in Q $ has Haar measure $0$. 
\end{lemma}

\begin{lemma} \label{infinite a and w lines} 
If a triangularization ${\mathcal T}({\bf q})$
has both an infinite level $a$-line and an infinite level $w$-line then their point of intersection
is a point of concurrence of three infinite level w-lines
and three infinite level a-lines, and the tiling is a {\bf CHT} tiling.
\end{lemma}
\noindent
{\bf Proof:} By Prop.~\ref{exceptional-a} and Prop.~\ref{exceptional-w},
\[{\bf q} \in x_1 + \overline{\ZZ_2} a \quad \mbox{and} \quad {\bf q} \in x_2 + \overline{\ZZ_2} 3w\]
for some $x_1,x_2 \in Q$ and $a\in \{a_1,a_2,a_1+a_2\}$, $w\in \{w_1,w_2,w_2-w_1\}$.
Putting these together, 
\[{\bf q} = x_1 + {\bf z_1} a = x_2 + {\bf z_2} 3w\]
for some ${\bf z_1,z_2} \in \overline{\ZZ_2}$. However $a$ and $3w$ are independent elements of 
$Q$ (over $\mathbb Z$), and hence are also independent over $\overline{\ZZ_2}$. Since
$x_2-x_1 \in Q$, this forces ${\bf z_1,z_2} \in \mathbb Z$. Thus $ {\bf q} \in Q$, which is the condition
for simultaneous concurrency of three a-lines and  three w-lines (Prop.~\ref{concurrent exceptional-a}). \qed

\medskip

\subsection{Coloring for the {\bf iCw-L} tilings}\  \label{Coloring for the exceptional tilings}
According to Prop.~\ref{noOrientation} we have a point of no orientation
precisely when ${\bf q} \in -{\bf s_2} - 2{\bf s_1} + Q$ or ${\bf q} \in -{\bf s_1} - 2{\bf s_2} + Q$.
In these cases, by Prop.~\ref{concurrent exceptional-w}, we have three
concurrent w-lines and their intersection is a point of no orientation. This point of intersection
is $x = {\bf q} +w_1+{\bf s_2} + 2{\bf s_1}$ or $x = {\bf q} +w_2+ {\bf s_1} + 2{\bf s_2}$. The former
can be anywhere in $w_1 +Q$ and the latter anywhere in $w_2+Q$. The triangulation can be
described as a set of nested triangles of levels $0,1,2,3,\dots$ (and all the lesser level
triangles that occur within them) all of which have the centroid $x$. The level $k=0$ triangle
is an up triangle in the $w_1$ case and a down triangle in the $w_2$ case. The infinite 
$l$-lines are in the directions $w_1, w_2, w_2-w_1$ through $x$ and these three lines
have no forced colorings. 

We call these tilings the {\bf iCw-L} tilings (infinite concurrent w-line tilings). We also refer
to the underlying triangulations with the same terminology.
See Figure \ref{iCw-L-tiling}.
\begin{figure}
\centering
\includegraphics[width=8cm]{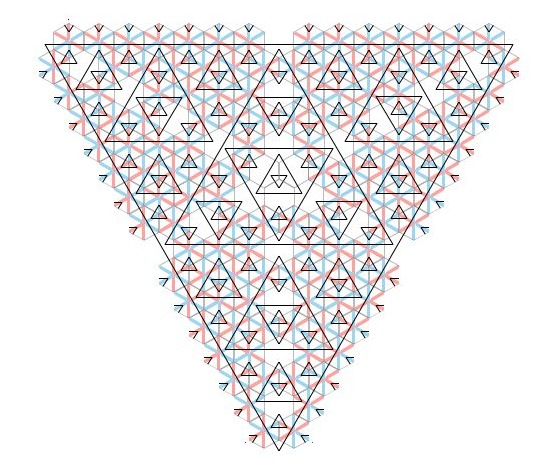}
\caption{{The \bf iCw-L} tilings.  The triangulation is generic-a but not generic-w. Of course the partial tiling shown is perfectly
consistent with generic tilings -- in fact all Taylor--Socolar tilings  contain this type of patch of tiles. 
However, if the pattern established in the picture is maintained at all scales, 
then indeed the result is a not a generic tiling since it fails generic-$w$.}
\label{iCw-L-tiling}
\end{figure}

The symmetry
belongs to the triangulation, not necessarily to the tilings themselves. The colorings of the 
three exceptional
lines of an {\bf iCw-L} tiling can be made in an arbitrary way without violating the tiling conditions
{\bf R1, R2} \,\cite{ST}. Of the $8$ possible colorings the two truly symmetric ones (the ones that give
an overall $3$-fold rotational symmetry -- including color symmetry -- to the actual tiling) are
exceptional in the sense that no other tilings in the Taylor--Socolar system 
have a point $p \in P\backslash Q$ (i.e a tile vertex) with the property that the three hexagon diagonals
emanating from it are all of the same color (see Prop.~\ref{threeColorThm}). These exceptional
symmetric {\bf iCw-L} tilings are called {\bf SiCw-L} tilings. In \cite{ST} these tilings are described as having
a `defect' at this point, and indeed they are not LI to any other tilings except other {\bf SiCw-L} tilings. 

Thus there are $2$ exceptional colorings for any {\bf iCw-L} triangulation. In the other $6$ colorings
there are at each hexagon vertex two diameters of the same color and one of the opposite
color, and we shall soon prove that they all occur in $X_Q$. 

Tilings for which there is just one infinite w-line in the triangulation are called {\bf iw-L} tilings.

\subsection{The structure of the hull}\ \label{HullStructure}
In this subsection we describe the hull $X_Q$ in more detail. We note that the only
symmetries of $X_Q$ which we discuss are translational symmetries (not rotational). These
translational symmetries are the elements of $Q$. Of course none of the elements of
$X_Q$ has any non-trivial translational symmetry; it is only the hull itself that has them. 
When we discuss LI classes below we mean local indistinguishability classes under translational symmetry. 

\begin{theorem} \label{minimal}
$X_Q$ consists of three LI classes, $X_Q^{b}$, $X_Q^{r}$, and $X_Q^\dagger$. 
Of these $X_Q^{b}$ 
is the countable set of {\bf SiCw-L} tilings with three blue-red (blue first) diameters emanating from some
hexagon vertex $q$, which form 
a single $Q$-orbit in $X_Q$, and $X_Q^{r}$ is the companion orbit with red-blue diameters. 
Both of these orbits are dense in $X_Q$. 

$X_Q^\dagger$ is the orbit closure of $X_Q^{gen}$ and contains all other tilings, including
all the {\bf iCw-L} tilings that are not color symmetric. 
Restricted to the minimal hull $X_Q^{\dagger}$, the mapping $\xi$ defined 
in \eqref{surjective map from X_Q to overlineQ} is:
\begin{itemize}
\item[(i)] $1:1$ on $X_Q^{gen}$;
\item[(ii)] $6:1$ at {\bf iCw-L} points except {\bf SiCw-L} points
\item[(iii)] $12:1$ at {\bf CHT} points;
\item[(iv)] $2:1$ at all other non-generic points. 
\end{itemize}    
\end{theorem} 

\begin{remark} 
The images of $\xi$ of the set of singular points (non-generic points) is dense in $\overline Q$.  For instance, the triangulations with three concurrent a-lines are parameterized by $Q$ which is a dense subset of $\overline Q$, and these tilings produce the {\bf CHT} tilings (or {\bf iCa-L} tilings) described above.
Both $X_Q^{gen}$ and $X_Q^\dagger$ are of full measure in 
$X_Q$.
\end{remark}

\noindent {\bf Proof:} First, we consider generic tilings. Let $\Lambda$, where $ \xi(\Lambda) = {\bf q}$,
be any generic tiling and let
$B_R$ be the ball of radius $R$ centered on $0$. Let $\CalT(\Lambda)$ be the triangulation determined
by $\Lambda$ (with edges not displaced) and let $\CalT_R(\Lambda)$ be the part of the triangulation
that is determined by $B_R$.  

Because we are in a generic situation, to know how to shift an edge of level $k$ we need only 
that edge to appear
as an inner edge of a triangle of level $k+1$. To determine the coloring of a w-line
we need to know its level (which is finite).  So to know all this information for $\CalT_R$
we need only choose $r$ large enough so that $B_r$ contains all the appropriate triangles.

Now if generic $\Lambda'$, where $\xi({\Lambda'}) = {\bf q'}$ produces the same pattern of triangles in $B_r$ then it is 
indistinguishable from $\Lambda$ in $B_R$. In particular if  ${\bf q'}$ satisfies ${\bf q'} -{\bf q} \in 2^kQ$ for large enough $k$ then $\Lambda$ and $\Lambda'$
must agree (as tilings) on $B_R$. This proves that convergence of ${\bf q'}$ to ${\bf q}$ produces corresponding convergence in $X_Q$.

With this it is easy to see that any two generic elements of $X$ are LI. Let $\xi(\Lambda)
= {\bf q}$ and $\xi(\Lambda')= {\bf q'}$ be generic. Let ${\bf q}$ correspond to $q_1, q_2, \dots$
and ${\bf q'}$ correspond to $q'_1, q'_2, \dots$. Then we can construct the tiling sequence
$q'_1 -q_1 +\Lambda, q'_1 +q'_2 -(q_1+q_2) +\Lambda,\dots $ and it converges to $\Lambda'$.

This same argument can be used to show that the orbit closure of any tiling contains all of $X_Q^{gen}$. 
Let $\Lambda$ be any tiling with ${\bf q}=\xi(\Lambda)$ and $\Lambda' $ be a generic tiling with ${\bf q'}=\xi(\Lambda')$.
Then one simply
forms a sequence of translates of $\Lambda$ that change ${\bf q}$ into ${\bf q}'$.
The convergence of the triangulation on increasing sized patches forces convergence of
the color and we see $\Lambda'$ in the orbit closure of $\Lambda$.

\medskip

Second, we consider the {\bf iCw-L} cases, where $x:= {\bf q}+ w_2 + {\bf s_1} + 2 {\bf s_2} \in w_2 + Q$ or
$x: = {\bf q} +w_1+ {\bf s_2} + 2 {\bf s_1} \in w_1+ Q$. In these cases $x$ is a non-orientable point
and there exists a nested sequence of triangles of all levels centered on $x$. This sequence
begins either with an up triangle of level $0$ or a down triangle of level $0$. In either case
everything about the triangulation is known and the entire tiling is determined except for the 
coloring of the three w-lines through $x$. In fact all of the $8$ potential colorings 
of these three lines are realizable as tilings, as we shall soon see. 

Of these {\bf iCw-L} triangulations we have the {\bf SiCw-L} tilings in which the
colors of the diagonals of the three hexagons of which $x$ is a vertex start off the same --
all red or all blue. This arrangement at a hexagon vertex never arises in a generic tilings,
and it is for this reason that these tilings produce different LI classes than the one
that the generic tilings lie in: one `red' LI class and one `blue' LI class. As pointed out in \cite{ST} these 
{\bf SiCw-L} tilings have the  amazing property that they are completely determined once the three hexagons around $x$ have been decided\footnote{It is also pointed out in \cite{BG} that the {\bf SiCw-L} tilings do not arise in the substitution
tiling process originally put forward in Taylor's paper. However, they do arise as legal tilings from the
matching rule perspective, though they could be trivially removed by adding in a third rule to forbid
them. A similar situation has been shown to occur with the Robinson tilings for which there is a matching rule
and also a substitution scheme that result in a hull and its minimal component \cite{GJS}. As pointed out in \cite{ST}, this is different from tilings like the Penrose rhombic tiling where the matching
rules determine the minimal hull.}. The form of the points $x$ with no orientation shows that there are just
two $Q$ orbits of them, one for each of the two non-trivial cosets of $Q$ in $P$. 

What about the other $6$ color arrangements around such a point $x$? Here we can argue that they
all exist in the following way. Since in any triangulation there are tile centroids of any desired level $k$,
we can start with any generic $\Lambda$ and form a sequence of translates of it that have 
centroids of ever increasing level at $0$. The sequence has at least one limit point 
and this is an {\bf iCw-L} tiling. 
Since each element of the sequence has a unique coloring and coloring in generic tiles
is locally determined by local conditions, there must be a subsequence of these tilings that
converges to one of some particular coloring. This must produce a coloring of diameters
with two diameters of one color and one of the other color since we are using only generic
tilings in the sequence. Now the rotational three-fold symmetry and the color symmetry of
$X_Q$ shows that all $6$ possibilities for the coloring will exist. This also shows that 
all these tilings are in the orbit closure of $X_Q^{gen}$.

Third, the {\bf CHT/iCa-L} triangulations have the form $\CalT(q)$ where $q\in Q$. They have three concurrent a-lines and three concurrent w-lines at $q$
and leave the central tile completely undetermined. This tile can be placed in anyway we wish, and
this fixes the entire tiling. There are a total of $12$ ways to place this missing tile ($6$ for each parity),
whence $\xi$ is $12:1$ over $q$. 

Finally, apart from the {\bf iCw-L} and {\bf CHT/iCa-L} tilings, the remaining singular values of {\bf q} correspond to the 
{\bf ia-L} and {\bf iw-L} triangulations where there is either a single infinite level $a$-line or a 
single infinite level $w$-line, 
\S\ref{Coloring for the exceptional tilings}.  Fortunately these two things cannot happen
at the same time, see Lemma~\ref{infinite a and w lines}. That means that there
is only one line open to question and there is only one line on which either the shift
or color is not determined.
In the {\bf ia-L} case there is an $a$-line for which edge shifting is un-defined and we wish to show
that all the two potentially available shifts lead to valid tilings. Likewise in the {\bf iw-L} case there is a
$w$-line to which no color can be assigned, and we wish to prove that both coloring options are viable.

Suppose one starts with a {\bf CHT} tiling $\Lambda$ centered at $0$. If one forms
a sequence $\{q_1+ \cdots + q_k + \Lambda\}$ of {\bf CHT} tilings 
and if $\{q_1+ \cdots + q_k\} $ converges to a point of on the line $\overline{\ZZ_2}a_1$ that is not
in $Q$ then the point of {\bf CHT} concurrence has vanished and one is left only with the $x$-axis as an single infinite level a-line, and it will have the shifting
induced by the original shifting along the $x$-axis in $\Lambda$ (which can be 
either of the two potential possibilities). Of course one can do this
in any of the $a$ directions. 
A similar type of procedure works to produce all of the {\bf iw-L} tilings.
This concludes the proof of the theorem. \qed

\section{Tilings as model sets}\label{model sets}

In this section we consider Taylor--Socolar tilings, and in particular the parity sets of 
such tilings, from the point of view of model sets. There are a number of advantages
to establishing that point sets are model sets since there
is a very extensive theory for them, including fundamental theorems regarding their intricate
relationship to their autocorrelation
measures and their pure
point diffractiveness \cite{MS, BM, BLM}. In fact there are various ways in which one can establish that the vertices
or the tile centres of a Taylor--Socolar tiling always form a model set. We have already pointed this out
in Cor.~\ref{ms1}. However, the set-up that we have created makes it easy to see the model set construction
rather explicitly, and that is the purpose of this section. 

The basic pre-requisite for the cut and project formalism is a
cut and project scheme. Most often, especially in mathematical physics, the cut and project
schemes have real spaces (i.e. spaces of the form $\RR^n$) as embedding spaces and internal spaces. But the theory of model sets is really part of the theory of locally compact Abelian groups 
\cite{RVM}.
In the case of limit-periodic sets, some sort of ``adic'' space is the natural ingredient for the internal space.
In our case the internal space is $\overline P$, see \cite{LM1}.

\subsection{The cut and project scheme}\label{modelset} \
Form the direct product of $\RR^2$ and $\overline P$. The subset
$\CalP =  \{(x,i(x))\in \RR^2 \times \overline P : x\in P \}$
is a lattice in $ \RR^2 \times \overline P$ (that is, $\CalP$ is discrete and
$((\RR^2 \times \overline P)/ \CalP$ is compact) with the properties that the projection mappings 
\begin{equation} \label{cpScheme}
   \begin{array}{ccccc}
      \RR^2 & \stackrel{\pi_{1}}{\longleftarrow} & 
        \RR^2 \times \overline P & \stackrel{\pi_{2}} 
      {\longrightarrow} & \overline P   \\ 
      &&  \cup \\  
      P &\stackrel{\simeq}{\longleftrightarrow} & \CalP   
   \end{array} 
\end{equation}
satisfy  $\pi_1|_{\CalP}$ is injective and $\pi_2(\CalP)$ is dense in $\overline P$.
The set up of \eqref{cpScheme} is called a {\bf cut and project scheme}. 

Then $P = \pi_1 (\CalP) \subset \RR^2$ and the ``star mapping'' 
$(\cdot)^\star \! : \, P \longrightarrow \overline P$
defined by $\pi_2 \circ (\pi_1|_{\CalP})^{-1}$ is none other than the embedding $i$
defined above. 

Let $W\subset \overline P$ which satisfies $W^\circ \subset W \subset \overline{W^\circ} = \overline W$ 
with $\overline W$ compact,
we define 
\[\curlywedge(W) := \{ x\in P \,:\, x^\star \in W\} \,.\]
This is the {\bf model set} defined by the {\bf window} $W$.
Most often we wish to have the additional condition that the boundary $\partial W :=\overline W\backslash W^\circ$
of $W$ has Haar measure $0$ in $\overline P$. 
In this case we call $\curlywedge(W)$ a {\bf regular model set}.

As an illustration of how the cut and project scheme is used to define the model sets,
we give here a model set interpretation for the sets $W^\uparrow(\{\bf q\})$ and
$W^\downarrow(\{\bf q\})$ of Prop.~\ref{Wupdownarrows}:

\begin{eqnarray}\label{triangleCenters}
\curlywedge^\uparrow(\{\bf q\}) &:=& \{ x \in P : x^\star \in W^\uparrow(\{\bf q\})\}\\
\curlywedge^\downarrow(\{\bf q\}) &:=& \{ x \in P : x^\star \in W^\downarrow(\{\bf q\}) \}\,. \nonumber
\end{eqnarray}

The windows $W^\uparrow({\bf q})$ and $W^\downarrow({\bf q})$ are compact and the closures
of their interiors, so these two sets are pure point diffractive model sets, and clearly
they are basically the points of $P\backslash Q$ which have orientation up and down respectively.
In the case of values of ${\bf q}$ treated in Prop.~\ref{noOrientation} there will be one point without
orientation. It is on the common boundary of $W^\uparrow({\bf q})$ and $W^\downarrow({\bf q})$.

However, our intention here is not to interpret features of the triangulation in terms of model
sets (which is more or less obvious) but to understand parity, which is a more subtle feature depending
on edge-shifting and color, in terms of model sets. In this paper we will need only to deal with
model sets lying in $Q$, and for this it is useful to restrict the cut and project scheme above to the 
lattice
$\CalQ =  \{(x,i(x))\in \RR^2 \times \overline Q : x\in Q \}$
in $ \RR^2 \times \overline Q$ :

\begin{equation} \label{cpSchemeQ}
   \begin{array}{ccccc}
      \RR^2 & \stackrel{\pi_{1}}{\longleftarrow} & 
        \RR^2 \times \overline Q & \stackrel{\pi_{2}} 
      {\longrightarrow} & \overline Q   \\ 
      &&  \cup \\  
      Q &\stackrel{\simeq}{\longleftrightarrow} & \CalQ \,.  
   \end{array} 
\end{equation}
Of course we shall not be looking for just one window
and one model set, but rather two windows, one for each of the two choices of parity. 

\medskip
\subsection{Parity in terms of model sets: the generic case} \label{modelset-generic}\
Each tiling in $X_Q$ is composed of hexagons centered at points of $Q$ that
are of one of the two types shown in Fig.~\ref{two-basic-tiles}. We call them white 
or gray according to the coloring shown in the figure. At the beginning we shall
work only with the generic cases, since for them the tiling is completely represented
by its value in $\overline Q$.

Let $\Lambda$ be a generic tiling for which $\xi(\Lambda) = {\bf q} \in \overline{Q}$. 
We define $Q({\bf q})^{wh}$ (resp. $Q({\bf q})^{gr}$) to be the set of points of $Q$ whose corresponding tiles in $\Lambda$ are white (resp. gray), so we have a partition 
\[ Q = Q({\bf q})^{wh}  \cup Q({\bf q})^{gr} \,.\]

We shall show that each of $Q({\bf q})^{wh}$ and $Q({\bf q})^{gr}$ is a union of  a countable number of $2^kQ$-cosets (for various $k$) of $Q$. 
If this is so then since the 
closure of a coset $x + 2^kQ$ is $x + 2^k\overline{Q}$
which is clopen in $\overline Q$, we see that $\overline{Q({\bf q})^{wh}}$ contains the open
set $U^{wh}$ consisting of the union of all the clopen sets coming from the closures of the
cosets of $Q({\bf q})^{wh} $,
and $\overline{Q({\bf q})^{wh}}$ is the closure of $U^{wh}$. 
Similarly $\overline{Q({\bf q})^{gr}}$ contains an open set $U^{gr}$. We note that 
$U^{wh}$ and $U^{gr}$  are disjoint
since they are the unions of disjoint cosets, and their union contains all of $Q$.  

We also point out that $U^{wh}$ is the interior of 
$\overline{Q({\bf q})^{wh}}$ since any open set in $\overline Q$ is a union of clopen sets
of the form $x + 2^k\overline{Q}$ with $x\in Q$ (they are a basis for the topology of $\overline{Q}$)
and each of these is either in $U^{wh}$ or $U^{gr}$. But no point of $U^{gr}$ is a limit point of $U^{wh}$
and so $U^{gr} \cap \overline{Q({\bf q})^{wh}} = \emptyset$. Similarly
$U^{gr}$ is the interior of $\overline{Q({\bf q})^{gr}}$.

Evidently $\overline{Q({\bf q})^{wh}} \cap \overline{Q({\bf q})^{gr}}$ is a closed set with no interior,
since $U^{wh}$ and $U^{gr}$ are disjoint.
Thus $\overline{Q({\bf q})^{wh}} \cap \overline{Q({\bf q})^{gr}}$ lies in the boundaries of
each set and contains no points of $Q$. Each of the sets $\overline{Q({\bf q})^{wh}}$
and $\overline{Q({\bf q})^{gr}}$ is compact and each is the closure of its interior. The boundaries
of the two sets are both of measure $0$ since $U^{wh}$ and $U^{gr}$ can account for the full
measure of $\overline Q$. Finally, $\overline Q = \curlywedge(\overline{Q({\bf q})^{wh}}) \cup \curlywedge(\overline{Q({\bf q})^{gr}}) $.

\begin{theorem}\label{whiteGreyModelSets}
Let $\Lambda$ be a generic tiling for which $\xi(\Lambda) = {\bf q} \in \overline{Q}$, where 
${\bf q} = (q_1, q_1+q_2, \dots,$ $ q_1+ \cdots + q_k, \dots) \in \overline{Q}$. 
We have the 
model-set decomposition for white and gray points of the hexagon centers of $\Lambda$:
\begin{eqnarray}\label{genericModelSetDecomp}
Q({\bf q})^{wh} &=& \curlywedge(\overline{Q({\bf q})^{wh}})\,, \\ \nonumber
Q({\bf q})^{gr} &=& \curlywedge(\overline{Q({\bf q})^{gr}})\, ,\\ \nonumber
\overline Q &=& \curlywedge(\overline{Q({\bf q})^{wh}}) \cup \curlywedge(\overline{Q({\bf q})^{gr}}) \, ,\nonumber
\end{eqnarray}
where $\overline{Q({\bf q})^{wh}}$ is the closure of the union of the clopen sets in $\overline{Q}$. 
Thus these sets are regular model sets.
\end{theorem}

\noindent {\bf Proof:} 
We have to show
that $Q({\bf q})^{wh}$ and $Q({\bf q})^{gr}$ are each unions of  a countable number of $2^kQ$-cosets 
(for various $k$) of $Q$. 
There are two components that
enter into the white/gray coloring: the diameter coloring of the tiles and the edge shifting. The
generic condition guarantees that both coloring and shifting are completely unambiguous. Our argument deals
with coloring first, and shifting second. Finally both parts are brought together.

Let $\Omega:= \{\pm w_1, \pm w_2, \pm(w_2-w_1)\}$ and
$\Omega^+:= \{w_1, w_2, w_2-w_1\}$.
For $w\in \Omega$ and $k=1,2, \dots$, define
\[ U_k := \bigcup_{w\in \Omega} q_1 + \cdots q_k +2^k Q + \ZZ 3w\,. \]
Recall that for $w\in\Omega$, $3w \in Q\backslash 2Q$.  The points
of $q_1 + \cdots q_k +2^k Q$ are the vertices of the level $k$ triangles
and the sets $U_k$ are composed of the points of $Q$ on the w-lines that pass through
such vertices. 
We have $Q = U_1 \supset U_2 \supset U_3 \supset \cdots$.

A point of $Q$ may be a vertex of many levels of triangles, but we wish to 
look at the highest level vertex that lies on a given w-line. Thus we define 
$V_k := U_k\backslash U_{k+1}$, $k=1,2,\dots $. The sets $V_k$ are mutually
disjoint. See Fig.~\ref{V_i-forModelSet-J}.
\begin{figure}
\centering
\includegraphics[width=8cm]{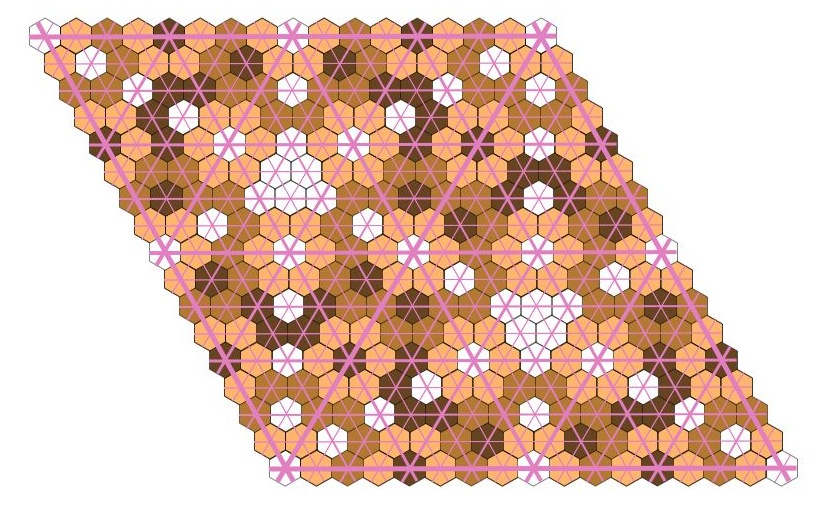}
\caption{The tiles associated with points of $V_1,V_2,V_3$ are indicated by increasingly
dark shades.}
\label{V_i-forModelSet-J}
\end{figure}

\smallskip
The sets $U_k$ are made up of various elements $q_1 + \cdots q_k +2^k u + 3nw$
where $u\in Q$ and $n \in \ZZ$. However we can restrict $n$ in the range
$0\le n \le 2^{k-1}$ since $3\, 2^k w \in 2^k Q$ and 
\[ q_1 + \cdots q_k +2^k u + 3 \,2^{k-1} w = q_1 + \cdots q_k +2^k u + 3 \, 2^{k } w +
3 \,2^{k-1} (-w)
\, ,\]
which changes $w$ to $-w$ at the expense of a translation in $2^k Q$. 
To make things unique we shall assume that $w \in \Omega^+$ in the extreme
cases when $n=0$ or $n= 2^{k-1}$. 

\begin{figure}
\centering
\includegraphics[width=8cm]{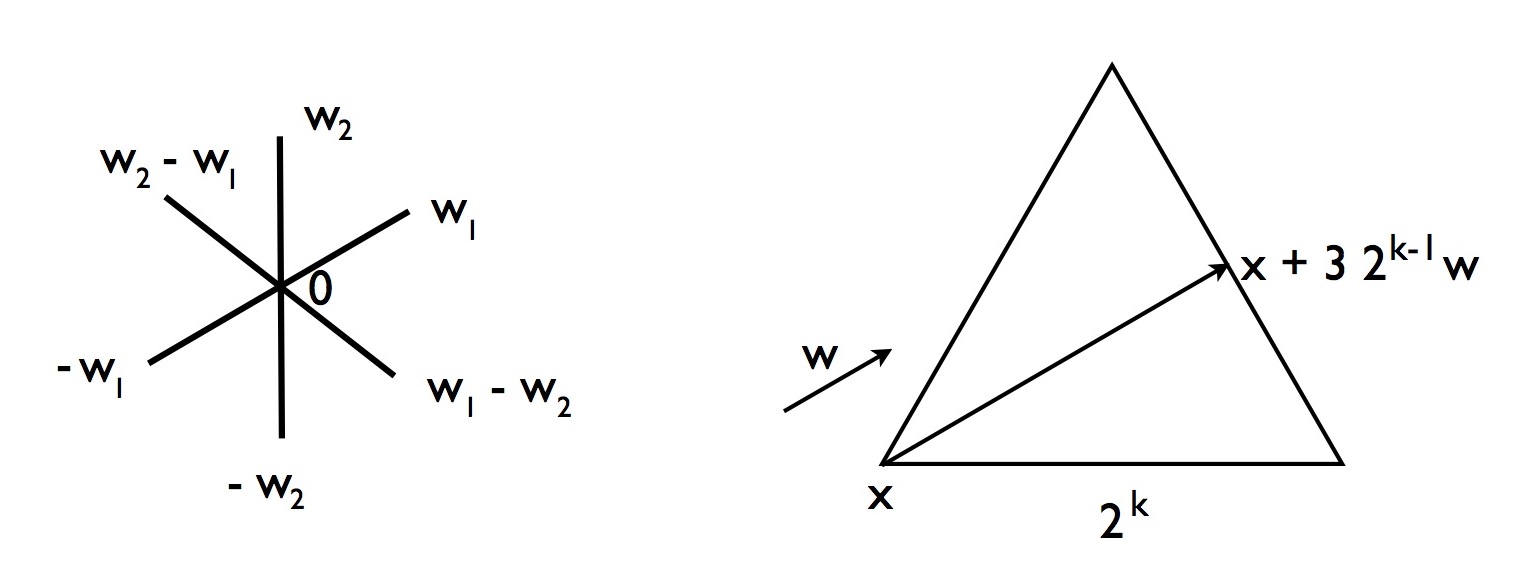}
\caption{$\Omega$ and a line through a vertex $x$ in the direction $w$
meeting the opposite edge at $3\, 2^{k-1} w$.}
\label{Omega}
\end{figure}

\begin{figure}
\centering
\includegraphics[width=8cm]{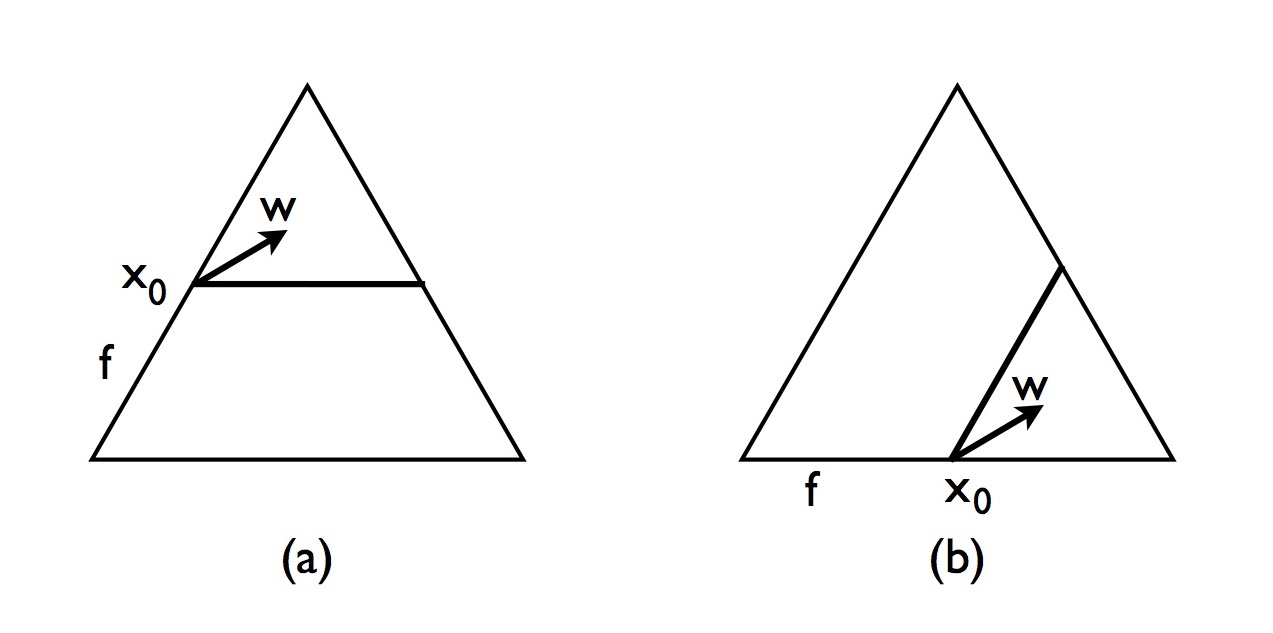}
\caption{Showing $x_0$ as a midpoint of an edge $f$ of
a triangle of level $2^{k+1}$ and the direction of $w$ from it. }
\label{VkFigure}
\end{figure}

We claim that under the condition generic-w we have $Q= \bigcup_{k=1}^\infty V_k$. The only way
that $x \in Q$ can fail to be in some $V_k$ is that $x \in U_k$ for all $k$.  Then $x$ is
on w-lines through vertices of arbitrarily high level triangles. At least
one $w\in\Omega$ occurs infinitely often in this. Fix such a $w$. The vertex of 
a level $k$ triangle is always  the vertex of $6$ such triangles around that
vertex. So whenever the w-line
passes through a vertex of a level $k$ triangle it also passes into the interior of one
of the level $k$ triangles of which this is a vertex and then through the centroid of this  
level $k$ triangle. Thus the w-line $x + \ZZ w$ has centroids of arbitrary level on it, violating
generic-w. 

Let $x = q_1 + \cdots q_k +2^k u + 3 n w = x_0 + 3nw \in V_k$ for some $k$, where $n$
satisfies our conventions noted above on the values it may take. Then by the definition of
$V_k$, $x_0$ is not the vertex of any edge of a triangle $T$ of side length $2^{k+1}$ and
so $x_0$ is the mid-point of an edge $f$ of such a $T$. In particular $x_0 \notin V_k$.
Fig.~\ref{Omega} and Fig.~\ref{VkFigure} indicate,
up to orientation, what all this looks like. The edge $f$ is on the highest level line through $x_0$
and so determines the stripe of the hexagon at $x_0$ and, more importantly, the coloring
of the w-line that we are studying. The coloring at $x_0$ in the direction $w$ starts
red in the case of Fig.~\ref{VkFigure}(a) and blue in the case of Fig.~\ref{VkFigure}(b). 
The color then alternates along the line in the manner illustrated in Fig.~\ref{coloring2}.

The color pattern determined here repeats modulo $2^{k+1}Q$ (not $2^k Q$), so 
$V_k$ splits into subsets, each of which is a union of cosets of $2^{k+1}Q$ in $Q$,
\[V_k = V_k^r  \cup V_k^b  = \bigcup_{w\in \Omega} \bigcup_{n =0}^{2^k -1} V^r_k(w,n) \cup V^b_k(w,n)\, ,\]
corresponding to the red-blue configurations 
and corresponding also to which $w \in \Omega$ is involved. Here $V^r_k(w,n)$ is the set of
points $q_1 + \cdots q_k +2^k u + 3nw \in V_k$ which {\em start} in the direction $w$ with the color
red, where $0 \le n < 2^{k-1}$ with the boundary conditions on $n$ established
as above. The situation with $V^b_k(w,n)$ is the same
except red is replaced by blue. 
Notice that each $V^r_k(w,n)$(or $V^b_k(w,n)$) is a union of $2^{k+1}Q$-cosets for various values of $k$.

\medskip
Now we need to look at the other aspect to determining the white and gray tiles, namely edge shifting.
For this we assume the condition generic-a. Every $x\in Q$ lies on an edge of level $k$ for some $k$.
Recall that this means that $x$ is on the edge of a triangle of level $k$ but not one of level $k+1$.
The condition generic-a says that such an edge must exist for $x$. Such an edge
must then appear as the inner edge $e$ of a corner triangle of a triangle $T$ of level $k+1$. The
edge $e$ then shifts towards the centroid of $T$, say in the direction $w\in \Omega$, carrying
corresponding diameters in along with it.

\begin{figure}
\centering
\includegraphics[width=8cm]{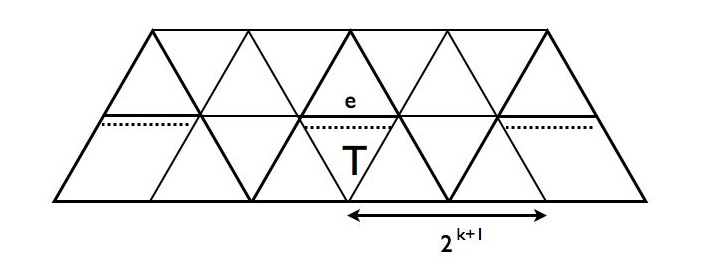}
\caption{Showing a level $k$ edge inside a triangle $T$ of level $k+1$
and its corresponding shift (dotted line). Note how edge shifting for edges of level $k$ repeats modulo $2^{k+1}$. }
\label{LkShifts-J}
\end{figure}

Let
\[ L_l(w) := \{ x \in Q \, :\, x \,\mbox{lies on an edge of level $l$ which shifts in the direction
$w$ }\} \,. \]
Then $Q = \bigcup_{l=1}^\infty \bigcup_{w\in \Omega} L_l(w)$. As one sees from Fig.~\ref{LkShifts-J},
$L_l(w) $ is a union of $2^{l+1}Q$-cosets (but not $2^lQ$-cosets). 

We now put these two types of information together. We display the results in the form
of two tables: any $x\in Q$ satisfies
\begin{equation}\label{colorAndShift}
x \in V_k^c(w,n)  \cap L_l(w')  
\end{equation}
for some $c\in\{r,b\}$, $k,l \ge 1$, $w \in \Omega$, $w'\in \Omega$. 

Notice that $V_k^c(w,n)  \cap L_l(w')$ is a union of $2^mQ$-cosets for $m=\max\{k+1,l+1\}$.
So we finally obtain that each of $Q^{wh}$ and $Q^{gr}$ is the union of such cosets. This is what we wanted
to show and concludes the proof of the theorem.

\subsection{Parity in terms of model sets: the non-generic case}\
For non-generic sets there are two situations to consider. First of all, let us consider
tilings of the minimal hull. Any such tiling $\Lambda'$ can be viewed as the limit of translates
of a generic tiling, $\Lambda$.  Let $\xi(\Lambda) = {\bf q}$ and let $W^{wh}$ and $W^{gr}$ denote
the two closed windows that define the parity point sets of the tile centers of $\Lambda$. Translation of
$t+ \Lambda$, $t\in Q$,  amounts to translation by $t^\star$ of $W^{wh}$ and $W^{gr}$. 
This is in fact just translation by $t$ but with $t$ seen as an element of $\overline Q$. Translation
does not affect the type of the tiling ({\bf iw-L, iCa-L}, etc.). 

Convergence of a sequence of translates
$t_1+\Lambda, t_1+t_2 +\Lambda, \dots $ to $\Lambda'$ in the hull topology implies $Q$-adic convergence of
$t_1+t_2 + \cdots$, say to ${\bf t} \in \overline Q$. The translated sets then also converge to  
${\bf t} +W^{c}$, $c= \{wh, gr\}$. However, if ${\bf t} \notin Q$ then we will not necessarily have
$\Lambda' = \curlywedge({\bf t} +W^{c})$. 

Here is what happens. If $u \in \Lambda'$ then for large enough $n$, $u \in t_1+t_2 + \dots +t_n + \Lambda $
and $u^\star \in t_1+t_2 + \dots +t_n + W^{c}$. Thus $u^\star \in {\bf t} +W^{c}$ and we have
that $\Lambda' \subset \curlywedge({\bf t} +W^{c})$.  On the other hand we have
$\Lambda' \supset \curlywedge({\bf t} +(W^{c})^\circ)$. For suppose that 
$x\in \curlywedge({\bf t} +(W^{c})^\circ))$.
Then $x^* \in ({\bf t} +(W^{c})^\circ)\cap Q^*$ and 
using the convergence $t_1+t_2 + \dots \to {\bf t}$ 
we see that for large $n$, $x^* \in (t_1+t_2 + \dots + t_n  +(W^{c})^\circ))\cap Q^*$. Thus for large $n$,
$x\in \curlywedge(t_1+t_2 + \dots + t_n  + W^{c}) = t_1+t_2 + \dots + t_n  +\Lambda$, and so $x\in \Lambda'$.
We conclude that $\Lambda' = \curlywedge(Z)$ for some window $Z$
satisfying ${\bf t} + (W^{c})^\circ \subset Z \subset {\bf t} + W^{c}$. This shows that
$\Lambda'$ is a model set since $Z$ lies between its interior and the (compact) closure
of its interior. Also $\partial Z \subset \partial({\bf t} + W^{c})$ has Haar measure $0$ as it was explained in the beginning of Section~\ref{modelset-generic}.
\medskip

The remaining cases are the {\bf SiCw-L} tilings. 
Let $\Lambda$ be such a tiling, which we may assume to be associated with $\xi(\Lambda)=0 \in Q$. Comparing the {\bf SiCw-L} tiling $\Lambda$ with an {\bf iCw-L} tiling $\Gamma$ for which $\xi(\Gamma) = 0$,
we notice that the only difference between $\Lambda$ and $\Gamma$ is on the lines through $0$ in the $w$-directions where $w \in \Omega$. The total index is introduced in \cite{LM1}. 
Notice that it is enough to compute that the total index of the set of all points off these $w$-lines is $1$ (see cite[LM1]). Because the set of points off the lines of $w$-directions is the disjoint union of cosets $V_k$ (we have seen this earlier),
we only need to show that the total index of $\cup_{1}^{\infty} V_k$ is $1$, i.e.
\[ \sum_{k=1}^{\infty} c(V_k) = 1 .\]
Following the construction of $V_k$, $k \ge 1$, already discussed above, we compute the coset index of $V_k$. Within each $V_k$ we need to divide the point set $V_k$ into two sets. One is the point set whose points are completely within the $(k+1)$-th level triangles and the other is the point set whose points are lying on the lines of the $(k+1)$-th level triangles.
We note that 
\begin{eqnarray*}
c(V_1) & = & \frac{6}{4 \cdot 2^2} \\
 c(V_2) & = & \frac{(2^1 -1) \cdot 2\cdot 3 \cdot 2}{4 \cdot 2^2 \cdot 2^2} + \frac{6}{4 \cdot 2^2 \cdot 2^2} \\
 c(V_3) & = & \frac{(2^2 -1)\cdot 2 \cdot 3 \cdot 2}{4 \cdot 2^2 \cdot 2^2 \cdot 2^2} + \frac{6}{4 \cdot 2^2 \cdot 2^2 \cdot 2^2} \\
 && \vdots \\
 c(V_k) & = & \frac{(2^{k-1} -1)\cdot 2\cdot 3 \cdot 2}{4 \cdot (2^2)^k} + \frac{6}{4 \cdot (2^2)^k} .
\end{eqnarray*}
So 
\begin{eqnarray*} 
 \sum_{k=1}^{\infty} c(V_k)  &=& \frac{6}{4} \left( \frac{1}{2^2} + \frac{1}{2^2 \cdot 2^2} + \cdots +   \frac{1}{(2^2)^k} \cdots \right) + 3 \left( \frac{2}{(2^2)^2} + \frac{2^2}{(2^2)^3} + \cdots + \frac{2^{k-1}}{(2^2)^k} + \cdots \right) \\
 && ~~~~~ - 3 \left( \frac{1}{(2^2)^2} + \frac{1}{(2^2)^3} + \cdots +  \frac{1}{(2^2)^k} + \cdots \right) \\
&=& \frac{6}{4}\left( \frac{1/4}{1 - 1/4}\right) + 3 \left( \frac{1/8}{1 - 1/2} \right) - 3 \left( \frac{1/16}{1-1/4}\right) \\
&=& 1.
\end{eqnarray*}

\section{A formula for the parity}\label{Formula}

In this section we develop formulae for tilings which determine the parity of each tile of a tiling
from the coordinates of the center of that tile. 
We begin with the formula for parity derived in \cite{ST} for the 
{\bf CHT} tilings centered at $(0,0)$. These correspond to the triangulation for 
${\bf q} =0$. The parity formula for a tile is based on the coordinates of the center
of the hexagonal tile. Due to the non-uniqueness of the {\bf CHT} tilings along the $6$-rays
at angles $2\pi k/6$ emanating from the origin, the basic formula is valid only off these rays. 
Later we show how to adapt this formula to arbitrary ${\bf q}$.

The parity of a tile depends on the relationship of its main stripe
to the diameter at right-angles to this stripe. In terms of the triangulation, the parity
of a tile depends on two things: the way the triangle edge on which the stripe is
located is shifted and the order of the two colors of the color line as it passes
through the tile: red-blue or blue-red. Changing the shift or the color order changes
the parity, changing them both retains the parity. Thus the parity can be expressed as the modulo $2$ sum
of two binary, i.e.  $\{0,1\}$, variables representing the shift  and the color order. Which parity belongs
to which type of tile is an arbitrary decision.  In our case we shall make it so that the white tile
has parity $1$ and the gray tile has parity $0$.

We introduce here a special coordinate system for the plane (which we really only use
for elements of $Q$). We take three axes through         
$(0,0)$, in the directions of $a_1, a_2, -(a_1+a_2)$ with these three vectors as unit vectors
along each. For convenience we define $a_3:= -(a_1 +a_2)$. Each $x\in Q$ is given the coordinates 
$(x_1,x_2,x_3)$ where $x_1$ is the $a_2$
coordinate where the line parallel to $a_1$ through $x$ meets the $a_2$-axis. Similarly
$x_2$ is the $a_3$
coordinate where the line parallel to $a_2$ through $x$ meets the $a_3$-axis, and
$x_3$ is the $a_1$
coordinate where the line parallel to $a_3$ through $x$ meets the $a_1$-axis. This is shown in
Fig.~\ref{coordinates}. Notice that $x_1 +x_2 +x_3 =0$. We call these coordinates the {\bf triple
coordinates}.

The redundant three label coordinate system that we use has the 
advantage that one can just cycle around the coordinates to deal with each of the three $w$-directions.
Counterclockwise rotation through $2\pi/3$ amounts to replacing $(x_1,x_2,x_3)$
by $(x_3,x_1,x_2)$.

We note that $x\in Q$ if and only if $x_1,x_2,x_3 \in \ZZ$. Let $\nu:\ZZ\longrightarrow \ZZ$
be the $2$-adic valuation defined by $\nu(z) = k$ if $2^k|| z$, i.e. if $2^k$ divides $z$ but $2^{k+1}$ does not divide $z$.
We define $\nu(0) = \infty$. Finally we define $D(z) = 2^{\nu(z)}$. Note that $D(-z)=D(z)$.
When levels appear, they are related to $\log_2(D(z)$. 

\smallskip
Now for $x\in Q$ we note that, except for $x=0$, exactly two of $D(x_1), D(x_2), D(x_3)$
are equal and the remaining one is larger. This is a consequence of $x_1 +x_2 +x_3 =0$.
\begin{figure}
\centering
\includegraphics[width=6cm]{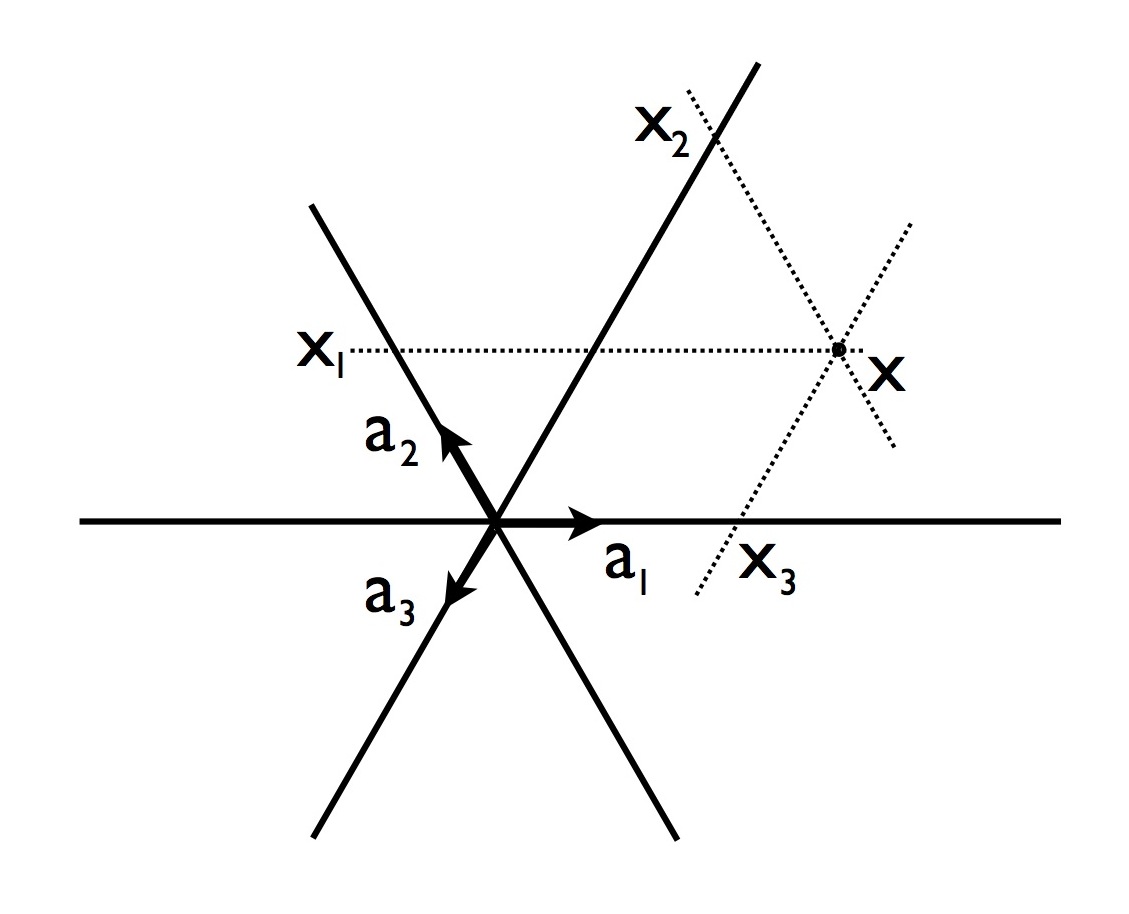}
\caption{The three coordinate system of labelling points in the plane.}
\label{coordinates}
\end{figure}

\subsection{{\bf CHT} formula}

In this section we derive the formula for parity for the {\bf CHT} tiling \cite{ST}. 
The  {\bf CHT} tiling has the advantage that all the shifting due to the choice of the 
triangulation is taken out of the way, and this makes it easier to see what is going on.
Our notation and use of coordinates
is different from that in \cite{ST}, but the argument is essentially the same.

\begin{figure}
\centering
\includegraphics[width=7cm]{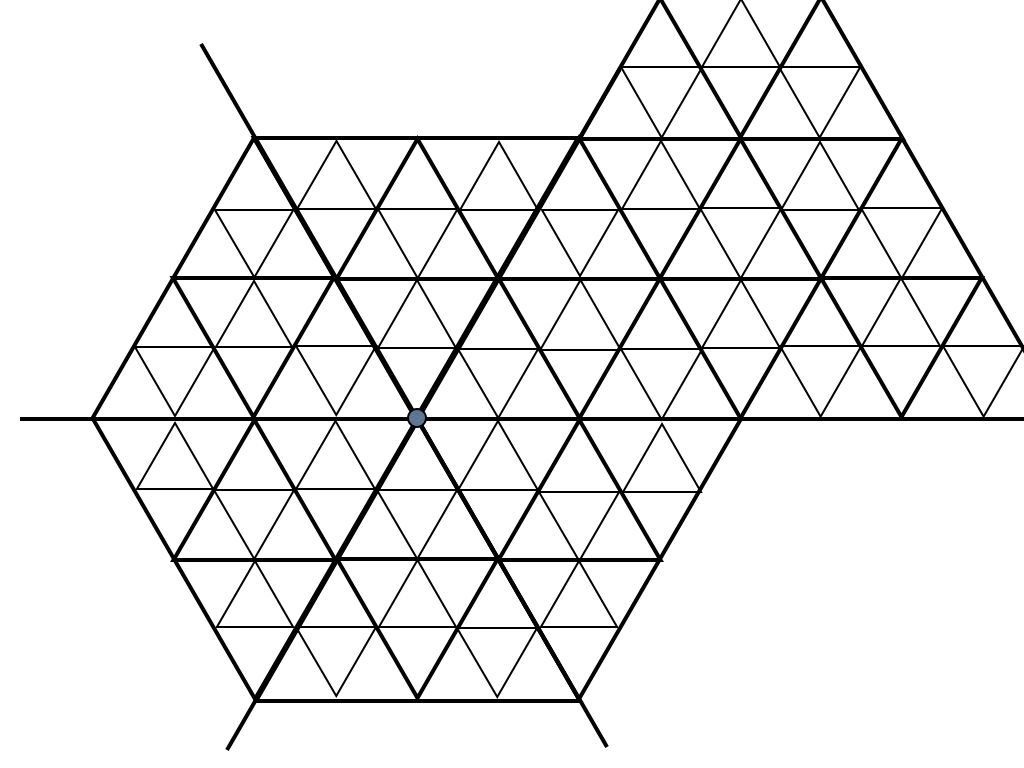}
\caption{Part of the triangulation corresponding to ${\bf q}$ centered at $0$ (indicated by the 
dot). Triangles of scales
$1,2,4$ are shown, as well as part of a triangle of scale $8$. }
\label{CHTSetUp}
\end{figure}
Fig.~\ref{CHTSetUp} shows how the {\bf CHT} triangulation looks around its center $(0,0)$.
The formula for parity is made of two parts each of which corresponds to one the two features which combine
to make parity: edge shifting and the color. 

First consider the shifting part of the formula. We consider a horizontal line of the  {\bf CHT} triangulation,
different from the $a_1$-axis. This line meets the $a_3$-axis at a point $(n2^k,-n2^k,0)$ for
some non-negative integer $k$ and some odd integer $n$. This point is the apex of a level
$k$ triangle and is the midpoint of an edge from a triangle of level $k+1$ (though the $a_3$-axis
itself is of infinite level here). As such we see that the horizontal edge to the right from 
$(n2^k,-n2^k,0)$ is shifted downwards. As the edge passes into the next level $k+1$ triangle we see
that the shift is upwards. This down-up pattern extends indefinitely both to the right and to the left. 
In Fig.~\ref{HorizontalParity} $n=1$ and $k$ is unspecified, but the underlying idea
does not depend on the value of $n$. We now note that the points
along the horizontal edge rightwards from $(2^k,-2^k,0)$ are $(2^k,-2^k -1,1), (2^k,-2^k-2,2), \dots$,
or $x= (2^k,-2^k - x_3,x_3)$ in general. Now we note that
\begin{equation}\label{shiftParity}
 \left \lfloor \frac{x_3}{D(x_3+x_2)} \right \rfloor  =
\begin{cases} 0 \quad \mbox{if $x$ corresponds to a shift down edge,}\\
1 \quad \mbox{if $x$ corresponds to a shift up edge\,.}
\end{cases}
\end{equation}

\begin{figure}
\centering
\includegraphics[width=8cm]{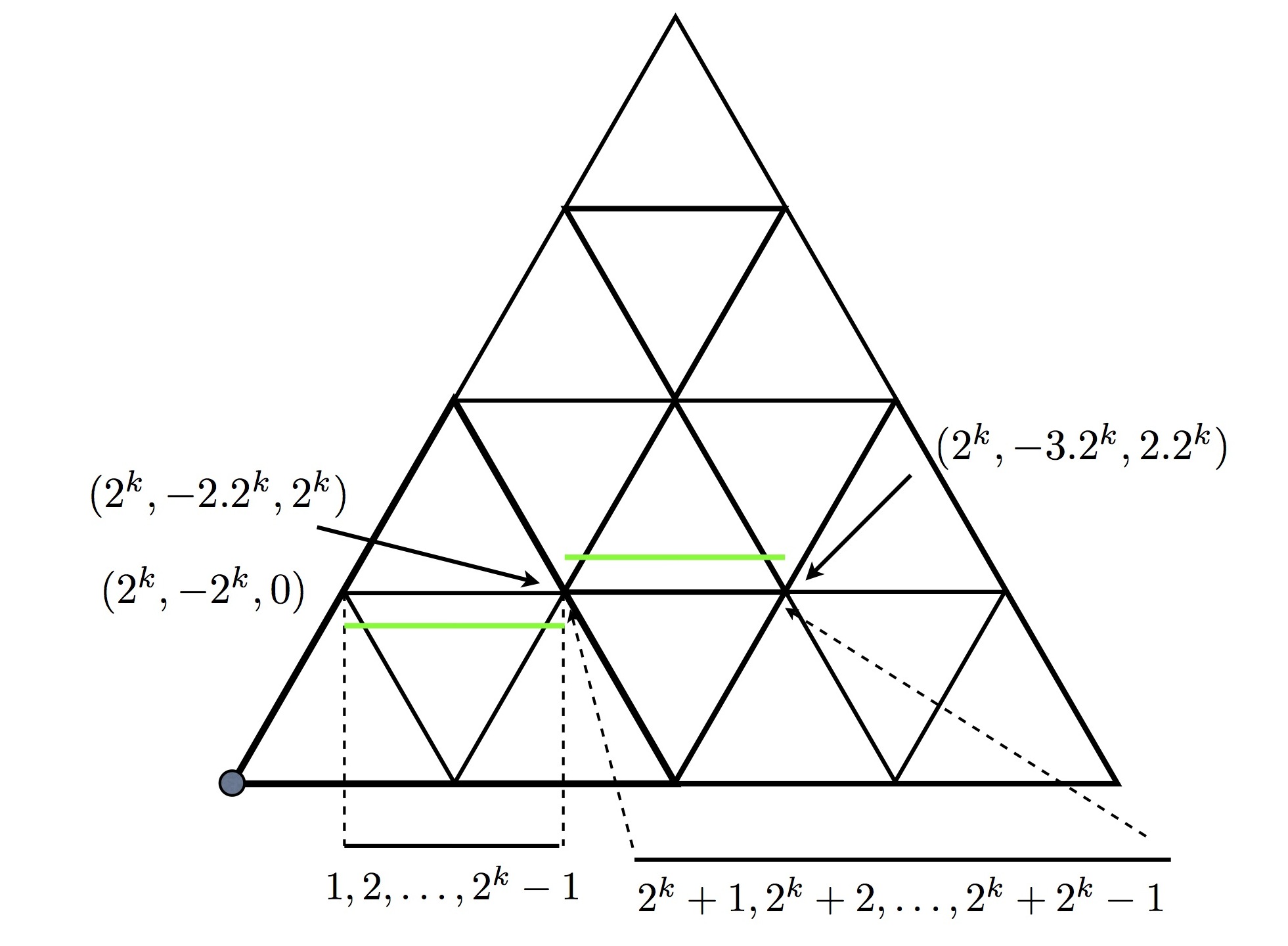}
\caption{The horizontal line through the point $(2^k,-2^k,0)$ is seen as passing through
midpoints of consecutive $2^{k+1}$ triangles. The corresponding edges along this
line shift downwards and upwards alternately. The points $(2^k,-2^k-m,m)$ with 
$m=1,2,\dots, 2^k-1$ are on a shift-down edge. The next set for $m= 2^k+1, 
\dots,2^k + 2^k -1$ are on a shift-up edge. The next set, $m= 2. 2^k+1, \dots 2.2^k + 2^k-1$
are on a shift-down edge, etc. The formula $\lfloor \frac{x_3}{D(x_3+x_2)}\rfloor
= \lfloor m/(2^k)\rfloor$ accounts
for this precisely, varying between $0$ and $1$ according to down and up. }
\label{HorizontalParity}
\end{figure}

Thus this is the formula for edge downwards ($0$) and edge upwards ($1$).
This formula is not valid if $x_3 \equiv 0 \mod 2^k$. What distinguishes these bad values is that 
for these, and these only, $D(x_2) \neq D(x_3)$. We see that the fact that we
are dealing with a horizontal line (in the direction of the $a_1$ axis) is related to the fact
that $D(x_1)$ is the largest of $\{D(x_1),D(x_2),D(x_3)\}$, and whenever that condition fails
the above formula fails. But then of course we should use the appropriate formula with
the indices cycled.

\medskip

Next we explain the color component of the formula. The underlying idea is much the same,
but, as to be expected, the details are a little more complicated. The color lines are the w-lines and 
are oriented in one of the three $w$ directions. We treat here the case of color lines that
are in the vertical direction. The formula utilizes the same three coordinate formulation above.
For other w-directions one cycles the three components around appropriately.

Consider the sector of the  {\bf CHT} tiling as indicated in Fig.~\ref{CHTVertical1}. The figure indicates
how the color must be on the $a_3$-axis as we proceed in the vertical direction.

\begin{figure}
\centering
\includegraphics[width=8cm]{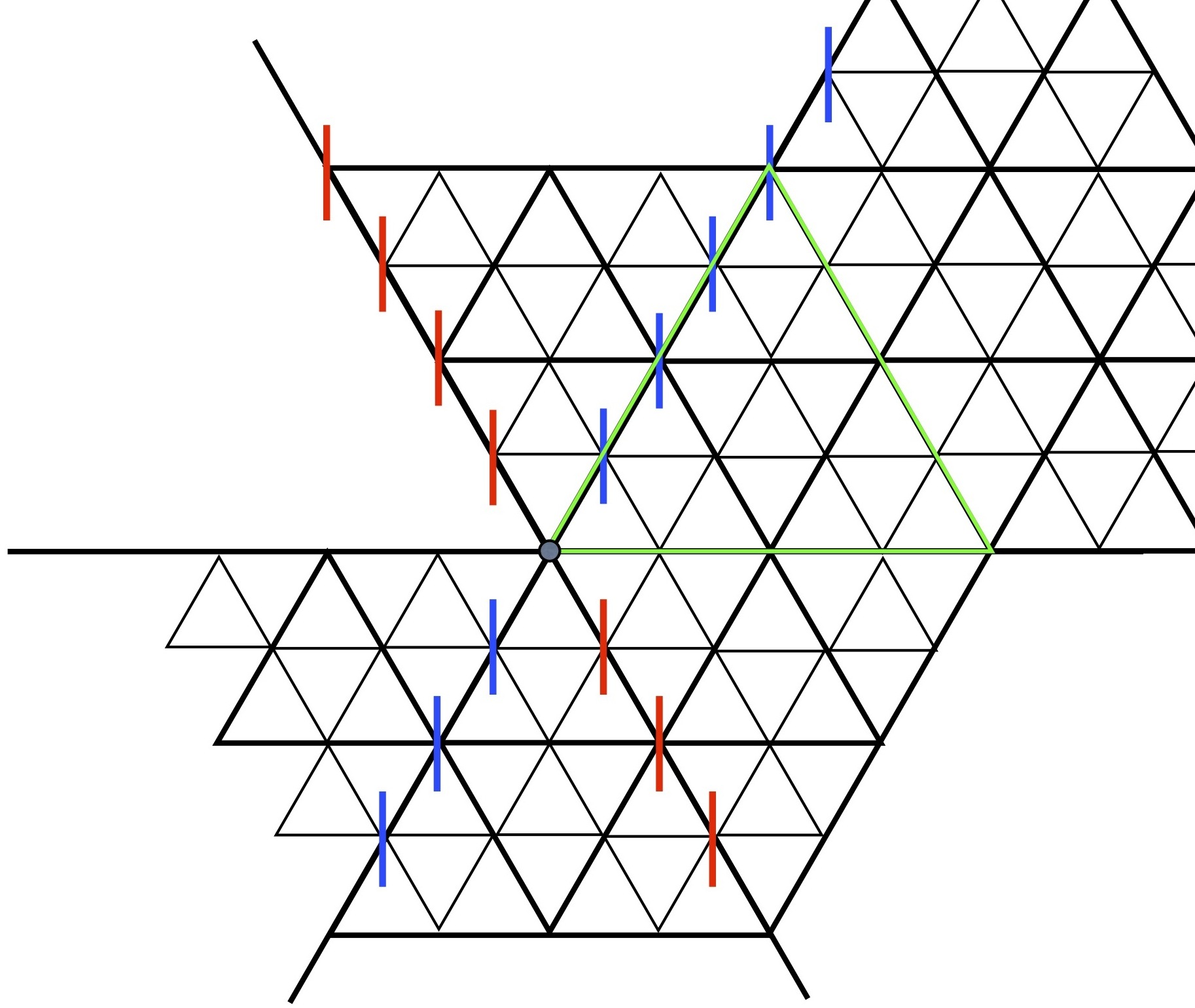}
\caption{Vertical color lines start with a full blue diameter or full red diameter as shown. The green triangle
indicates why the second blue line segment up from the origin is in fact blue. The vertical line
is centered at the mid-point of an edge of a level $4$ triangle and 
passes into one of the level $2$ corner triangles of this level $4$ triangle. The discussion on
color shows that it must be blue. All the other blue line segments are explained in the same
way. We cannot assign color at the origin itself since the origin is not the mid-point of any edge. }
\label{CHTVertical1}
\end{figure}

\begin{figure}
\centering
\includegraphics[width=8cm]{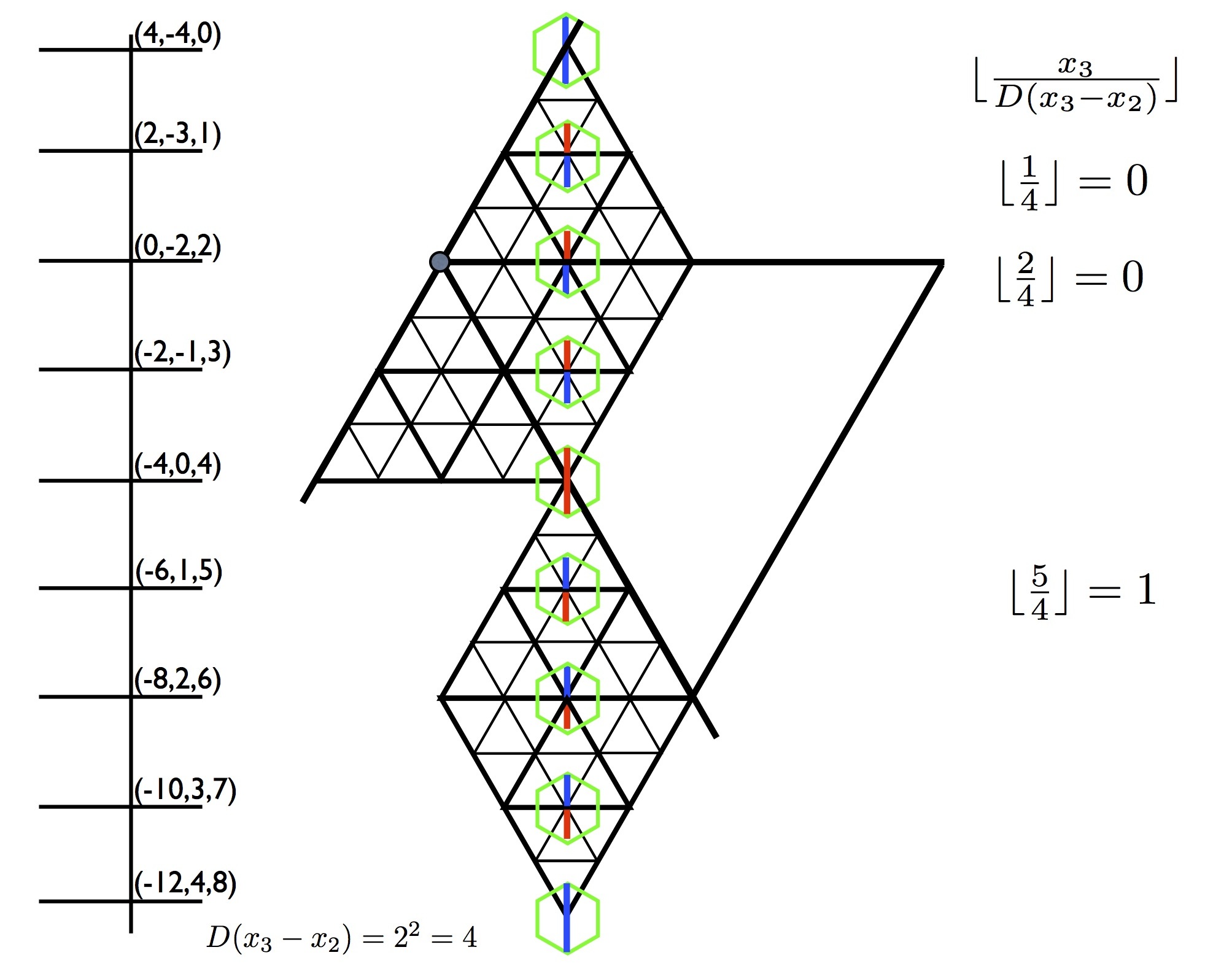}
\caption{In the figure the small circle indicates $(0,0)$. A vertical color line is shown, which meets 
the $a_3$ axis in the point $p= (4,-4,0)$. 
The points on this axis are all of the form $(u,-u,0)$, $u\in \ZZ$. In the  {\bf CHT} tiling set-up, 
D(u) indicates that the point is a vertex of a level $\log_2 D(u)$ triangle, in this case
$D(4)=4$, so we are on a level $2$ triangle ($2^2=4$). We already saw that the color starts
with a full blue diagonal at $(4,-4,0)$. Moving down the line to the next point decreases $x_1$ by $2$ and
increases $x_2,x_3$ by $1$. We note that $x_3-x_2$ remains constant, and $D(x_3 -x_2) = 4$.
At the second step we cross, at right-angles into another triangle of level $2$, 
and the color proceeds without interruption. At the $4$th step we are at a 
vertex that is the midpoint of the side of a triangle of level $3$ and looking up our vertical line
we can see that it is passing into a corner triangle of level $4$ -- namely where we just
came from, and we see a forced full red diagonal. In the first three steps the diameters
are all red-blue (top to bottom), whereas after the red-red diagonal the next three
steps are blue-red diameters, so the is a switch that affects parity.  We can ignore the points
with full diameters (they get sorted out in a different w-direction). We note
that $\lfloor\frac{x_3}{D(x_3-x_2)}\rfloor$ maintains the value $0$ on steps $1,2,3$ and maintains
the value $1$ on steps $5,6,7$, showing that the formula notices the change of diameters
correctly. If we continue $\lfloor\frac{x_3}{D(x_3-x_2)}\rfloor =2$ on the next three step sequence,
but modulo $2$ this is the same as $0$. }
\label{CHTVertical2}
\end{figure}

Most of the explanation for the color part of the formula appears in the caption to 
Fig.~\ref{CHTVertical2}. Although that picture seems tied to the point of intersection
of the vertical line and the $a_3$-axis having the special form $(4,-4,0)$ we note that
the same applies whenever $D(x_3-x_2) =4$. It is $D(x_3-x_2)$ that determines the level
of the triangle that we are looking at and thus how the stepping sequence will modify the
hexagon diameters. In the case where it is $4$ here there are $3$-step sequences of one diagonal type
followed by $3$-step sequences of the other type. If $D(x_3-x_2) =2^k$ these become
$2^k-1$ step sequences, and still $\lfloor\frac{x_3}{D(x_3-x_2)}\rfloor$ changes by $1$ each time
we move from one $2^k-1$ step sequence to the next.  All sequences start from the
full blue diameter with $x_3 =0$ and $x_3$ increases by $1$ at each step.

In putting the two formulas together, we note first of all that although the formulas have been
derived along specific $a$ and $w$ axes, the formulas remain unchanged if the same configurations
are rotated through an angle of $ \pm 2 \pi/3$. Likewise the coloring and shifting rules depend on the
geometry and not the orientation modulo $\pm 2 \pi/3$. The final formula is then effectively just the
sum of the two formulas that we have derived, and it is only a question of determining which color of tile
belongs to parity $0$. 

\begin{theorem}\cite{ST} \label{parityThm}
In the  {\bf CHT} tilings centered at $(0,0)$, the parity of a hexagonal tile
centered on $x= (x_1,x_2,x_3)$ is 
\[ \PP(x) = \left \lfloor \frac{x_{j+2}}{D(x_{j+2}-x_{j+1})}\right \rfloor +  
\left \lfloor \frac{x_{j+2}}{D(x_{j+2} +x_{j+1})}\right \rfloor 
\mod 2  \,,\]
provided that $D(x_j)$ is the maximum of $D(x_j), D(x_{j+1}), D(x_{j+2})$ and 
$x_{j+1} \pm x_{j+2} \neq 0$ {\rm (}subscripts $j$ are taken modulo $3${\rm )}. 
\qed
\end{theorem}

{\bf Proof}: Referring to Fig.~\ref{CHTVertical2}, we check the parity of the tile at $(2,-3,1)$.
In this case $j$ of the theorem is $1$ and the displayed formula gives the value $0$. On the other
hand, the edge shift is down at $(2,-3,1)$ and the shifted edge meets the blue part of the hexagonal
diameter, whence the tile is gray. This establishes the parity formula everywhere. \qed

\begin{remark} Recall that in the paper we have the convention that white corresponds to $1$ and gray corresponds to $0$.
\end{remark}

\begin{remark}
Notice that in the  {\bf CHT} triangulation centered at $(0,0)$ the hexagon diameters 
along the three axes defined by $a_1,a_2,a_3$ 
have no shift forced upon them and can be shifted independently
either way to get legal tilings. These are the hexagons centered on the 
points excluded by the condition $x_{j+1} + x_{j+2} \neq 0$. Similarly the
three w-lines through the origin have no coloring pattern forced upon them
and can independently take either. The centers of the hexagons that lie
on these lines are excluded by the condition  $x_{j+1} - x_{j+2} \neq 0$.
In the  {\bf CHT} tiling, the central tile can be taken to be either of the two hexagons and in any
of its six orientations. Having chosen one of these $12$ options for the
central tile the rest of the missing information for tiles is automatically
completed.
The parity function $\PP$ can be then extended to a function $\PP^e$ so as to take
the appropriate parity values on the $6$ lines that we have just described. 
\end{remark}

\subsection{Parity for other tilings}
We can create a formula for arbitrary triangulations $\CalT({\bf q})$ by the following 
argument. First of all consider what happens if we shift the center of our triple coordinate
system to some new point $c=(c_1,c_2,c_3)_0 \in Q$ (we are in the triple
coordinate system centered at the origin and have indicated this with the subscript
$0$). Then relative to the new center $c$, still using axes in the directions $a_1,a_2,a_3$,
the triple coordinates of $x= (x_1,x_2,x_3)_0$ are $(y_1,y_2,y_3)_c = (x_1-c_1,x_2-c_2,x_3-c_3)$.
Thus, if consider ${\bf q} = c$ (or more properly ${\bf q} = (c,c,c, \dots)$) so we are looking at
the  {\bf CHT} triangulation now centered at $c$, then the formulae above become
\[ \PP_c(x) = \left \lfloor \frac{x_j -c_j}{D(x_{j+2}-c_{j+2}-x_{j+1} +c_{j+1})}\right \rfloor +  
\left \lfloor \frac{x_j-c_j}{D(x_{j+2} -c_{j+1} +x_{j+1}-c_{j+2})}\right \rfloor 
\mod 2  \,.\]

Consider now an element ${\bf q} = \lim_{k\to\infty}{q_1+q_2 + \dots +q_k}$ and 
the corresponding sequence of triangulations $\{\CalT(q_1+q_2 + \dots +q_k)\}$. 
These converge to $\CalT({\bf q})$, and with them also we get convergence of
edge shifting and color. 

It is also true that for any fixed $x= (x_1,x_2,x_3)_0$ in the plane
the $D$- values of the three triple components of
$y=(y_1,y_2,y_3) := x-(q_1+q_2 + \dots +q_k)$, as well their various pairwise sums and
differences, 
do not change once $k$ is high enough, since if the $2$-content of a number $n$ is $2^m$ then
so also is the $2$-content of $n + 2^p$ for any $p>m$. Thus 
$\PP(x -(q_1+q_2 + \dots +q_k))$ is constant once $k$ is large enough, and we can denote
this constant value by $\PP_{\bf q} (x):= \PP(x - {\bf q})$. This defines the parity function 
$\PP_{\bf q}$ for $\CalT({\bf q})$.
Although $\{\CalT(q_1+q_2 + \dots +q_k)\}$ is a  {\bf CHT} triangulation, its limit $\CalT({\bf q})$
need not be. In fact we know that the translation orbit of any of the  {\bf CHT} tilings centered
at $(0,0)$ is dense in the minimal hull, and so we can compute a parity function 
$\PP_{\bf q}$ 
of any tiling of the minimal hull in this way. In the case of generic ${\bf q}$ this results
in a complete description of the parity of the tiling. In the case that there is convergence
of either a-lines or w-lines (so one is not in a generic case) one can still start with one of the extension functions $\PP^e$ and
arrive at a complete parity description $\PP^e_{\bf q}$ of any of the possible tilings associated
to $\CalT({\bf q})$.

\begin{coro} The parity function for a generic tiling is $\PP_{\bf q}$. 
\end{coro}

\section{The hull of parity tilings} \label{parityHull}
A Taylor--Socolar tiling is a hexagonal tiling with two tiles (if we allow rotations). With the appropriate
markings (not the ones we use in this paper), the two tiles can be considered as reflections of each other.
If we just consider the tiling as a tiling by two types of hexagons, white and gray, then we get the striking
parity tilings, for example,  of Fig.~ \ref{parityPattern}. We may consider the hull $Y_Q$ created by these parity tilings.
Evidently $Y_Q$ is a factor of $X_Q$. In this section we show that in fact the factor mapping is 
one-to-one -- in other words, when we discard all the information of the marked
tiles except the colors white and gray -- no information is lost, we can recover
the fully marked tiles if we know the full parity tiling. 
The argument uses a tool that is central to the
original work of Taylor, but has only played an implicit role in our argument: the Taylor--Socolar tilings have an underlying 
scaling inflation by a scale factor of $2$. One form of this scaling symmetry is especially obvious from the point of view
of the $Q$-adic triangularization.

\subsection{Scaling} \label{scaling}\
 Suppose that we have a Taylor--Socolar tiling with $\CalT({\bf q})$, where ${\bf q}= (q_1, q_1+q_2, $ $\dots, q_1 + \cdots+ q_k, \dots)$.
Then $q_1 + 2Q$ is the set of triangle vertices of all triangles of level at least $1$. To make things quite specific, which
we need to do to go on, we choose $q_1 \in \{0, a_1,a_2, a_1+a_2\}$. In the same way we shall assume
$q_2 \in 2\{0, a_1,a_2, a_1+a_2\}$, and so on. We can view
$q_1 + 2Q$ as being a new lattice (even though it may not be centered at $0$) and then we note that 
${\bf q}' := (0, q_2, q_2+q_3, \dots, q_2 + \cdots+ q_k, \dots)$ is another triangularization (now of this larger scale lattice, and taken
relative to an origin located at $q_1$) that
determines a Taylor--Socolar tiling with hexagonal tiles of twice the size.
Each of these new double-sized hexagons is centered on a hexagon of the original tiling which itself is centered at a vertex of a level $1$ triangle. 

The new triangularization has its own edge shifting, and since all that has happened is that all the lines
of the triangularization that do not pass through a vertex of a level 1 triangle have gone, we can see that
the shifting rules mean that the remaining lines still shift exactly as they did before. 
Also the whole process of coloring the new double-size hexagons goes just as it did before. 
Fig.~\ref{rescaling1} shows
why the coloring of the new double-size hexagons is the same as the coloring of the original size hexagons
on which they are centered. 

What happens if the Taylor--Socolar tiling $\CalT({\bf q})$ is non-generic? A tiling is non-generic if and only if
it has an $a$-line or a $w$-line of infinite level. From the definitions, one sees that 
just rescaling so that the level $0$ triangles vanish and all other triangles are now 
lowered in level by $1$ still leaves infinite level lines in tact and so we are still in a non-generic case. We can
also see this in detail. Non-generic tilings happen if and only if
${\bf q} \in x + \overline{\ZZ_2} \,a$ or ${\bf q} \in x + \overline{\ZZ_2} \,3w$ for some $x\in Q$.
Consider ${\bf q}' = (0, q_2, q_2+q_3, \dots, q_2 + \cdots+ q_k, \dots) \in 2 \overline Q$. For definiteness, take the second case. 
Then we can write 
\[{\bf q'} \in -q_1 + x + \alpha 3w + \overline{\ZZ_2}  3(2w) \]
where $\alpha =0$ if $-q_1 + x \in 2Q$ and $\alpha =1$ if $-q_1 + x \notin 2Q$. This gives
${\bf q'} \in  x' +  \overline{\ZZ_2}  3(2w)$ for some $x' \in 2Q$, and we are in the non-generic case again.
The same thing happens in the other case. 

\begin{figure}
\centering
\includegraphics[width=8cm]{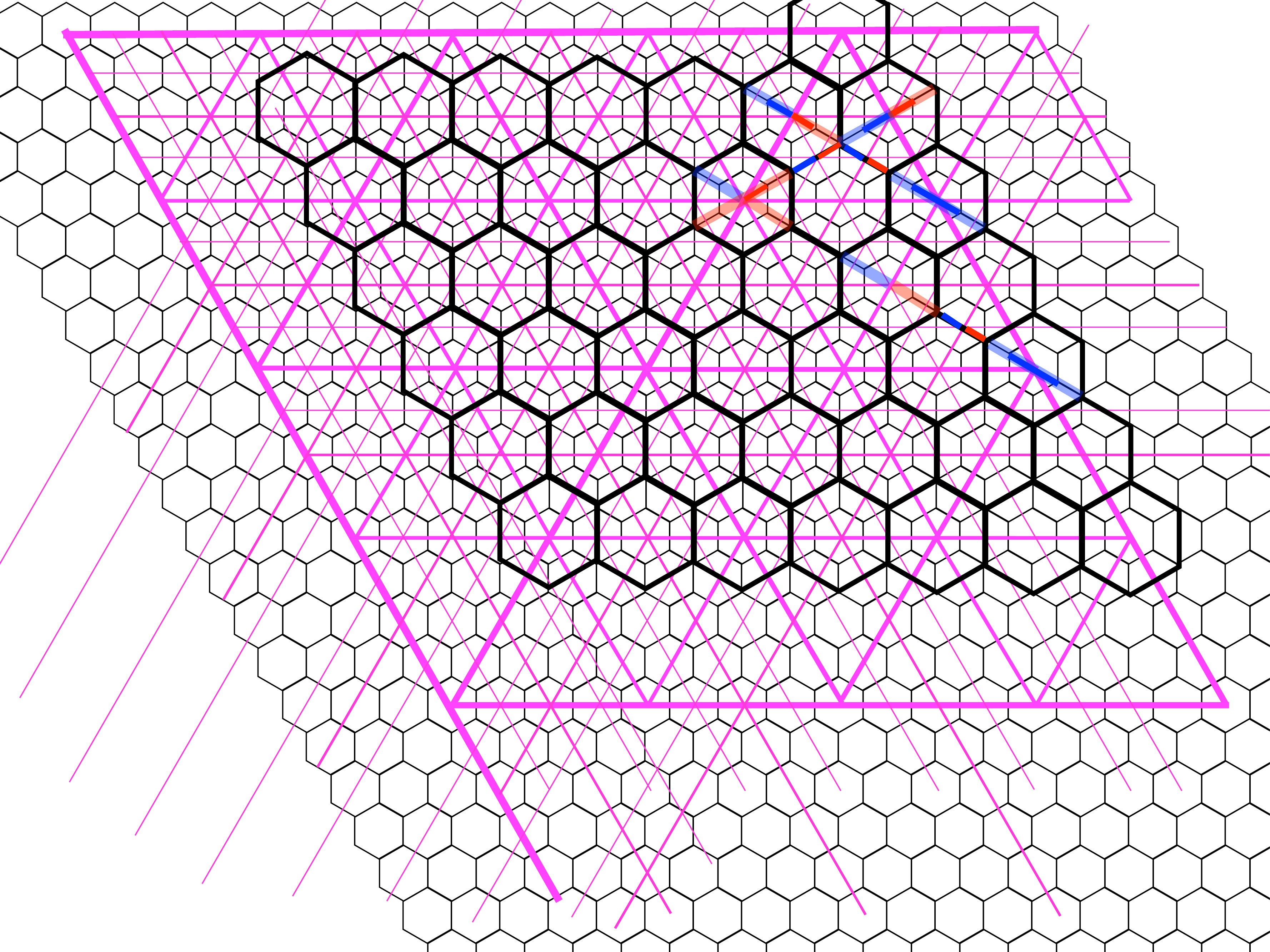}
\caption{This shows how the scaled up hexagon inherit the colors 
of the smaller hexagons on which they are centered. The coloring is carried
out by the same procedure that was used to color the original sized hexagons,
though now the diameter sizes have doubled and only the lines of level at least
$1$ are used.}
\label{rescaling1}
\end{figure}

This scaling self-similarity of the Taylor--Socolar tilings is an important part of Taylor's original construction of the tilings. Here we see self-similarity arise from the simple procedure of `left shifting' the $Q$-adic number ${\bf q}$. Geometrically, shifting means subtracting $q_1$ and dividing by $2$.
We have restricted $q_1, q_2, \dots$ to be specific coset representatives so that this process of subtraction/division is uniquely prescribed. 

We can also work this process in the other direction. Given ${\bf q}$ there are four ways in which
to choose a coset from $(\frac{1}{2} Q)/Q $ and for each of these choices we get a new
element of the $Q$-adic completion of $\frac{1}{2} Q$. Thus reducing the scale by a factor of $2$ 
we obtain new triangularizations and hexagonal tilings.\footnote{In our original
analysis of the Taylor--Socolar tilings, we chose the triangle vertices to be in $Q$, whereas
higher level triangles have vertices in cosets of various $2^k Q$, not necessarily 
in $2^kQ$ itself. This initial choice of elements in $Q$ was convenient to keep all
the triangle vertices in $Q$ itself. In the halving process we are doing here, we
may choose any of the cosets of $(\frac{1}{2} Q)/Q$. The canonical choice might
be to choose $0$, but this is not necessary to get a tiling.} The original tiling reappears via the scaling
up by a factor of $2$, as we have just seen. Also note that the new tilings are non-generic if and only if
the original tiling was since infinite $a$ or $w$ lines remain as such in the new tilings. Thus we see
that scaling has no effect on the generic or non-generic nature of the tilings in question. It is also clear
that {\bf CHT} tilings, {\bf iCw-L} tilings, and even {\bf SiCw-L} tilings
all transform into tilings of the same type.

\subsection{From parity to Taylor--Socolar tilings}

Creating a full Taylor--Socolar tiling from a parity tiling is carried out in two steps: first work out
the triangularization, then add the color. For generic tilings the second step is superfluous. 
Even for non-generic tilings, the fact that the tiling is non-generic is an observable property of 
the triangularization and then knowing the parity makes it straightforward to recover the colouring
along the ambiguous line or lines. The color is not a significant issue and we discuss only the recovery of the triangularization here. 

So let us assume that we have a parity tiling $\CalT_P$ that arises from some 
tiling $\CalT({\bf q})$ which at this point we do not know. To obtain the triangularization, i.e. ${\bf q}$,
 we need to work out the translated lattices $q_1 + 2Q, q_2 + 4Q, \dots$. These are the sets of vertices
 of the triangles of levels $\ge1$, $\ge 2$, etc. Suppose that we have a method that
 can recognize the translated lattice $q_1 +2Q$ (which is a subset of the hexagon centers of the parity tiling). 
Then, in effect, we know $q_1$ (at least modulo $2Q$, which is all we need to know about it). We can now imagine changing our view point to double the scale by redrawing the parity tiling with double
 sized hexagonal tiles at the points of  $q_1 +2Q$ and while retaining the color. We know that this new tiling will be the parity tiling of the scaled up tiling $q_1+ \CalT({\bf q'})$ where ${\bf q}' = (0, q_2, q_2+q_3, \dots, q_2 + \cdots+ q_k, \dots)$.  
 
 At this point we can proceed by induction to determine $q_2 + 4Q$, rescale again and get $q_3 + 8Q$, and so on. So what is needed is only to determine the vertices of the level $ \ge 1$ triangles, or equivalently determine
 which hexagons of the parity tiling lie on such vertices.

\medskip

Let us call a patch of $7$ tiles which consists of a tile in the center and its $6$ surrounding tiles a {\em basic patch of $7$ tiles}. The key observation is that in a Taylor--Socolar-tiling it never happens that a basic patch of $7$ tiles whose center tile is a corner tile of a triangle of level $\ge 1$ has $5$ surrounding tiles all of the same color. Thus in the parity tiling
the hexagons centered on the  (as yet unknown) vertices of the level $\ge 1$ triangles, cannot appear as shown in Fig.~\ref{5-hexagonalTile-patch} or in any of their rotated forms.  On the other hand, as we shall see, for all the other cosets of $Q$ modulo
$2Q$ there are points around which such patches (up to rotational symmetry) do occur. Furthermore we do not have to look far in any part of the tiling to find such examples. This is the feature that allows us to distinguish the coset $q_1 + 2Q$ from
the remaining cosets of $Q \mod 2Q$.

The actual proof goes in three steps. In the first we show the `no five hexagons' rule for hexagons centered on the vertices of triangles of level $1$. Next we do the same for all the vertices of triangles of level $\ge 2$. This deals with all hexagons centered on $q_1 + 2Q$. Finally we show that for each of the other cosets of $Q \mod 2Q$ the 
`no five hexagons' rule fails at least somewhere. 

\begin{figure}
\centering
\includegraphics[width=7.2cm]{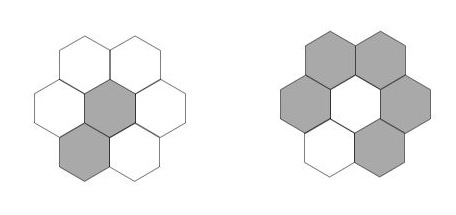}
\caption{Patches of (parity) tiles in which the central tile has $5$ of its surrounding tiles of the same color.}
\label{5-hexagonalTile-patch}
\end{figure}

\medskip

\noindent
 \begin{enumerate}
\item {\bf The vertices of $1$-level triangles}:   
We wish to show that  around each hexagon centered on a point of
$q_1 +2Q$ there are at least two different pairs of tiles of with mis-matched colors
amongst its six surrounding tiles (and hence the no five-hexagons rule is true). This is explained in the text below and in 
Fig.~\ref{5same-color-tile-argument-III}. 
\begin{figure}
\centering
\includegraphics[width=11cm]{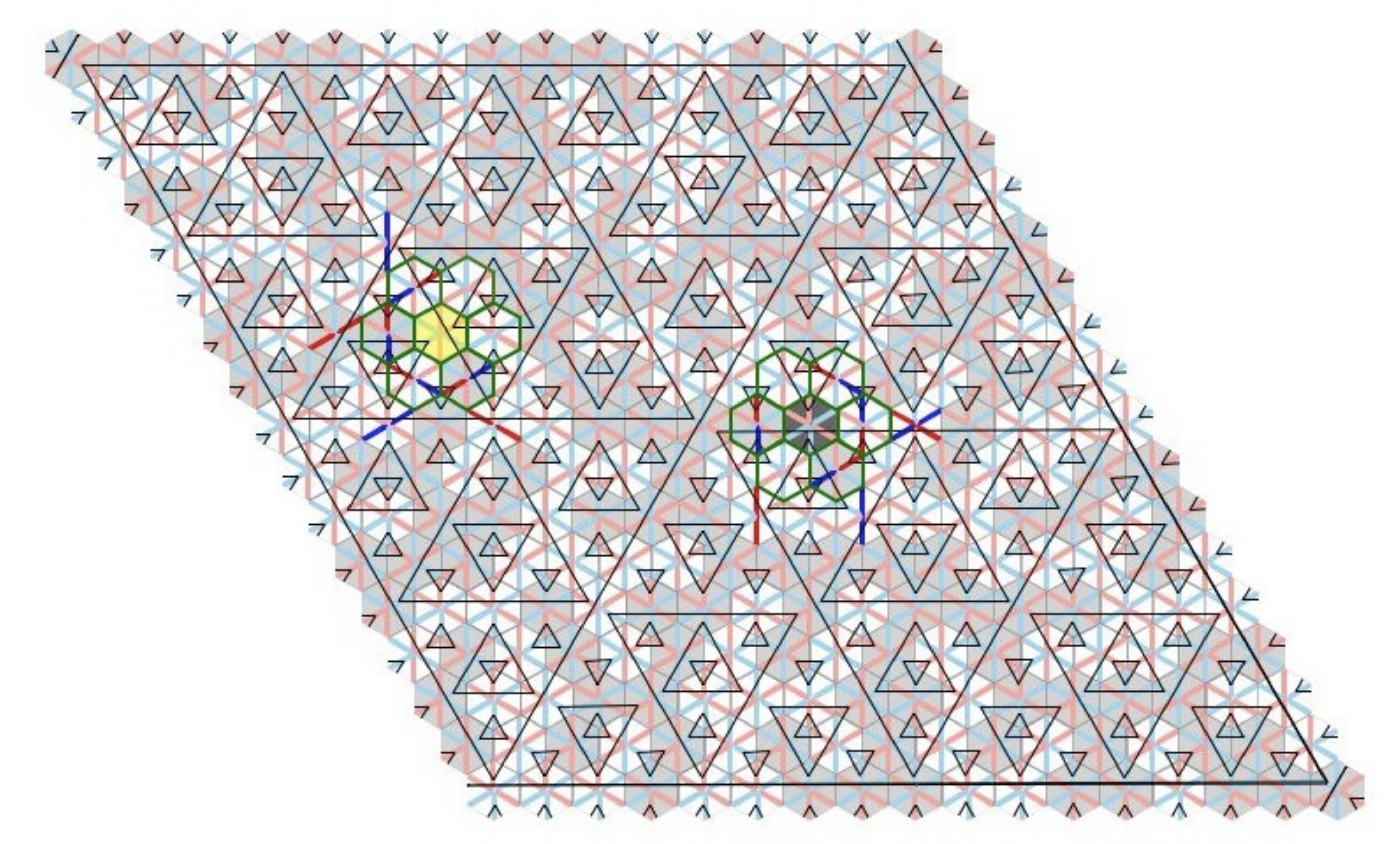}
\caption{Examples are shown which demonstrate how we can see that 
around each hexagon centered on a level $1$ vertex of
$q_1 +2Q$ there are at least two different pairs of tiles with mis-matched colors
amongst its six surrounding tiles. The two cases correspond to a level $1$ vertex at the mid-point
of a level $2$ triangle and a level $1$ vertex a non-midpoint of a higher level triangle. As explained in the text, one pair is found in a uniform way in both cases. The other pair is found by one method in the first
case and another method in the other.}
\label{5same-color-tile-argument-III}
\end{figure}

There are two different situations for a hexagon $H$ centered on the vertex $v$ of a $1$-level triangle. One is the case that $v$ is at the midpoint of a $2$-level triangle. This is indicated on the left side of Fig. \ref{5same-color-tile-argument-III}. The other is the case that $v$ is on the edge of a $n$-level triangle 
where $n \ge 3 $. This is indicated on the right side of Fig. \ref{5same-color-tile-argument-III}. 
In both cases, we note that the two tiles on opposite sides of $H$ which share the long edge of a triangle of level 
$\ge 2$ have different colors. The reason is the following. Apart from the red-blue diameters, the long red and blue diameters and black stripes are same for the both the Taylor--Socolar tile and the reflected Taylor--Socolar tile. However the red and blue diameters of the middle tiles of $2$-level triangles determine different red-blue diameters for the two tiles, and so they have different parity. 

Now we wish to find another pair of tiles with opposite colors for each of the cases.
Let us look at the first case (see the left side of Fig. \ref{5same-color-tile-argument-III}).  Consider
the corner of the level $3$ triangle $T$ that is defined by $H$, and consider the two edges
of $T$ that bound this corner.  The red and blue diameters of the tiles at the mid-points of these two edges of $T$
determine different red-blue diameters for two neighbouring tiles in the surrounding tiles of $H$. This again
results in different parities.

Finally consider the second case (see the right side of Fig.~\ref{5same-color-tile-argument-III}). 
Notice that the long black stripe of $H$ is the part of the long edge of $3$ or higher level triangle.  If this long
edge is from a level $4$ or higher triangle, we consider just the part of it that is the level $3$ triangle $T$ whose edge coincides with the stripe of $H$. The red and blue diameters of the tile centered on the mid-point of this edge of $T$ determine different red-blue diameters for two of the ring of tiles around $H$, as shown.   
\smallskip
\item {\bf The vertices of level $\ge 2$ triangles}:
Next we look at the six tiles surrounding the corner tiles of level $\ge 2$ triangles in a Taylor--Socolar tiling. Notice that the pattern of the colored diameters of six surrounding tiles is same for every corner tile of a level $2$ or higher
 triangle, Fig.~\ref{basicPatch-2hlevelTriangle}. 
\begin{figure}
\centering
\includegraphics[width=11cm]{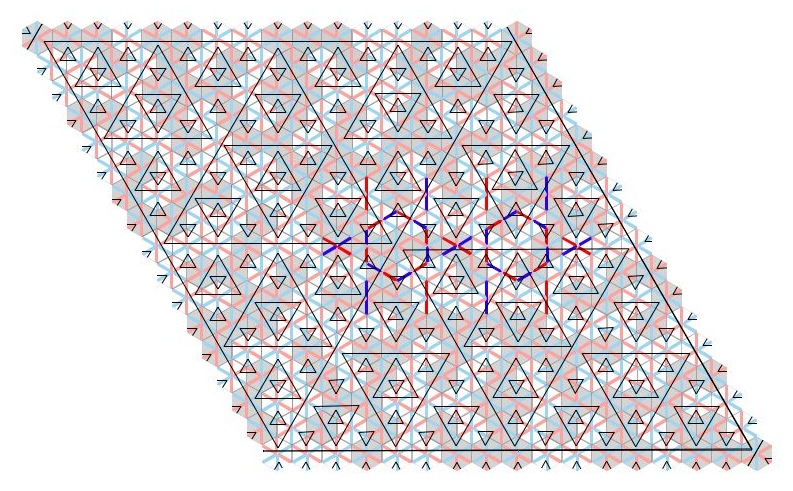}
\caption{}
\label{basicPatch-2hlevelTriangle}
\end{figure}
So what determines the basic patches of $7$ tiles around the corner tiles of level $2$ or higher triangles is the 
pattern of black stripes on it. Furthermore, there is a one-to-one correspondence between these basic patches of $7$ tiles and 
the basic patches of $7$ tiles of white and gray colors. This is shown in Fig. \ref{Parity-center-case}.  
The key point is that this means that the basic patches of $7$ tiles around the vertices of the level $2$
or higher triangles already determine the coloring.
\begin{figure}
\centering
\includegraphics[width=12cm]{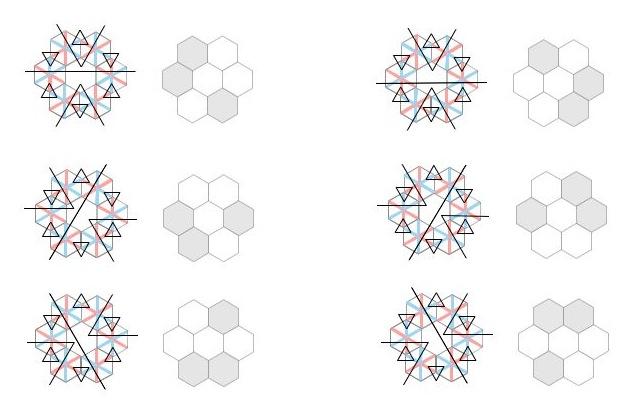}
\caption{The edge and color patterns and corresponding parity patterns that can occur in the hexagons surrounding
the corners of triangles of level $\ge 2$. }
\label{Parity-center-case}
\end{figure}

At this point we know that all the vertices of level $\ge 1$ have at least two tiles of each color
in the ring of any basic patch of 7 tiles.
\smallskip
\item {\bf Seeing how the other cosets of $Q$ mod $2Q$ violate the `two tiles of each color
in the ring of any basic patch of $7$ tiles' rule.}
From a given parity tiling, there are four choices in determining the $1$-level triangles. One of these
is the coset $q_1 +2Q$, and we know that the $7$ tile patches around each of these points
satisfy the `no five-hexagon' rule.  However the three choices, corresponding to the 
other three cosets of $Q \mod 2Q$ all have some $7$ tile patches that violate the rule. We can see violations
to the rule for each of the other three cosets in Fig.~\ref{5colorRuleViolated},
which is a small piece from the lower right corner of Fig.~\ref{parityPattern}. Since any parity tilings in the hull are repetitive, we observe the patches frequently over the parity tiling. Furthermore, since there is only one way which is allowed to determine $1$-level triangles, it does not depend on where one starts to find the $1$-level triangles. They will all match in the end.   

\begin{figure}
\centering
\includegraphics[width=9cm]{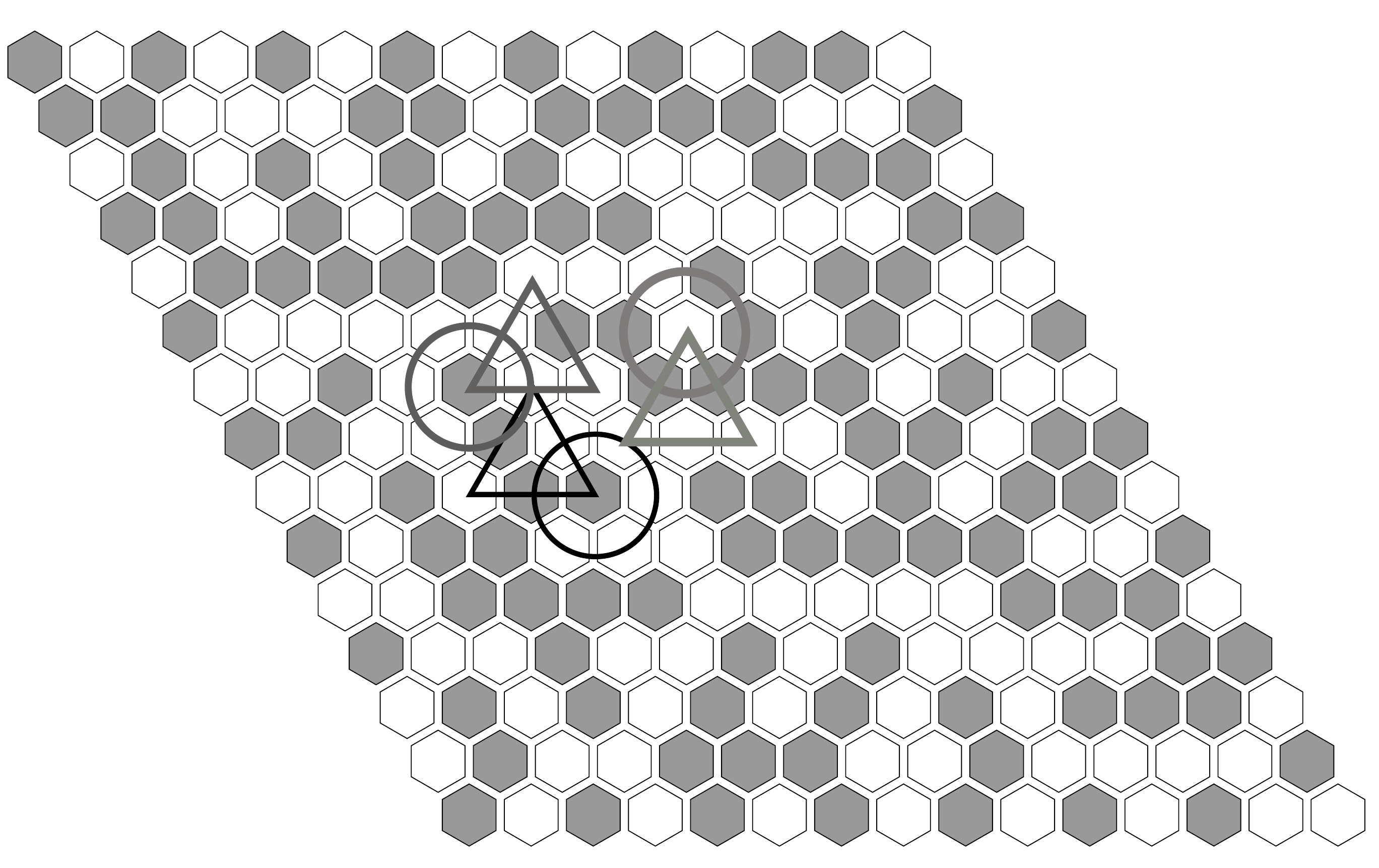}
\caption{Here we see three violations to the five-hexagons rule, the rings of hexagons being
indicated by the circles. The shaded triangles show how the cosets determined by the center look, and
show that the violating hexagons are from three different cosets. The hexagonal pattern comes
from the lower right corner of Fig.~\ref{parityPattern}.}
\label{5colorRuleViolated}
\end{figure}

\end{enumerate}

This finishes the discussion of how the coset $q_1 + 2Q$ is identified in the parity tiling $\CalT_P$.
The scaling argument shows that we can continue the process to identify $q_2 + 4Q, q_3 + 8Q$ and so on. This finally identifies $\bf q$. In fact, even stronger, once we identify $q_2 + 4Q$, we can use the one-to-one correspondence between these basic patches of $7$ tiles and the basic patches of $7$ tiles of white and gray colors as it is shown in Fig. \ref{Parity-center-case} to determine the entire triangularization.  

We note two things here: that this argument of this subsection has not required that the tiling be generic,
and that the process of reconstruction of the triangularization from a parity tiling is local in character.

We conclude that the parity hull loses no color information, and hence the factor mapping from
the Taylor--Socolar hull to the parity hull that simply forgets all information about each tile
except for its parity, is in fact an isomorphism. In fact the two hulls are mutually locally derivable
(MLD) in the sense that each tiling in one is derivable from a corresponding tiling in the other,
using only local information around each tile (see \cite{HF}).

\begin{coro} \label{MLD}
Each Taylor--Socolar tiling and its corresponding parity tiling are mutually locally derivable. \qed
\end{coro}

\section{Concluding remarks} \label{cr}
The paper has developed an algebraic setting for the Taylor--Socolar hexagonal tilings. 
This approach leads naturally to a cut and project scheme with a compact $Q$-adic internal space $\overline{Q}$.
We have determined the structure of the tiling hull and in particular the way in which it
lies over compact group $Q$. Each tiling is a model set from this cut and project formalism.
The corresponding parity tilings are in fact completely equivalent to the fully decorated
tilings. We have also reproved the parity formula of \cite{ST}.

The tiling is both remarkably simple and remarkably subtle. 
The parity tilings remain fascinating and inviting of further study. 

\section{Acknowledgment}

\noindent
We are grateful to Michael Baake and Franz G\"ahler for their interest and helpful
suggestions in the writing of this paper, and to Joan Taylor who provided the 
idea behind \S\ref{parityHull}.

\end{document}